\documentclass[reqno]{amsart} 
 
\usepackage{hyperref}
\usepackage{amsmath}
\usepackage{amssymb}
\usepackage{amsfonts}
\usepackage{amsthm}
\usepackage{latexsym}
\usepackage{esint}
\usepackage{mathtools}
\usepackage{mathrsfs}
\usepackage{graphicx}
\usepackage{wrapfig}
\usepackage{bbm}
\usepackage{color}
\usepackage{bm}
\usepackage[top=1in, bottom=1.25in, left=1.10in, right=1.10in]{geometry}
\usepackage{todonotes}
\usepackage{stackengine}
\allowdisplaybreaks

\newtheorem{theorem}{Theorem}[section]

\newtheorem{proposition}[theorem]{Proposition}

\theoremstyle{definition}
\newtheorem{definition}[theorem]{Definition}

\theoremstyle{remark}
\newtheorem{remark}[theorem]{Remark}

\numberwithin{equation}{section}

\newcommand{\bx}{\mathbf{x}}
\newcommand{\by}{{y}}

\newcommand{\bq}{\mathbf{q}}
\newcommand{\bu}{\mathbf{u}}
\newcommand{\bfu}{\mathbf{u}}

\newcommand{\bff}{\mathbf{f}}

\newcommand{\bv}{\mathbf{v}}
\newcommand{\bn}{\mathbf{n}}

\newcommand{\dy}{\, \mathrm{d}y}
\newcommand{\dd}{\,\mathrm{d}}
\newcommand{\dq}{\, \mathrm{d} \mathbf{q}}
\newcommand{\dx}{\, \mathrm{d} \mathbf{x}}

\newcommand{\dt}{\, \mathrm{d}t}
\newcommand{\ds}{\, \mathrm{d}s}

\newcommand{\Div}{\mathrm{div}_{\mathbf{x}}}
\newcommand{\divx}{\mathrm{div}_{\mathbf{x}}}
\newcommand{\divq}{\mathrm{div}_{\mathbf{q}}}

\newcommand{\delx}{\Delta_{\mathbf{x}}}

\newcommand{\nabx}{\nabla_{\mathbf{x}}}
\newcommand{\naby}{\partial_y}
\newcommand{\nabq}{\nabla_{\mathbf{q}}}

\newcommand{\Delx}{\Delta_{\mathbf{x}}}
\newcommand{\Dely}{\Delta_{\mathbf{y}}}

\newcommand{\bT}{\mathbb{T}}
\newcommand{\R}{\mathbb{R}}

\newcommand{\Oeta}{\Omega_{\eta}}
\newcommand{\Ozeta}{\Omega_{\zeta}}

\begin{document}

\title[Polymeric fluid-structure interaction of {O}ldroyd-{B} type]{Weak and strong solutions for polymeric fluid-structure interaction of {O}ldroyd-{B} type}

%
%
\author{Prince Romeo Mensah}
\address{Institut f\"ur Mathematik,
Technische Universit\"at Clausthal,
Erzstra{\ss}e 1,
38678 Clausthal-Zellerfeld}
\email{prince.romeo.mensah@tu-clausthal.de \\ orcid ID: 0000-0003-4086-2708}
\thanks{The author would like to thank Dominic Breit for many useful discussions.}


\date{\today}



\begin{abstract}
We prove the existence of weak solutions and a unique strong solution to the Oldroyd-B dumbbell model describing the evolution of a two-dimensional dilute polymer fluid interacting with a one-dimensional viscoelastic shell. The polymer fluid consists of a mixture of an incompressible viscous solvent and a solute comprising two massless beads connected by a Hookean spring with center-of-mass diffusion. This solute-solvent mixture then interacts with a flexible structure that evolves in time. An arbitrary nondegenerate reference domain for the polymer fluid is allowed and both solutions exist globally in time provided no future degeneracies occur with the structure deformation. Futhermore, weak-strong uniqueness holds unconditionally.
\end{abstract}

\maketitle


\section{Introduction}
A polymer fluid is a complex fluid with a high molecular weight consisting of a mixture of a solvent and a solute. In this work, the solvent is a  viscous fluid modeled by the incompressible Navier--Stokes equation and the solute is described by the macroscopic average of the probability distribution function of two monomers connected by a Hookean spring and modeled by the  Oldroyd-B system (also referred to as the \textit{convected Jeffreys model} by some authors \cite{bird1987dynamics}). We refer to \cite{barrett2017existenceOldroyd} for the modeling of the generalized state-of-art where the fluid's density is also allowed to vary.

We initiate work on the rigorous analysis of such a polymer fluid within a flexible structure and the interaction between the evolution of the structure and the polymer fluid. The polymer fluid is two-dimensional and contained in a domain whose boundary is a flexible structure in 1-D modeled by a viscoelastic shell equation. We deviate from the usual practice in the literature where one considers a simple flat reference domain for fluids and allows for a generalized reference domain for the solute-solvent mixture.

\subsection{Geometric setup and equations of motion}
\label{sec:geo}
We consider a fluid domain whose reference configuration is  $\Omega \subset \mathbb{R}^2$. The boundary of this reference domain $\partial\Omega$ may consist of a flexible part $\omega \subset \mathbb{R}$ and a rigid part $\Gamma \subset \mathbb{R}$. However, because the analysis at the rigid part is significantly simpler, we shall identify the whole of $\partial \Omega$ with $\omega$. Let $I:=(0,T)$ represent a time interval for a given constant $T>0$. 
We represent the time-dependent  displacement of the structure by $\eta:\overline{I}\times\omega\rightarrow(-L,L)$ where $L>0$ is a fixed length of the tubular neighbourhood of $\partial\Omega$ given by
\begin{align*}
S_L:=\{\bx\in \mathbb{R}^2\,:\, \mathrm{dist}(\bx,\partial\Omega
)<L \}.
\end{align*}
Now, for some $k\in\mathbb{N}$ large enough, we assume that $\partial\Omega$  is parametrized by an injective mapping $\bm{\varphi}\in C^k(\omega;\mathbb{R}^2)$ with $\naby \bm{\varphi}\neq0$ such that
\begin{align*}
\partial{\Omega_{\eta(t)}}=\big\{\bm{\varphi}_{\eta(t)}:=\bm{\varphi}(\by)+\bn(\by)\eta(t,\by)\, :\, t\in I, \by\in \omega\big\}.
\end{align*}
The set $\partial{\Omega_{\eta(t)}}$ represents the boundary of the flexible domain at any instant of time $t\in I$ and the vector $\bn(y)$ is a unit normal at the point $y\in \omega$. 
We also let $\bn_{\eta(t)}(y)$ be the corresponding normal of $\partial{\Omega_{\eta(t)}}$ at the spacetime point $y\in \omega$ and $t\in I$. Then for $L>0$ sufficiently small, $\bn_{\eta(t)}(y)$ is close to $\bn(y)$ and $\bm{\varphi}_{\eta(t)}$ is close to $\bm{\varphi}$. Since $\naby \bm{\varphi}\neq0$,  it will follow that
\begin{align*}
\naby \bm{\varphi}_{\eta(t)} \neq0 \quad\text{ and }\quad \bn(y)\cdot \bn_{\eta(t)}(y)\neq 0 
\end{align*}
for $y\in \omega$ and $t\in I$. Thus, in particular, there is no loss of strict positivity of the Jacobian determinant provided that $\Vert \eta\Vert_{L^\infty(I\times\omega)}<L$.

For the interior points, we  transform the  reference domain $\Omega$ into a time-dependent moving domain $\Omega_{\eta(t)}$  whose state at time $t\in\overline{I}$ is given by
\begin{align*}
\Omega_{\eta(t)}
 =\big\{
 \bm{\Psi}_{\eta(t)}(\bx):\, \bx \in \Omega 
  \big\}.
\end{align*}
Here,
\begin{align*}
\bm{\Psi}_{\eta(t)}(\bx)=
\begin{cases}
\bx+\bn(\by(\bx))\eta(t,\by(\bx))\phi(s(\bx))     & \quad \text{if } \mathrm{dist}(\bx,\partial\Omega)<L,\\
    \bx & \quad \text{elsewhere. } 
  \end{cases}
\end{align*}
is the Hanzawa transform with inverse $\bm{\Psi}_{-\eta(t)}$ and where for a point $\bx$ in the neighbourhood of $\partial\Omega$, the vector $\bn(y(\bx))$ is the unit normal at the point $y(\bx)=\mathrm{arg min}_{y\in\omega}\vert\bx -\bm{\varphi}(y)\vert$. Also, $s(\bx)=(\bx-\bm{\varphi}(y(\bx)))\cdot\bn(y(\bx))$ and $\phi\in C^\infty(\mathbb{R})$ is a cut-off function that is $\phi\equiv0$ in the neighbourhood of $-L$ and $\phi\equiv1$ in the neighbourhood of $0$. Note that $\bm{\Psi}_{\eta(t)}(\bx)$ can be rewritten as
\begin{align*}
\bm{\Psi}_{\eta(t)}(\bx)=
\begin{cases}
\bm{\varphi}(y(\bx))+\bn(\by(\bx))[s(\bx)+\eta(t,\by(\bx))\phi(s(\bx)) ]    & \quad \text{if } \mathrm{dist}(\bx,\partial\Omega)<L,\\
    \bx & \quad \text{elsewhere. } 
  \end{cases}
\end{align*}
\\
With the above preparation in hand, we consider the Oldroyd-B model for the flow of a dilute polymeric fluid interacting with a flexible structure in  the closure of the deformed spacetime cylinder
\begin{align*}
I\times\Oeta:=\bigcup_{t\in I}\{t\}\times\Oeta
\end{align*}
with $\Oeta:={\Omega_{\eta(t)}}$.
Our goal is to find a structure displacement function $\eta:(t, \by)\in I \times \omega \mapsto   \eta(t,\by)\in \mathbb{R}$, a fluid velocity field $\mathbf{u}:(t, \mathbf{x})\in I \times \Oeta \mapsto  \mathbf{u}(t, \mathbf{x}) \in \mathbb{R}^2$, a pressure function $p:(t, \mathbf{x})\in I \times \Oeta \mapsto  p(t, \mathbf{x}) \in \mathbb{R}$, a polymer number density $\rho :(t, \mathbf{x} )\in I \times \Oeta  \mapsto \rho(t, \mathbf{x} ) \in \mathbb{R}$
and an extra stress tensor $\bT :(t, \mathbf{x} )\in I \times \Oeta  \mapsto \bT (t, \mathbf{x} ) \in \mathbb{R}^{2\times2}$
 such that the system of equations 
\begin{align}
\divx \bu=0, \label{divfree}
\\
\partial_t \rho+ (\mathbf{u}\cdot \nabx) \rho
= 
\varepsilon\Delx \rho 
,\label{rhoEqu}
\\
\partial_t \bu  + (\mathbf{u}\cdot \nabx)\mathbf{u} 
= 
\nu \delx \bu -\nabx p
+\bff
+K
\divx   \bT, \label{momEq}
\\
\varrho_s \partial_t^2 \eta - \gamma\partial_t\partial_y^2 \eta + \alpha \partial_y^4 \eta = g - ( \mathbb{S}\bn_\eta )\circ \bm{\varphi}_\eta\cdot\bn \,\det(\partial_y\bm{\varphi}_\eta), 
\label{shellEQ}
\\
\partial_t \bT + (\mathbf{u}\cdot \nabx) \bT
=
(\nabx \bu)\bT + \bT(\nabx \bu)^\top - 2k(\bT - \rho \mathbb{I})+\varepsilon\Delx \bT \label{solute}
\end{align}
holds on $I\times\Oeta\subset \mathbb R^{1+2}$ where
\begin{align*}
\mathbb{S}= \nu(\nabx \bu +(\nabx \bu)^\top) -p\mathbb{I}+ K\bT,
\end{align*}
the parameters $\varepsilon,K,\gamma,k,\nu,\varrho_s,\alpha$ are all positive constants, $\bn_\eta$ is the normal at $\partial\Oeta$ and $\mathbb{I}$ is the identity matrix.
We complement \eqref{divfree}--\eqref{solute} with the following initial and boundary conditions
\begin{align}
&\eta(0,\cdot)=\eta_0(\cdot), \qquad\partial_t\eta(0,\cdot)=\eta_\star(\cdot) & \text{in }\omega,
\\
&\bu(0,\cdot)=\bu_0(\cdot) & \text{in }\Omega_{\eta_0},
\\
&\rho(0,\cdot)=\rho_0(\cdot),\quad\bT(0,\cdot)=\bT_0(\cdot) &\text{in }\Omega_{\eta_0},
\label{initialCondSolv}
\\
& 
\bn_{\eta}\cdot\nabx\rho=0,\qquad
\bn_{\eta}\cdot\nabx\bT=0 &\text{on }I\times\partial\Omega_{\eta}.
\label{bddCondSolv}
\end{align}
Furthermore, for simplicity, we impose periodicity on the boundary of $\omega$ and the following interface condition
\begin{align} 
\label{interface}
&\bu\circ\bm{\varphi}_\eta=(\partial_t\eta)\bn & \text{on }I\times \omega
\end{align}
between the polymeric fluid and the flexible part of the boundary with normal $\bn$.
\\
The two unknowns $\rho$ and $\bT$ for the solute component of the polymer fluid are related via the identities
\begin{align*} 
\bT(t, \bx)= k\int_{B} f(t,\bx,\bq)\bq\otimes\bq \dq,
\qquad
\rho(t, \bx)= \int_{B} f(t,\bx,\bq) \dq
\end{align*}
where for  $B=\mathbb{R}^2$ with elements $\bq\in B$, the function $f$ is the probability density function ($f\geq0$ a.e. on $\overline{I}\times\Oeta\times B$) satisfying the
Fokker--Planck equation
\begin{align}
\label{fokkerPlanck}
\partial_t f+\divx(\bu f)+\divq((\nabx\bu)\bq f)=\varepsilon\Delx f+k\divq(M\nabq(f/M))
\end{align}
in $I\times\Oeta\times B$ for a Hookean dumbbell spring potential and Maxwellian
\begin{align*} 
U\Big(\frac{1}{2}|\bq|^2 \Big)=\frac{1}{2}|\bq|^2, \qquad\qquad M=\frac{\exp(-U(\tfrac{1}{2}|\bq|^2))}{\int_B\exp(-U(\tfrac{1}{2}|\bq|^2))\dq},
\end{align*}
respectively. Whether the center-of-mass diffusion parameter $\varepsilon$ in \eqref{fokkerPlanck} is zero or positive has been a source of many debates over the years. In this work, we follow the school of thought such as \cite{barrett2007existence, degond2009kinetic, schieber2006generalized}  that gives justifications
for the inclusion of this term. As shown in \cite{barrett2017existenceOldroyd}, this choice leads to the eligant parabolized macroscopic closure equation for the polymer number density $\rho$ and the elastic extra stress tensor $\bT$ satisfying \eqref{rhoEqu} and \eqref{solute}, respectively.

The existence of smooth solutions for the coupling of the Fokker-Planck equation \eqref{fokkerPlanck} with a generalized viscous compressible fluid in a fixed domain has been shown in \cite{breit2021local} whereas weak solutions have earlier been constructed in  \cite{barrett2016existenceA, barrett2016existence, feireisl2016dissipative}. More work has been done in the incompressible case in fixed domains. Weak solutions are constructed in \cite{barrett2005existence, barrett2007existence, barrett2008existence, barrettSuli2011existence, barrett2012existenceMMMAS, barrett2012existenceJDE} and in \cite{masmoudi2013global}  when $\varepsilon=0$ in \eqref{fokkerPlanck} whereas results on the existence of strong solutions are shown in \cite{constantin2005nonlinear, constantin2007regularity, jourdain2004existence, kreml2010local, li2004local,  luo2017global, masmoudi2008well, renardy1991existence, zhang2006local}

There are several works on the rigorous analysis of Oldroyd-B-type models in a fixed spatial domain or for domains without boundaries.
Strong solutions have been constructed in 
\cite{guillope1990existence} where the authors also show the existence of stable
time-periodic solutions when the forcing term is small and time-periodic.
For small data, the existence of global-in-time classical solutions has been shown in \cite{lei2005global} using an incompressible limit argument. Small data global solutions living in H\"older spaces and Besov spaces are also shown in \cite{constantin2012remarks} and \cite{zi2014global, fang2016global, zhai2021global}, respectively. See also \cite{zhu2018global} for the three-dimensional case.
The authors in \cite{constantin2012remarks} also cover large data for a model in which the potential responds to high rates of strain in the fluid. When stress diffusion is taken into account, the authors in \cite{constantin2012note} study the global existence and regularity of strong solutions. See also \cite{elgindi2015global, elgindi2015global1, wu2022global} which covers the case where the solvent is inviscid and \cite{chen2023global} where there is no damping or dissipation in the equation for the extra stress tensor. Results on weak solutions are unusually fewer and we refer to \cite{barrett2018existence, hu2016global, lions2000global} where global weak solutions are constructed and to \cite{bathory2021large} which combines the Oldroyd-B and the Giesekus models.  
Related results also include the Peterlin model \cite{luka2017global} and the compressible Oldroyd-B model without stress diffusion \cite{lu2020global}.

There are only a couple of results on polymer fluids evolving and interacting with a flexible structure. The existence of weak solutions to a finitely extensible dilute polymer \eqref{fokkerPlanck} of Warner-type immersed in a viscous incompressible fluid \eqref{divfree}, \eqref{momEq} and with this $3$-D mixture interacting with a $2$-D elastic structure \eqref{shellEQ} ($\gamma=0$) has been shown in \cite{breit2021incompressible}. These weak solutions are global provided the shell does not self-intersect and there are no degeneracies in the system. This was recently followed up with the local-in-time existence of a unique strong solution \cite{breit2023existence} for a viscoelastic structure with  ($\gamma>0$). The mathematical analysis of an Oldroyd-B model interacting with a flexible structure is completely open. Related works, however, include \cite{gotze2013strong, sabbagh2022motion} where the authors study strong solutions of the immersion of solid objects in viscoelastic fluids.

\subsection{Work plan}
The paper is structured in the following way: In Section \ref{sec:prelim}, we introduce some notations, give the precise definition of the types of solutions we are investigating, and state our main results. To make the construction of solutions as concise as possible, we present in Section \ref{sec:Apriori}, the a priori estimates leading to global bounds. These bounds will be referred to several times in later sections. We then devote the entirety of Section \ref{sec:weakSol} to the proof of the existence of weak solutions. This will come in three main steps. Firstly, we construct a finite-dimensional approximation of the linearised solvent-structure subproblem and of the solute subproblem independently of each other.
Here, the linearised subproblem is advected by a given regularized vector field and evolves on a given regularised geometry. A fixed point argument is then used to obtain a finite-dimensional approximation of the fully coupled linearized system. After this, we pass to the limit in the discrete parameter to obtain a unique solution to the linear and regularised system. Our second step will be to remove the given advected vector field and the given spatial geometry to obtain a weak solution to the fully nonlinear regularised system posed on a regularised spatial geometry that is evolving according to the structure equation. Here, we use a Schauder-type fixed point argument whose key ingredient is an Aubin-Lion-type compactness result  \cite[Theorem 5.1]{muha2019existence} for establishing the compactness of the fluid's velocity fields. We lose uniqueness at this point due to the nonlinearity in the system and the fact that the regularity of weak solutions is low. The final step then involves the passage to the limit in the regularisation parameter to obtain a weak solution to the original system.

We devote Section \ref{sec:strongSol} to the construction of a unique strong solution via a fixed point argument.
Strong solutions are only an instant of differentiability stronger than weak solutions. This low regularity of the strong solutions coupled with the boundary-valued effect of the moving domain (e.g. boundary terms from applying the Reynolds transport theorem) makes it very difficult to obtain useful estimates. Indeed, the most challenging part of Section \ref{sec:strongSol} is the proof of contraction which requires obtaining estimates for the difference of two solutions.  
Firstly, each solution is defined on a unique spatial domain evolving in time and as such, it is not clear what the spatial domain for the difference between the two solutions is. Secondly, the strong solutions are not very regular and are actually only an instant on integrability stronger than weak solutions. Consequently, the standard H\"older doubles or triples for estimating the product of functions leads one to require more regularity for some of the terms in the product than they actually possess.

There are at least two ways to remedy the first problem. We can either transform each of the solutions to the fixed reference domain and study the difference equation on this fixed domain or we transform one solution to the domain of the other and study the difference equation on this latter domain. The second approach is more difficult than the first. However, we opt for the second approach since it allows us to take full advantage of the regularity properties of the latter geometry on which we study the difference equation. This approach has been used in \cite{guidoboni2012continuous} to show the continuous dependence on initial data in a fluid-structure problem and in \cite{schwarzacher2022weak} to obtain a weak-strong uniqueness result for a fluid interacting with elastic plates and in \cite{BMSS} for elastic shells.

The key to solving the second problem is to find equivalent (in the sense that two spaces are continuously embedded in each other) square-integrable Sobolev spaces of fractional differentiability for non-square-integrable Sobolev spaces in order to obtain sharp interpolation between the weakest and strongest spaces of the functions being estimated.

 Finally, we show in Section \ref{sec:weak-strong} that weak-strong uniqueness hold unconditionally, i.e. weak solutions are always unique in the class of strong solutions.

\section{Preliminaries and main results}
\label{sec:prelim}
\noindent  
Henceforth, without loss of generality, we set all the parameters $\{ \varepsilon,K,\gamma,k,\nu,\varrho_s,\alpha \}$ in \eqref{divfree}-\eqref{bddCondSolv} to 1.
For two non-negative quantities $F$ and $G$, we write $F \lesssim G$  if there is a constant $c>0$  such that $F \leq c\,G$. If $F \lesssim G$ and $G\lesssim F$ both hold, we use the notation $F\sim G$. The symbol $\vert \cdot \vert$ may be used in four different contexts. For a scalar function $f\in \mathbb{R}$, $\vert f\vert$ denotes the absolute value of $f$. For a vector $\bff\in \mathbb{R}^d$, $\vert \bff \vert$ denotes the Euclidean norm of $\bff$. For a square matrix $\mathbb{F}\in \mathbb{R}^{d\times d}$, $\vert \mathbb{F} \vert$ shall denote the Frobenius norm $\sqrt{\mathrm{trace}(\mathbb{F}^T\mathbb{F})}$. Finally, if $S\subseteq  \mathbb{R}^d$ is  a (sub)set, then $\vert S \vert$ is the $d$-dimensional Lebesgue measure of $S$.

For $I=(0,T)$, $T>0$, and $\eta\in C(\overline{I}\times\omega)$ with $\|\eta\|_{L^\infty(I\times\omega)}< L$, we define for $1\leq p,r\leq\infty$,
\begin{align*} 
L^p(I;L^r(\Omega_\eta))&:=\Big\{v\in L^1(I\times\Omega_\eta):\substack{v(t,\cdot)\in L^r(\Omega_{\eta(t)})\,\,\text{for a.e. }t,\\\|v(t,\cdot)\|_{L^r(\Omega_{\eta(t)})}\in L^p(I)}\Big\},\\
L^p(I;W^{1,r}(\Omega_\eta))&:=\big\{v\in L^p(I;L^r(\Omega_\eta)):\,\,\nabx v\in L^p(I;L^r(\Omega_\eta))\big\}.
\end{align*} 
Higher-order Sobolev spaces can be defined accordingly. For $k>0$ with $k\notin\mathbb N$, we define the fractional Sobolev space $L^p(I;W^{k,r}(\Oeta))$ as the class of $L^p(I;L^r(\Omega_\eta))$-functions $v$ for which 
\begin{align*}
\|v\|_{L^p(I;W^{k,r}(\Oeta))}^p
&=\int_I\bigg(\int_{\Oeta} \vert v\vert^r\dx
+\int_{\Oeta}\int_{\Oeta}\frac{|v(\bx)-v(\bx')|^r}{|\bx-\bx'|^{d+k r}}\dx\dx'\bigg)^{\frac{p}{r}}\dt
\end{align*}
is finite. Accordingly, we can also introduce fractional differentiability in time for the spaces on moving domains.
Next, for  $\eta\in C(\overline{I}\times\omega)$ satisfying
\begin{align*}
\|\eta\|_{L^\infty(I \times \omega)}\leq L\quad\text{ and  }\quad\|\naby\eta\|_{L^\infty(I \times \omega)}\leq c(L),
\end{align*}
where $L>0$ is a constant, we let
\begin{align*}
\mathrm{Bog}_\eta : C_0^\infty(\Oeta) \rightarrow C_0^\infty(\Oeta; \mathbb{R}^d) \quad\text{with}\quad \divx \mathrm{Bog}_\eta(f)= f-b\int_{\Oeta}f\dx,
\end{align*}
with $b\in C_0^\infty(\Oeta\setminus S_L)$, be the time-dependent Bogovskij operator (see \cite[Theorem 2.11 \& Remark 2.12]{BMSS}). The operator $\mathrm{Bog}_\eta$ vanishes on the boundary $\partial\Oeta$, it commutes with time derivatives, and it continuously maps scalar elements in $W^{s,p}(\Ozeta)$ into vectors in  $W^{s+1,p}(\Ozeta;\mathbb{R}^d)$ for any $p\in(1,\infty)$ and $s\geq0$. Furthermore,  the Hanzawa transform $\bm{\Psi}_\eta$ together with its inverse 
 $\bm{\Psi}_\eta^{-1} : \Oeta \rightarrow\Omega$  possesses the following properties, see \cite{breit2022regularity, BMSS} for details. If for some $\ell,R>0$, we assume that
\begin{align*}
\Vert\eta\Vert_{L^\infty(\omega)}
+
\Vert\zeta\Vert_{L^\infty(\omega)}
< \ell <L \qquad\text{and}\qquad
\Vert\naby\eta\Vert_{L^\infty(\omega)}
+
\Vert\naby\zeta\Vert_{L^\infty(\omega)}
<R
\end{align*}
holds, then for any  $s>0$, $\varrho,p\in[1,\infty]$ and for any $\eta,\zeta \in B^{s}_{\varrho,p}(\omega)\cap W^{1,\infty}(\omega)$ (where $B^{s}_{\varrho,p}$ is a Besov space), we have that the estimates
\begin{align}
\label{210and212}
&\Vert \bm{\Psi}_\eta \Vert_{B^s_{\varrho,p}(\Omega\cup S_\ell)}
+
\Vert \bm{\Psi}_\eta^{-1} \Vert_{B^s_{\varrho,p}(\Omega\cup S_\ell)}
 \lesssim
1+ \Vert \eta \Vert_{B^s_{\varrho,p}(\omega)},
\\
\label{211and213}
&\Vert \bm{\Psi}_\eta - \bm{\Psi}_\zeta  \Vert_{B^s_{\varrho,p}(\Omega\cup S_\ell)} 
+
\Vert \bm{\Psi}_\eta^{-1} - \bm{\Psi}_\zeta^{-1}  \Vert_{B^s_{\varrho,p}(\Omega\cup S_\ell)} 
\lesssim
 \Vert \eta - \zeta \Vert_{B^s_{\varrho,p}(\omega)}
\end{align}
and
\begin{align}
\label{218}
&\Vert \partial_t\bm{\Psi}_\eta \Vert_{B^s_{\varrho,p}(\Omega\cup S_\ell)}
\lesssim
 \Vert \partial_t\eta \Vert_{B^{s}_{ \varrho,p}(\omega)},
\qquad
\eta
\in W^{1,1}(I;B^{s}_{\varrho,p}(\omega))
\end{align}
holds uniformly in time with the hidden constants depending only on the reference geometry, on $L-\ell$ and $R$.

\subsection{The concept of solutions and the main results}
Let us start with a precise definition of what we mean by a weak solution.   We recall that $L>0$ is a sufficiently small constant as introduced in Section \ref{sec:geo}.
\begin{definition}[Weak solution]
\label{def:weaksolmartFP}
Let $(\bff, g, \rho_0, \bT_0, \bu_0, \eta_0, \eta_\star)$
be a dataset that satisfies
\begin{equation}
\begin{aligned}
\label{mainDataForAll}
&\bff \in L^2(I;L^{2}_\mathrm{loc}(\mathbb{R}^2)),
\qquad g\in L^2(I;L^{2}(\omega)) ,
\\&
\eta_0 \in W^{2,2}(\omega) \text{ with } \Vert \eta_0 \Vert_{L^\infty( \omega)} < L, \quad \eta_\star \in L^{2}(\omega),
\\&\bu_0 \in L^{2}_{\divx}(\Omega_{\eta_0} )\text{ is such that }\bu_0 \circ \bm{\varphi}_{\eta_0} =\eta_\star \bn \text{ on } \omega,
\\&
\rho_0\in L^{2}(\Omega_{\eta_0}), \quad
\bT_0\in L^{2}(\Omega_{\eta_0}),
\\&
\rho_0\geq 0,\,\, \bT_0>0 \quad \text{a.e. in } \Omega_{\eta_0}.
\end{aligned}
\end{equation}
We call
$(\eta, \bu, \rho,\bT)$  
a \textit{weak solution} of   \eqref{divfree}--\eqref{interface} with data $(\bff, g, \rho_0, \bT_0, \bu_0, \eta_0, \eta_\star)$  if: 
\begin{itemize}
\item[(a)] the following properties 
\begin{align*}
&\eta\in  W^{1,\infty}\big(I;L^{2}(\omega)  \big)  \cap W^{1,2}\big(I;W^{1,2}(\omega)  \big) \cap L^{\infty}\big(I;W^{2,2}(\omega)  \big),
\\&
\Vert\eta\Vert_{L^\infty(I\times\omega)}<L,
\\
&\bu\in
L^{\infty} \big(I; L^{2}(\Oeta) \big)\cap L^2\big(I;W^{1,2}_{\divx}(\Oeta)  \big),
\\
&
\rho \in   L^{\infty}\big(I;L^{2}(\Oeta)  \big)
\cap 
L^2\big(I;W^{1,2}(\Oeta)  \big),
\\
&
\bT \in   L^{\infty}\big(I;L^{2}(\Oeta)  \big)
\cap 
L^2\big(I;W^{1,2}(\Oeta)  \big)
\end{align*}
holds;
\item[(b)] for all  $  \psi  \in C^\infty (\overline{I}\times \R^2  )$, we have
\begin{equation}
\begin{aligned} 
\label{weakRhoEq}
\int_I  \frac{\mathrm{d}}{\dt}
\int_{\Oeta } \rho\psi \dx \dt 
&=
\int_I
\int_{\Oeta}[\rho\partial_t\psi + (\rho\bu\cdot\nabx) \psi] \dx\dt
\\&-
\int_I\int_{\Oeta}\nabx \rho \cdot\nabx \psi  \dx\dt;
\end{aligned}
\end{equation}
\item[(c)] for all  $  \mathbb{Y}  \in C^\infty (\overline{I}\times \R^2  )$, we have
\begin{equation}
\begin{aligned} 
\label{weakFokkerPlanckEq}
\int_I  \frac{\mathrm{d}}{\dt}
\int_{\Oeta } \bT:\mathbb{Y} \dx \dt 
&=
\int_I
\int_{\Omega_{\eta }}[\bT :\partial_t\mathbb{Y} + \bT:(\bu\cdot\nabx) \mathbb{Y}] \dx\dt
\\&+
\int_I\int_{\Oeta}
[(\nabx \bu )\bT  + \bT (\nabx \bu )^\top]:\mathbb{Y} \dx\dt
\\&
-2\int_I\int_{\Oeta }(\bT :\mathbb{Y}  - \rho  \mathrm{tr}(\mathbb{Y} ))\dx\dt
\\&-
\int_I\int_{\Oeta}\nabx \bT ::\nabx \mathbb{Y}  \dx\dt
\end{aligned}
\end{equation}
where $\nabx\bT::\nabx\mathbb{Y}=\sum_{i=1}^2\partial_{x_i}\bT:\partial_{x_i}\mathbb{Y}$;
\item[(d)] for all  $(\bm{\phi},\phi)  \in C^\infty_{\divx} (\overline{I}\times \R^2 )\otimes  C^\infty (\overline{I}\times \omega )$ with $\bm{\phi}(T,\cdot)=0$, $\phi(T,\cdot)=0$ and $\bm{\phi}\circ \bm{\varphi}_\eta= \phi \bn$, we have
\begin{equation}
\begin{aligned}
\label{weakFluidStrut}
\int_I  \frac{\mathrm{d}}{\dt}\bigg(
\int_{\Oeta } \bu\cdot\bm{\phi} &\dx 
+
\int_\omega\partial_t\eta\phi\dy\bigg) \dt 
\\&=
\int_I
\int_{\Oeta }[\bu \cdot\partial_t\bm{\phi} + \bu\cdot(\bu\cdot\nabx) \bm{\phi}] \dx\dt
\\&-
\int_I\int_{\Oeta }\big[\nabx \bu  :\nabx \bm{\phi}-\bff\cdot\bm{\phi} +\bT:\nabx\bm{\phi}\big] \dx\dt
\\&
+
\int_I\int_\omega\big[\partial_t\eta\partial_t\phi-\partial_t\naby\eta\naby\phi
-
\naby^2\eta\naby^2\phi+g\phi\big]\dy\dt;
\end{aligned}
\end{equation}
\item[(e)] the energy inequality
\begin{equation}
\begin{aligned}
\sup_{t\in I}&
\bigg(
\int_{\Oeta}\mathrm{tr}(\bT(t))\dx
+
\Vert \bu(t)\Vert_{L^2(\Oeta)}^2 
+
 \Vert\partial_t \eta(t)\Vert_{L^2(\omega)}^2 
+ 
\Vert\partial_y^2\eta(t)\Vert_{L^2(\omega)}^2
\bigg)
\\&+
\int_I
\int_{\Oeta}\mathrm{tr}(\bT)\dx\dt
+\int_I\Vert \nabx \bu \Vert_{L^2(\Oeta)}^2\dt
+\int_I\Vert\partial_t\partial_y \eta \Vert_{L^2(\omega)}^2\dt
\\&\lesssim
\int_{\Omega_{\eta_{0}}}\mathrm{tr}(\bT_0)\dx
+
\Vert \bu_0\Vert_{L^2(\Omega_{\eta_{0}})}^2 
+
 \Vert\eta_\star\Vert_{L^2(\omega)}^2 
+ 
\Vert\partial_y^2\eta_0\Vert_{L^2(\omega)}^2
\\&+
T\int_{\Omega_{\eta_{0}}}\rho_0 \dx
+
\int_I\Vert \bff\Vert_{L^2(\Oeta)}^2\dt
+
\int_I\Vert g\Vert_{L^2(\omega)}^2\dt
\end{aligned}
\end{equation}
holds.
\item[(f)] In addition, we have
\begin{equation}
\begin{aligned}   
\sup_{t\in I}&
\big(\Vert\rho(t)\Vert_{L^2(\Oeta)}^2
+
\Vert\bT(t)\Vert_{L^2(\Oeta)}^2
\big)
+
\int_I\big(\Vert \rho \Vert_{W^{1,2}(\Oeta)}^2
+
\Vert  \bT \Vert_{W^{1,2}(\Oeta)}^2\big)\dt
\\&+
\int_I \Vert\bT\Vert_{L^2(\Oeta)}^2\dt 
\lesssim
\Vert\rho_0 \Vert_{L^2(\Omega_{\eta_0})}^2
+
\big(
\Vert\bT_0\Vert_{L^2(\Omega_{\eta_0})}^2
+
T \Vert\rho_0  \Vert_{L^2(\Omega_{\eta_{0}})}^2
\big)
\\&\qquad\qquad\qquad\qquad\qquad
\times 
\exp\bigg(
c
\int_I
\Vert\nabx \bu\Vert_{L^2(\Oeta)}^2 \dt\bigg).
\end{aligned}
\end{equation}
\end{itemize}
\end{definition}
With this definition in hand, we now state our first main result.
\begin{theorem}
\label{thm:main}
For a dataset $(\bff, g, \rho_0, \bT_0, \bu_0, \eta_0, \eta_\star)$ 
that satisfies \eqref{mainDataForAll}, there  exists a 
weak solution $(\eta, \bu, \rho,\bT)$ of   \eqref{divfree}--\eqref{interface}.  
\end{theorem} 
\begin{remark}
As in \cite{barrettBoy2011existence} (see also \cite[Section 8.2]{barrett2017existenceOldroyd}),  a solution $\bT$ of \eqref{solute} remains strictly positive if it were initially so.
Furthermore,  a solution $\rho$ of  \eqref{rhoEqu} also remains nonnegative if it were initially nonnegative.
Indeed,  if we test  \eqref{rhoEqu} with the nonpositive part $\rho_-=\min\{0,\rho\}$ of $\rho$, integrate over $\Oeta $ and use the boundary condition \eqref{bddCondSolv} together with Reynold's transport theorem, we obtain
\begin{align*}
\frac{1}{2}\frac{\dd}{\dt}\int_{\Oeta } \vert \rho_-\vert^2 \dx
+\int_{\Oeta } \vert \nabx \rho_-\vert^2\dx =0.
\end{align*} 
Therefore,   it follows that  $ \rho_-=0$ a.e. in $\Oeta$ for any $t\in I$ and thus, $ \rho=\rho_+=\max\{0,\rho\}$. Rigourously showing the strict positivity of $\rho$ an $\mathbb{T}$, however, requires another approximating layer. For simplicity, however, we omit the requirement of preserving positivity in our notion of a weak solutions. 
\end{remark}

We note that since the test function for the (weak) momentum equation in \eqref{weakFluidStrut} is divergence-free, our weak formulation is consequently pressure-free. 
However, when the weak solution and the forcing $\bff$ are sufficiently regular, the pressure can be recovered by solving an elliptic problem obtained by taking the divergence in \eqref{momEq}.
A first step to obtaining said regularity is by showing the existence of a strong solution which is the subject of our second main result. Here, by a `strong solution', we mean the following.
\begin{definition}[Strong solution]
\label{def:strongSolution}
Let $(\bff, g, \rho_0, \bT_0, \bu_0, \eta_0, \eta_\star)$
be a dataset that satisfies
\begin{equation}
\begin{aligned}
\label{mainDataForAllStrong}
&\bff \in L^2(I;L^{2}_\mathrm{loc}(\mathbb{R}^2)),
\qquad g\in L^2(I;L^{2}(\omega)) ,
\\&
\eta_0 \in W^{3,2}(\omega) \text{ with } \Vert \eta_0 \Vert_{L^\infty( \omega)} < L, \quad \eta_\star \in W^{1,2}(\omega),
\\&\bu_0 \in W^{1,2}_{\divx}(\Omega_{\eta_0} )\text{ is such that }\bu_0 \circ \bm{\varphi}_{\eta_0} =\eta_\star \bn \text{ on } \omega,
\\&
\rho_0\in W^{1,2}(\Omega_{\eta_0}), \quad
\bT_0\in W^{1,2}(\Omega_{\eta_0}),
\\&
\rho_0\geq 0,\,\, \bT_0>0 \quad \text{a.e. in } \Omega_{\eta_0}.
\end{aligned}
\end{equation}
We call 
$(\eta, \bu, p, \rho,\bT)$ 
a \textit{strong solution} of   \eqref{divfree}--\eqref{interface} with data $(\bff, g, \rho_0, \bT_0, \bu_0, \eta_0, \eta_\star)$  if:
\begin{itemize}
\item[(a)]  $(\eta, \bu, \rho,\bT)$ is a weak solution of   \eqref{divfree}--\eqref{interface};
\item[(b)] the structure-function $\eta$ is such that $
\Vert \eta \Vert_{L^\infty(I \times \omega)} <L$ and
\begin{align*}
\eta \in &W^{1,\infty}\big(I;W^{1,2}(\omega)  \big)\cap L^{\infty}\big(I;W^{3,2}(\omega)  \big) \cap  W^{1,2}\big(I; W^{2,2}(\omega)  \big)
\\&\cap  W^{2,2}\big(I;L^{2}(\omega)  \big)  
\cap
L^{2} (I;W^{4,2}(\omega))
;
\end{align*}
\item[(c)] the velocity $\bu$ is such that $\bu  \circ \bm{\varphi}_{\eta} =(\partial_t\eta)\bn$ on $I\times \omega$ and
\begin{align*} 
\bu\in  W^{1,2} \big(I; L^2_{\divx}(\Oeta ) \big)\cap L^2\big(I;W^{2,2}(\Oeta)  \big);
\end{align*}
\item[(d)] the pressure $p$ is such that 
\begin{align*}
p\in L^2\big(I;W^{1,2}(\Oeta)  \big);
\end{align*}
\item[(e)] the pair $(\rho,\bT)$ is such that 
\begin{align*}
\rho,\bT \in W^{1,2}\big(I;L^{2}(\Oeta)  \big) \cap L^\infty\big(I;W^{1,2}(\Oeta)  \big)\cap L^2\big(I;W^{2,2}(\Oeta)  \big);
\end{align*}
\item[(f)] equations \eqref{divfree}--\eqref{solute} are satisfied a.e. in spacetime with $\eta(0)=\eta_0$ and $\partial_t\eta(0)=\eta_\star$ a.e. in $\omega$, as well as $\bfu(0)=\bfu_0$, $\rho(0)=\rho_0$ and $\bT(0)=\bT_0$ a.e. in $\Omega_{\eta_0}$.
\end{itemize}
\end{definition}
Our second main result concerning the existence of a strong solution is given as follows:
\begin{theorem}
\label{thm:strongSol}
For a dataset $(\bff, g, \rho_0,  \bT_0, \bu_0, \eta_0, \eta_\star)$ 
that satisfies \eqref{mainDataForAllStrong}, there  exists a 
strong solution $(\eta, \bu,p, \rho,\bT)$ of   \eqref{divfree}--\eqref{interface}. 
\end{theorem} 
Finally, a consequence of Theorem \ref{thm:strongSol} and its proof is the following unconditional weak-strong uniqueness result.
\begin{theorem}\label{thm:weakstrong}
Let $(\eta_w, \bu_w, \rho_w,\bT_w)$   be a  weak solution of   \eqref{divfree}--\eqref{interface} with dataset $(\bff_w, g_w, \rho_{0,w}, \bT_{0,w}, \bu_{0,w}, \eta_{0,w}, \eta_{\star,w})$ in the sense of Definition \ref{def:weaksolmartFP} and let $(\eta_s, \bu_s, \rho_s,\bT_s)$   be a  strong solution of   \eqref{divfree}--\eqref{interface} with dataset $(\bff_s, g_s, \rho_{0,s}, \bT_{0,s}, \bu_{0,s}, \eta_{0,s}, \eta_{\star,s})$ in the sense of Definition \ref{def:strongSolution}.
Set
\begin{align*}
&\rho_{ws}:=\rho_w-\overline{\rho}_s,
\quad
\bT_{ws}:=\bT_w-\overline{\bT}_s,
\quad
\bu_{ws}=\bu_w-\overline{\bu}_s, 
\\&
\eta_{ws}=\eta_w-\eta_s,
\quad
\bff_{ws}:=\bff_w-\overline{\bff}_s,
\quad
g_{ws}=g_w-g_s, 
\end{align*} 
where $\overline{f}_s:=f_s\circ \bm{\Psi}_{\eta_s-\eta_w}$. 
Then the inequality 
\begin{align*}
&\sup_{t\in I}\Big(
\Vert\rho_{ws}(t)\Vert_{L^2(\Omega_{\eta_w(t)})}^2
+
\Vert\bT_{ws}(t)\Vert_{L^2(\Omega_{\eta_w(t)})}^2
+
\Vert \bu_{ws}(t)\Vert_{L^2(\Omega_{\eta_w(t)})}^2 
+
 \Vert\partial_t \eta_{ws}(t)\Vert_{L^2(\omega)}^2 \Big)
\\
&\qquad\qquad\qquad+ 
\sup_{t\in I}
\Vert\partial_y^2\eta_{ws}(t)\Vert_{L^2(\omega)}^2
+\int_I\Big(
\Vert \nabx\rho_{ws} \Vert_{L^2(\Omega_{\eta_w})}^2
+
\Vert  \nabx\bT_{ws} \Vert_{L^2(\Omega_{\eta_w})}^2
\Big)
\dt
\\
&\qquad\qquad\qquad+
\int_I\Big(
\Vert \nabx \bu_{ws}\Vert_{L^2(\Omega_{\eta_w})}^2 
+
\Vert \partial_t\naby  \eta_{ws} \Vert_{L^2(\omega)}^2
\Big)
\dt
\\
&\lesssim 
\Vert\rho_{ws}(0) \Vert_{L^2(\Omega_{\eta_w(0)})}^2
+ 
\Vert\bT_{ws}(0)\Vert_{L^2(\Omega_{\eta_w(0)})}^2
+
\Vert \bu_{ws}(0) \Vert_{L^2(\Omega_{\eta_w(0)})}^2
+\Vert\partial_t \eta_{ws}(0)\Vert_{L^2(\omega)}^2
\\
&\qquad\qquad\qquad 
+\Vert \naby^2 \eta_{ws}(0) \Vert_{L^2(\omega)}^2+
\int_I\Big(\Vert \mathbf f_{ws}\Vert_{L^2(\Omega_{\eta_w})}^2+\Vert g_{ws}\Vert_{L^2(\omega)}^2 \Big)\dt
\end{align*}
hold.
\end{theorem}

\section{A priori estimates} 
\label{sec:Apriori}
In this section, we derive formal estimates in energy norms satisfied by smooth solutions of \eqref{divfree}-\eqref{solute}. The first two estimates will be shown to be satisfied by weak solutions in subsequent sections and the last estimate will be satisfied by a strong solution.
\begin{proposition}
Any smooth solution $(\eta, \bu, p, \rho,\bT)$ 
 of   \eqref{divfree}--\eqref{interface} with smooth dataset $(\bff, g, \rho_0, \bT_0, \bu_0, \eta_0, \eta_\star)$ satisfies
\begin{equation}
\begin{aligned}
\label{apriori1}
\sup_{t\in I}&
\bigg( 
\int_{\Oeta}\mathrm{tr}(\bT(t))\dx
+
\Vert \bu(t)\Vert_{L^2(\Oeta)}^2 
+
 \Vert\partial_t \eta(t)\Vert_{L^2(\omega)}^2 
+ 
\Vert\partial_y^2\eta(t)\Vert_{L^2(\omega)}^2
\bigg)
\\&+
\int_I
\int_{\Oeta}\mathrm{tr}(\bT)\dx\dt
+\int_I\Vert \nabx \bu \Vert_{L^2(\Oeta)}^2\dt
+\int_I\Vert\partial_t\partial_y \eta \Vert_{L^2(\omega)}^2\dt
\\&\lesssim
\int_{\Omega_{\eta_{0}}}\mathrm{tr}(\bT_0)\dx
+
\Vert \bu_0\Vert_{L^2(\Omega_{\eta_{0}})}^2 
+
 \Vert\eta_\star\Vert_{L^2(\omega)}^2 
+ 
\Vert\partial_y^2\eta_0\Vert_{L^2(\omega)}^2
\\&+
T
\int_{\Omega_{\eta_{0}}}\rho_0 \dx
+
\int_I\Vert \bff\Vert_{L^2(\Oeta)}^2\dt
+
\int_I\Vert g\Vert_{L^2(\omega)}^2\dt.
\end{aligned}
\end{equation}
\end{proposition}
\begin{proof}
Take $(\bm{\phi},\phi)=(\bu,\partial_t\eta)$ in \eqref{weakFluidStrut} to obtain 
\begin{equation}
\begin{aligned}
\label{enerFormalFluiStru}
\frac{1}{2}\int_I  \frac{\mathrm{d}}{\dt}\Big(\Vert \bu\Vert_{L^2(\Oeta)}^2 
&+
 \Vert\partial_t \eta\Vert_{L^2(\omega)}^2 
+ 
\Vert\partial_y^2\eta\Vert_{L^2(\omega)}^2
\Big)\dt
\\&\qquad+
 \int_I\Big(\Vert \nabx \bu\Vert_{L^2(\Oeta)}^2
 +
 \Vert\partial_t\partial_y \eta\Vert_{L^2(\omega)}^2\Big)\dt
\\&
=-\int_I\int_{\Oeta}\bT :\nabx \bu \dx\dt
+
\int_I\int_{\Oeta}\bff \cdot \bu \dx\dt
\\&\qquad+
\int_I \int_\omega g\partial_t \eta\dy\dt
\end{aligned}
\end{equation}
where
\begin{align*}
&\int_I\int_{\Oeta}\bff \cdot \bu \dx\dt
\leq
c
\int_I\Vert \bff\Vert_{L^2(\Oeta)}^2\dt
+
\frac{1}{4}
\sup_{t\in I}\Vert\bu(t)\Vert_{L^2(\Oeta)}^2 ,
\\&
\int_I \int_\omega g\partial_t \eta\dy\dt
\leq
c
\int_I\Vert g\Vert_{L^2(\omega)}^2\dt
+
\frac{1}{4}
\sup_{t\in I}\Vert\partial_t \eta(t)\Vert_{L^2(\omega)}^2 .
\end{align*}
On the other hand, if we take the trace in \eqref{solute}, integrate and use \eqref{bddCondSolv}  and the relation $\mathrm{tr}(\mathbb{A}\mathbb{B}^\top)=\mathbb{A}:\mathbb{B}$ which holds for all $\mathbb{A},\mathbb{B}\in \mathbb{R}^{d\times d}$, we obtain
\begin{equation}
\begin{aligned}
\label{freeEnergyEst}
\frac{1}{2}\int_I\frac{\dd}{\dt}\int_{\Oeta} \mathrm{tr}(\bT)\dx\dt
+&
  \int_I
  \int_{\Oeta}
  \mathrm{tr}(\bT)\dx\dt 
\\&=
\int_I
\int_{\Oeta}
 \bT:\nabx \bu \dx \dt
+3
\int_I
\int_{\Oeta}\rho \dx\dt.
\end{aligned}
\end{equation}
Now note that due to the Neumann boundary condition \eqref{bddCondSolv} for $\rho$, if we integrate \eqref{rhoEqu} over space, apply Gauss theorem and the fact that the left-hand side is transported by divergence-free velocity, it follows that
\begin{align*}
\int_{\Oeta}\rho \dx=\int_{\Omega_{\eta_{0}}}\rho_0 \dx.
\end{align*} 
Combining everything finishes the proof.
\end{proof}

\begin{proposition}  
Any smooth solution $(\eta, \bu, p,\rho,\bT)$
 of   \eqref{divfree}--\eqref{interface} with smooth dataset $(\bff, g, \rho_0, \bT_0, \bu_0, \eta_0, \eta_\star)$  satisfies  
\begin{equation}
\begin{aligned}
\label{strgEstFP3}  
\sup_{t\in I}&
\big(\Vert\rho(t)\Vert_{L^2(\Oeta)}^2
+
\Vert\bT(t)\Vert_{L^2(\Oeta)}^2
\big)
+\int_I\big(\Vert \rho \Vert_{W^{1,2}(\Oeta)}^2
+
\Vert  \bT \Vert_{W^{1,2}(\Oeta)}^2
\big)\dt
\\&
\lesssim
\Vert\rho_0 \Vert_{L^2(\Omega_{\eta_0})}^2
+
\big(
\Vert\bT_0\Vert_{L^2(\Omega_{\eta_0})}^2
+
T \Vert\rho_0  \Vert_{L^2(\Omega_{\eta_{0}})}^2
\big) 
\exp\bigg(
c
\int_I
\Vert\nabx \bu\Vert_{L^2(\Oeta)}^2 \dt\bigg) .
\end{aligned}
\end{equation}
\end{proposition}
\begin{proof}
If we set $ \psi=\rho$ in \eqref{weakRhoEq}, we obtain
\begin{align*}  
\int_I  \frac{\mathrm{d}}{\dt}
\int_{\Oeta } \vert\rho\vert^2 \dx \dt 
&=
\int_I
\int_{\Oeta}[\rho\partial_t\rho + (\rho\bu\cdot\nabx) \rho] \dx\dt
-
\int_I\int_{\Oeta}\vert\nabx \rho \vert^2  \dx\dt.
\end{align*}
Thus, by Reynold's transport theorem, we obtain
\begin{equation}
\begin{aligned}  
\label{weakRhoEqx}
\frac{1}{2}\sup_{t\in I} 
\Vert\rho(t)\Vert_{L^2(\Oeta)}^2 
+
\int_I \Vert\nabx \rho \Vert_{L^2(\Oeta)}^2   \dt
=
\frac{1}{2}
\Vert\rho_0 \Vert_{L^2(\Omega_{\eta_0})}^2
\end{aligned}
\end{equation}
On the other hand, if we set $\mathbb{Y}=\bT$ in \eqref{weakFokkerPlanckEq}, we obtain  
\begin{align*}
\int_I  \frac{\mathrm{d}}{\dt}
\int_{\Oeta } \vert\bT\vert^2 \dx \dt 
&=
\int_I
\int_{\Omega_{\eta }}[\bT :\partial_t\bT + \bT:(\bu\cdot\nabx) \bT] \dx\dt
\\&+
\int_I\int_{\Omega_{\eta }}
[(\nabx \bu )\bT  + \bT (\nabx \bu )^\top]:\bT \dx\dt
\\&
-2\int_I\int_{\Omega_{\eta }}(\vert\bT\vert^2  - \rho  \mathrm{tr}(\bT ))\dx\dt
-
\int_I\int_{\Omega_{\eta }}\vert\nabx \bT \vert^2 \dx\dt.
\end{align*}
If we now use the estimate
\begin{align*}
\int_{\Oeta} \vert \mathrm{tr}(\bT)\vert^2\dx
\leq
\int_{\Oeta} \vert \bT\vert^2\dx
\end{align*}
and Reynold's transport theorem, we obtain
\begin{equation}
\begin{aligned}
\label{strgEstFP}
\frac{1}{2}\int_I\frac{\dd}{\dt}&\Vert\bT(t)\Vert_{L^2(\Oeta)}^2\dt
+
\int_I\Vert  \bT \Vert_{W^{1,2}(\Oeta)}^2\dt
\\&\leq
2\int_I\int_{\Oeta} \vert\nabx \bu\vert\vert\bT\vert^2\dx \dt
+
\int_I\Vert\rho\Vert_{L^2(\Oeta)}^2 \dt
\\&\leq
2\int_I\int_{\Oeta} \vert\nabx \bu\vert\vert\bT\vert^2\dx \dt
+
T\Vert\rho_0\Vert_{L^2(\Omega_{\eta_0})}^2 
\end{aligned}
\end{equation}
where the second inequality follows from \eqref{weakRhoEqx}. We now note that 
\begin{align*}
2\int_{\Oeta} \vert\nabx \bu\vert\vert\bT\vert^2\dx
&\lesssim
\Vert\nabx \bu\Vert_{L^2(\Oeta)} \Vert\bT\Vert_{L^4(\Oeta)}^2
\\
&\lesssim
\Vert\nabx \bu\Vert_{L^2(\Oeta)} \Vert\bT\Vert_{L^2(\Oeta)}\Vert  \bT \Vert_{W^{1,2}(\Oeta)}
\\
&\leq
\delta
\Vert  \bT \Vert_{W^{1,2}(\Oeta)}
+
c
\Vert\nabx \bu\Vert_{L^2(\Oeta)}^2 \Vert\bT\Vert_{L^2(\Oeta)}^2
\end{align*}
holds for any $\delta>0$.
Consequently, 
it follows from Gr\"onwall's lemma that
\begin{equation}
\begin{aligned} 
\sup_{t\in I}&
\Vert\bT(t)\Vert_{L^2(\Oeta)}^2
+\int_I\Vert  \bT \Vert_{W^{1,2}(\Oeta)}^2\dt 
\\&
\lesssim
\big(
\Vert\bT_0\Vert_{L^2(\Omega_{\eta_0})}^2
+
T \Vert\rho_0  \Vert_{L^2(\Omega_{\eta_{0}})}^2
\big) 
\exp\bigg(
c
\int_I
\Vert\nabx \bu\Vert_{L^2(\Oeta)}^2 \dt\bigg).
\end{aligned}
\end{equation}
Combining everything finishes the proof.
\end{proof} 
We now present a result that will be satisfied by strong solutions.
\begin{proposition}
\label{prop:strongEst}
Any smooth solution $(\eta, \bu, p,\rho,\bT)$ 
 of   \eqref{divfree}--\eqref{interface} with smooth dataset $(\bff, g, \rho_0, \bT_0, \bu_0, \eta_0, \eta_\star)$ satisfies
 \begin{equation}
\begin{aligned} 
\int_I&\big(\Vert \partial_t \rho\Vert_{L^2(\Oeta)}^2
+
\Vert \partial_t \bT\Vert_{L^2(\Oeta)}^2
+
\Vert \partial_t \bu\Vert_{L^2(\Oeta)}^2
+
\Vert \partial_t^2\eta\Vert_{L^{2}(\omega)}^2
\big)\dt
\\&+
\sup_{t\in I} \Big(\Vert \rho(t)\Vert_{W^{1,2}(\Oeta)}^2
+
\Vert \bT(t)\Vert_{W^{1,2}(\Oeta)}^2
+
\Vert \bu(t)\Vert_{W^{1,2}(\Oeta)}^2
+
\Vert \partial_t\eta(t)\Vert_{W^{1,2}(\omega)}^2\Big)
\\&+\sup_{t\in I}
\Vert  \eta(t)\Vert_{W^{3,2}(\omega)}^2
+\int_I\big(\Vert  \rho\Vert_{W^{2,2}(\Oeta)}
+
\Vert  \bT\Vert_{W^{2,2}(\Oeta)}
+
\Vert  \bu\Vert_{W^{2,2}(\Oeta)}\big)\dt
\\&+\int_I\big(
\Vert  \nabx p\Vert_{L^{2}(\Oeta)}
+
\Vert \partial_t\eta\Vert_{W^{2,2}(\omega)}^2
+
\Vert \eta\Vert_{W^{4,2}(\omega)}^2
\big)\dt
\\&\lesssim
\Vert  \rho_0\Vert_{W^{1,2}(\Omega_{\eta_0})}
+
\Vert  \bT_0\Vert_{W^{1,2}(\Omega_{\eta_0})}
+
\Vert  \bu_0\Vert_{W^{1,2}(\Omega_{\eta_0})}
+
 \Vert\eta_\star\Vert_{W^{1,2}(\omega)}^2 
+ 
\Vert\eta_0\Vert_{W^{3,2}(\omega)}^2
\\&
+
\int_I\big(\Vert \bff\Vert_{L^2(\Oeta)}^2
+
\Vert g\Vert_{L^{2}(\omega)}^2\big)\dt
+
T
\big(
\Vert\bT_0\Vert_{L^2(\Omega_{\eta_0})}^2
+
T \Vert\rho_0  \Vert_{L^2(\Omega_{\eta_0})}^2
\big)
\exp(
c
h_0).
\label{eq:thm:mainFP00}
\end{aligned}
\end{equation} 
where $h_0$ is the right-hand side of \eqref{apriori1}.
 \end{proposition}
 \begin{proof}
 We begin with the following acceleration estimate for the solvent-structure subsystem \eqref{divfree} and \eqref{momEq}-\eqref{shellEQ} (see \cite[(5.4)]{breit2022regularity} and \cite[(4.5)]{BMSS}) 
\begin{equation}
\begin{aligned}
\label{acceleration}
\sup_{t\in I}&
\bigg( 
\Vert \nabx\bu(t)\Vert_{L^2(\Oeta)}^2 
+
 \Vert\partial_t \partial_y\eta(t)\Vert_{L^2(\omega)}^2 
+ 
\Vert\partial_y^3\eta(t)\Vert_{L^2(\omega)}^2
\bigg)
\\&\qquad+
\int_I \big(
\Vert \nabx^2 \bu \Vert_{L^2(\Oeta)}^2
+
\Vert \partial_t \bu \Vert_{L^2(\Oeta)}^2
+
\Vert \nabx p \Vert_{L^2(\Oeta)}^2\big)\dt
\\&\qquad+
\int_I \big(
\Vert\partial_t\partial_y^2 \eta \Vert_{L^2(\omega)}^2
+
\Vert\partial_t^2 \eta \Vert_{L^2(\omega)}^2
\big)\dt
\\&\lesssim 
\Vert \nabx \bu_0\Vert_{L^2(\Omega_{\eta_{0}})}^2 
+
 \Vert \partial_y\eta_\star\Vert_{L^2(\omega)}^2 
+ 
\Vert\partial_y^3\eta_0\Vert_{L^2(\omega)}^2
\\&\qquad
+
\int_I\Vert \bff\Vert_{L^2(\Oeta)}^2\dt
+
\int_I\Vert \nabx \bT\Vert_{L^2(\Oeta)}^2\dt
+
\int_I\Vert g\Vert_{L^2(\omega)}^2\dt.
\end{aligned}
\end{equation}
The constant in the bound depends only on the right-hand side of \eqref{apriori1} and the estimate is the key  ingredient in showing the global-in-time existence of strong solutions for the solvent-structure subsystem \eqref{divfree} and \eqref{momEq}-\eqref{shellEQ} satisfying $\Vert\eta\Vert_{L^\infty(I\times\omega)}<L$. The original estimate \cite[(5.4)]{breit2022regularity} required the $L^2$-norm in time of $\Vert \partial_yg\Vert_{L^2(\omega)}^2$  on the right-hand side of \eqref{acceleration}. The fact that  $\Vert g\Vert_{L^2(\omega)}^2$ (rather than $\Vert \partial_yg\Vert_{L^2(\omega)}^2$) is sufficient is shown in \cite[(4.5)]{BMSS}.
We also note that by \eqref{strgEstFP3}  
\begin{equation}
\begin{aligned}
\int_I\Vert \nabx \bT\Vert_{L^2(\Oeta)}^2\dt
\lesssim
\Vert\rho_0 \Vert_{L^2(\Omega_{\eta_0})}^2
+
\big(
\Vert\bT_0\Vert_{L^2(\Omega_{\eta_0})}^2
+
T \Vert\rho_0  \Vert_{L^2(\Omega_{\eta_{0}})}^2
\big) 
e^{
c
h_0}
\end{aligned}
\end{equation}
where
\begin{equation}
\begin{aligned}
\label{h0}
h_0&:=\int_{\Omega_{\eta_{0}}}\mathrm{tr}(\bT_0)\dx
+
\Vert \bu_0\Vert_{L^2(\Omega_{\eta_{0}})}^2 
+
 \Vert\eta_\star\Vert_{L^2(\omega)}^2 
+ 
\Vert\partial_y^2\eta_0\Vert_{L^2(\omega)}^2
\\&+
T
\int_{\Omega_{\eta_{0}}}\rho_0 \dx
+
\int_I\Vert \bff\Vert_{L^2(\Oeta)}^2\dt
+
\int_I\Vert g\Vert_{L^2(\omega)}^2\dt.
\end{aligned}
\end{equation}
%
We now test \eqref{rhoEqu}  with $\Delx \rho$. This yields
\begin{equation}
\begin{aligned}  
\label{spaceRegRho0}
\int_I  \frac{\mathrm{d}}{\dt}
\Vert\nabx \rho\Vert_{L^2(\Oeta)}^2 \dt 
&+
\int_I\Vert\Delx \rho\Vert_{L^2(\Oeta)}^2\dt
\\&=
\int_I
\int_{\Oeta}((\bu\cdot \nabx) \rho)\Delx\rho \dx\dt
\\&\quad+
\frac{1}{2}
\int_I
\int_{\partial\Oeta}(\partial_t\eta\bn)\circ\bm{\varphi}_\eta^{-1}\cdot\bn_\eta \vert\nabx\rho\vert^2 \dd\mathcal{H}^1\dt
\\&=:I_1+I_2
\end{aligned}
\end{equation}
where
\begin{align*}
I_1\leq \delta
\int_I\Vert\Delx \rho\Vert_{L^2(\Oeta)}^2\dt
+
c(\delta)
\int_I\Vert\bu\Vert_{W^{2,2}(\Oeta)}^2\Vert\nabx \rho \Vert_{L^2(\Oeta)}^2  \dt
\end{align*}
for any $\delta>0$, and by using $\eta\in L^\infty(I;W^{1,\infty}(\omega))$ and $\partial_t\eta\in L^\infty(I;W^{1,2}(\omega))$ (which follows from \eqref{acceleration}), we also obtain 
\begin{align*}
I_2&\lesssim
\int_I\Vert\nabx \rho \Vert_{L^2(\partial\Oeta)}^2\Vert(\partial_t\eta\bn)\circ\bm{\varphi}_\eta^{-1}\cdot\bn_\eta  \Vert_{L^\infty(\partial\Oeta)} \dt
\\
&\lesssim
\int_I\Vert  \rho \Vert_{W^{2,2}(\Oeta)}  \Vert\nabx \rho \Vert_{W^{3/4,2}(\Oeta)}
\Vert\naby\eta  \Vert_{L^{\infty}(\omega)}
\Vert\partial_t\eta  \Vert_{W^{5/4,2}(\omega)} \dt
\\
&\lesssim 
\int_I\Vert  \rho \Vert_{W^{2,2}(\Oeta)}^{7/4}  \Vert\nabx \rho \Vert_{L^{2}(\Oeta)}^{1/4}
\Vert\partial_t\eta  \Vert_{W^{1,2}(\omega)} ^{3/4}
\Vert\partial_t\eta  \Vert_{W^{2,2}(\omega)}^{1/4}  \dt
\\
&\leq
\delta
\int_I\Vert  \rho \Vert_{W^{2,2}(\Oeta)}^2+c(\delta)\int_I  \Vert\nabx \rho \Vert_{L^{2}(\Oeta)}^2\Vert\partial_t\eta  \Vert_{W^{2,2}(\omega)}^2  \dt
\end{align*}
or any $\delta>0$.
Consequently, it follows from Gr\"onwall's lemma that
\begin{equation}
\begin{aligned}
\label{spaceRegRho1}
\sup_{t\in I}\Vert \nabx\rho(t)\Vert_{L^2(\Oeta)}&+\int_I\Vert \Delx\rho\Vert_{L^2(\Oeta)}\dt
\\&\lesssim
\Vert \nabx\rho_0\Vert_{L^2(\Omega_{\eta_0})}
+
\int_I\big(\Vert\partial_t\eta  \Vert_{W^{2,2}(\omega)}^2+\Vert\bu\Vert_{W^{2,2}(\Oeta)}^2\big)\dt.
\end{aligned}
\end{equation}
If we also test \eqref{solute} with $\Delx \bT$, we obtain 
\begin{equation}
\begin{aligned}  
\label{spaceRegT0}
\int_I  \frac{\mathrm{d}}{\dt}
\Vert\nabx \bT\Vert_{L^2(\Oeta)}^2\dt 
&+
\int_I\Vert\Delx \bT\Vert_{L^2(\Oeta)}^2\dt
\\=&
\int_I
\int_{\Oeta}((\bu\cdot \nabx) \bT)\Delx\bT \dx\dt
\\&+
\frac{1}{2}
\int_I
\int_{\partial\Oeta}(\partial_t\eta\bn)\circ\bm{\varphi}_\eta^{-1}\cdot\bn_\eta \vert\nabx\bT\vert^2 \dd\mathcal{H}^1\dt
\\&+
2\int_I
\int_{\Oeta}(\bT-\rho\mathbb{I})\Delx\bT \dx\dt
\\&
-\int_I
\int_{\Oeta}((\nabx \bu)\bT + \bT(\nabx \bu)^\top)\Delx\bT \dx\dt
\\=:&J_1+J_2+J_3+J_4.
\end{aligned}
\end{equation}
The terms $J_1$ and $J_2$ can be treated as $I_1$ and $I_2$ above. By using \eqref{strgEstFP3}, we obtain
\begin{align*}
J_3\leq &\delta
\int_I\Vert\Delx \bT\Vert_{L^2(\Oeta)}^2\dt
+
cT
\big(
\Vert\bT_0\Vert_{L^2(\Omega_{\eta_0})}^2
+
 T \Vert\rho_0  \Vert_{L^2(\Omega_{\eta_0})}^2
\big)
\exp\bigg(
c
\int_I
\Vert\nabx \bu\Vert_{L^2(\Oeta)}^2 \dt\bigg)
\end{align*}
for any $\delta>0$ and by Ladyzhenskaya's inequality and  \eqref{strgEstFP3},
\begin{align*}
J_4&\lesssim
\int_I  
\Vert\nabx \bu\Vert_{L^2(\Oeta)}^{1/2}\Vert \bu\Vert_{W^{2,2}(\Oeta)}^{1/2}\Vert \bT\Vert_{L^2(\Oeta)}^{1/2}\Vert\nabx \bT\Vert_{L^2(\Oeta)}^{1/2}\Vert\Delx \bT\Vert_{L^2(\Oeta)}
\dt
\\
\leq &\delta
\int_I\Vert\Delx \bT\Vert_{L^2(\Oeta)}^2\dt
+
c\int_I \Vert \bu\Vert_{W^{2,2}(\Oeta)}^2\Vert  \bT \Vert_{W^{1,2}(\Oeta)}^2
\dt
\\&+
cT
\big(
\Vert\bT_0\Vert_{L^2(\Omega_{\eta_0})}^2
+
 T \Vert\rho_0  \Vert_{L^2(\Omega_{\eta_0})}^2
\big)
\exp\bigg(
c
\int_I
\Vert\nabx \bu\Vert_{L^2(\Oeta)}^2 \dt\bigg).
\end{align*}
Thus, it follows from Gr\"onwall's lemma that
\begin{equation}
\begin{aligned}
\label{spaceRegT1}
\sup_{t\in I}&\Vert  \bT(t)\Vert_{W^{1,2}(\Oeta)}^2
+
\int_I\Vert \Delx\bT\Vert_{L^2(\Oeta)}^2\dt
\\&\lesssim
\Vert \nabx\bT_0\Vert_{L^2(\Omega_{\eta_0})}^2
+
\int_I\big(\Vert\partial_t\eta  \Vert_{W^{2,2}(\omega)}^2+\Vert\bu\Vert_{W^{2,2}(\Oeta)}^2\big)\dt
\\&
+
T
\big(
\Vert\bT_0\Vert_{L^2(\Omega_{\eta_0})}^2
+
T \Vert\rho_0  \Vert_{L^2(\Omega_{\eta_0})}^2
\big)
\exp\bigg(
c
\int_I
\Vert\nabx \bu\Vert_{L^2(\Oeta)}^2 \dt\bigg).
\end{aligned}
\end{equation}
To obtain regularity in time, we test \eqref{rhoEqu}  with $\partial_t \rho$. This yields
\begin{equation}
\begin{aligned}
\label{timeRegRho0}
\int_I\Vert \partial_t \rho\Vert_{L^2(\Oeta)}^2\dt
&+
\int_I\frac{\dd}{\dt}\Vert \nabx \rho\Vert_{L^2(\Oeta)}^2\dt
\\&=
\frac{1}{2}
\int_I
\int_{\partial\Oeta}(\partial_t\eta\bn)\circ\bm{\varphi}_\eta^{-1}\cdot\bn_\eta \vert\nabx\rho\vert^2 \dd\mathcal{H}^1\dt
\\&
-\int_I\int_{\Oeta}(\bu\cdot \nabx) \rho\partial_t\rho\dx\dt
\\
&\leq
c
\int_I\Vert  \rho \Vert_{W^{2,2}(\Oeta)}^2
+
c\int_I  \Vert\nabx \rho \Vert_{L^{2}(\Oeta)}^2\Vert\partial_t\eta  \Vert_{W^{2,2}(\omega)}^2  \dt
\\
&+c(\delta)\int_I \Vert\nabx \rho \Vert_{L^{2}(\Oeta)}^2 \Vert\bu \Vert_{W^{2,2}(\Oeta)}^2  \dt
+
\delta
\int_I\Vert  \partial_t\rho \Vert_{L^{2}(\Oeta)}^2
\end{aligned}
\end{equation}
for any $\delta>0$. Note the estimate for the boundary term done earlier in \eqref{spaceRegRho0}.
By using \eqref{spaceRegRho1}, it follows from \eqref{timeRegRho0} that
\begin{equation}
\begin{aligned}
\label{timeRegRho1}
\int_I\Vert \partial_t \rho\Vert_{L^2(\Oeta)}^2\dt
&+
\sup_{t\in I} \Vert \nabx \rho(t)\Vert_{L^2(\Oeta)}^2
\\&\lesssim
\Vert \nabx\rho_0\Vert_{L^2(\Omega_{\eta_0})}
+
\int_I\big(\Vert\partial_t\eta  \Vert_{W^{2,2}(\omega)}^2+\Vert\bu\Vert_{W^{2,2}(\Oeta)}^2\big)\dt.
\end{aligned}
\end{equation}
Now, we note that  (compare with the estimate for $J_3$ in \eqref{spaceRegT0})
\begin{align*}
2k\int_I\int_{\Oeta}(\bT - \rho \mathbb{I}):\partial_t\bT\dx\dt
&\leq \delta
\int_I\Vert\partial_t \bT\Vert_{L^2(\Oeta)}^2\dt
+
cT
\big(
\Vert\bT_0\Vert_{L^2(\Omega_{\eta_0})}^2
+
T \Vert\rho_0  \Vert_{L^2(\Omega_{\eta_0})}^2
\big)
\\&\times
\exp\bigg(
c
\int_I
\Vert\nabx \bu\Vert_{L^2(\Oeta)}^2 \dt\bigg)
\end{align*}
and   (compare with the estimate for $J_4$ in \eqref{spaceRegT0})
\begin{align*}
\int_I\int_{\Oeta}
[(\nabx \bu)\bT &+ \bT(\nabx \bu)^\top]:\partial_t\bT\dx\dt
\\&\leq \delta
\int_I\Vert\partial_t \bT\Vert_{L^2(\Oeta)}^2\dt
+
c\int_I \Vert \bu\Vert_{W^{2,2}(\Oeta)}^2\Vert\nabx \bT\Vert_{L^2(\Oeta)}^2
\dt
\\&+
cT
\big(
\Vert\bT_0\Vert_{L^2(\Omega_{\eta_0})}^2
+
T \Vert\rho_0  \Vert_{L^2(\Omega_{\eta_0})}^2
\big) 
\exp\bigg(
c
\int_I
\Vert\nabx \bu\Vert_{L^2(\Oeta)}^2 \dt\bigg).
\end{align*}
Therefore, by testing \eqref{solute} with $\partial_t\bT$, we obtain
\begin{equation}
\begin{aligned}
\label{timeRegT1}
\int_I\Vert \partial_t \bT&\Vert_{L^2(\Oeta)}^2\dt
+
\sup_{t\in I} \Vert \nabx \bT(t)\Vert_{L^2(\Oeta)}^2
\\
\lesssim&
\Vert \nabx\bT_0\Vert_{L^2(\Omega_{\eta_0})}
+
\int_I\big(\Vert\partial_t\eta  \Vert_{W^{2,2}(\omega)}^2+\Vert\bu\Vert_{W^{2,2}(\Oeta)}^2\big)\dt
\\&
+
cT
\big(
\Vert\bT_0\Vert_{L^2(\Omega_{\eta_0})}^2
+
T \Vert\rho_0  \Vert_{L^2(\Omega_{\eta_0})}^2
\big) 
\exp\bigg(
c
\int_I
\Vert\nabx \bu\Vert_{L^2(\Oeta)}^2 \dt\bigg).
\end{aligned}
\end{equation}
If we now combine \eqref{spaceRegRho1}, \eqref{spaceRegT1} \eqref{timeRegRho1} and \eqref{timeRegT1}, we obtain
\begin{equation}
\begin{aligned} 
\int_I\big(\Vert \partial_t \rho\Vert_{L^2(\Oeta)}^2
&+
\Vert \partial_t \bT\Vert_{L^2(\Oeta)}^2
\big)\dt
+
\sup_{t\in I} \big(\Vert \rho(t)\Vert_{W^{1,2}(\Oeta)}^2
+
\Vert \bT(t)\Vert_{W^{1,2}(\Oeta)}^2
\big)
\\&
+\int_I\big(\Vert  \rho\Vert_{W^{2,2}(\Oeta)}
+
\Vert  \bT\Vert_{W^{2,2}(\Oeta)}\big)
\\\lesssim
\Vert  \rho_0&\Vert_{W^{1,2}(\Omega_{\eta_0})}
+
\Vert  \bT_0\Vert_{W^{1,2}(\Omega_{\eta_0})}
\\&+
\int_I\big(\Vert\partial_t\eta  \Vert_{W^{2,2}(\omega)}^2+\Vert\bu\Vert_{W^{2,2}(\Oeta)}^2\big)\dt
\\&
+
T
\big(
\Vert\bT_0\Vert_{L^2(\Omega_{\eta_0})}^2
+
T \Vert\rho_0  \Vert_{L^2(\Omega_{\eta_0})}^2
\big)
\exp\bigg(
c
\int_I
\Vert\nabx \bu\Vert_{L^2(\Oeta)}^2 \dt\bigg).
\label{eq:thm:mainFPxxx}
\end{aligned}
\end{equation}  
Further combining \eqref{acceleration} and \eqref{eq:thm:mainFPxxx} yields the desired estimate. Note that we can use the equation \eqref{shellEQ} to directly obtain a bound for $\int_I\Vert \eta\Vert_{W^{4,2}(\omega)}^2\dt$ in terms of the dataset by virtue of \eqref{acceleration} and \eqref{eq:thm:mainFPxxx}.

 \end{proof}

\section{Weak solutions}
\label{sec:weakSol}
\subsection{Galerkin approximation}
\label{sec:Galerkin}
Our goal in this section is to construct a finite-dimensional approximation of a linearized version of the polymeric fluid-structure system for a given geometric setup.   
The approach follows a fixed-point argument where we decouple the equation for the solute from the solvent-structure system. In this case, the solute system becomes a bilinear system whose finite-dimensional approximation can be constructed using the classical Galerkin method. On the other hand, the problem for the solvent-structure system can follow the construction done in \cite{lengeler2014weak}  where the Galerkin basis on the moving domain is constructed from the Piola transform of the basis of the fixed reference domain.
For completeness, but to avoid repetition, we summarize the construction in what follows. 
\\ 
We consider a  given structure displacement $\zeta\in C(\overline{I}\times \omega)$  with an initial state $\zeta(0,\cdot)=\eta_0$ and a given driving divergence-free velocity field $\bv\in L^2_{\divx}(I\times \mathbb{R}^2)$. To ensure that the pair $(\zeta,\bv)$ are sufficiently smooth so that the subsequent analyses are well defined, we consider their spatial regularisation $(\zeta_\kappa,\bv_\kappa)$\footnote{Here, $f_\kappa:=\mathcal{R}_\kappa f$, where the regularising kernels $(\mathcal{R}_\kappa)_{\kappa>0}$ commutes with $\partial_t$. See \cite{lengeler2014weak} for more details.} and assume that they satisfy the interface condition $ \bv_\kappa\circ\bm{\varphi}_{\zeta_\kappa}=\bn\partial_t \zeta_\kappa$ on $I\times\omega$. 
For $\kappa>0$ fixed, we are going to use a Galerkin approximation to construct a weak solution $(\eta_\kappa, \bu_\kappa, \rho_\kappa,\bT_\kappa)$ of the following linearized  system
\begin{align}
\divx \bu_\kappa=0, \label{divfreek}
\\
\partial_t \rho_\kappa+ (\bv_\kappa\cdot \nabx) \rho_\kappa
= 
\Delx \rho_\kappa 
,\label{rhoEquk}
\\
\partial_t \bu_\kappa  + (\bv_\kappa\cdot \nabx)\bu_\kappa
= 
\delx \bu_\kappa -\nabx p_\kappa
+\bff_\kappa
+
\divx   \bT_\kappa, \label{momEquk}
\\
 \partial_t^2 \eta_\kappa - \partial_t\partial_y^2 \eta_\kappa + \partial_y^4 \eta_\kappa = g_\kappa -  (\mathbb{S}_\kappa\bn_{\zeta_\kappa})\circ \bm{\varphi}_{\zeta_\kappa}\cdot\bn\, \det(\partial_y\bm{\varphi}_{\zeta_\kappa}),\label{shellEquk}
\\
\partial_t \bT_\kappa + (\bv_\kappa\cdot \nabx) \bT_\kappa
=
(\nabx \bu_\kappa)\bT_\kappa + \bT_\kappa(\nabx \bu_\kappa)^\top - 2(\bT_\kappa - \rho_\kappa \mathbb{I})+\Delx \bT_\kappa \label{solutek}
\end{align}
on $I\times\Omega_{\zeta_\kappa}\subset \mathbb R^{1+2}$ where
\begin{align*}
\mathbb{S}_\kappa= (\nabx \bu_\kappa +(\nabx \bu_\kappa)^\top) -p_\kappa\mathbb{I}+ \bT_\kappa
\end{align*}
and
\begin{align}
&\eta_\kappa(0,\cdot)=\eta_{0,\kappa}(\cdot), \qquad\partial_t\eta_\kappa(0,\cdot)=\eta_{\star,\kappa}(\cdot) & \text{in }\omega,
\\
&\bu_\kappa(0,\cdot)=\bu_{0,\kappa}(\cdot) & \text{in }\Omega_{\zeta_{0,\kappa}},
\\
&\rho_\kappa(0,\cdot)=\rho_{0,\kappa}(\cdot),\quad\bT_\kappa(0,\cdot)=\bT_{0,\kappa}(\cdot) &\text{in }\Omega_{\zeta_{0,\kappa}},
\label{initialCondSolvK}
\\
& 
\bn_{\zeta_\kappa}\cdot\nabx\rho_\kappa=0,\qquad
\bn_{\zeta_\kappa}\cdot\nabx\bT_\kappa=0 &\text{on }I\times\partial\Omega_{\zeta_\kappa}.
\label{bddCondSolvK}
\end{align}
We now make precise, the notion of a weak solution $(\eta_\kappa, \bu_\kappa, \rho_\kappa,\bT_\kappa)$ in this linearised setting. 
\begin{definition}
\label{def:weaksolmartFPReg}
Let $(\bff_\kappa, g_\kappa, \rho_{0,\kappa}, \bT_{0,\kappa}, \bu_{0,\kappa}, \eta_{0,\kappa}, \eta_{\star,\kappa})$
be a dataset that satisfies
\begin{equation}
\begin{aligned}
\label{mainDataForAllReg}
&\bff_\kappa \in L^2(I;L^{2}_\mathrm{loc}(\mathbb{R}^2)),
\qquad g\in L^2(I;L^{2}(\omega)),
\\&
\eta_{0,\kappa} \in W^{2,2}(\omega) \text{ with } \Vert \eta_{0,\kappa} \Vert_{L^\infty( \omega)} < L, \quad \eta_{\star,\kappa} \in L^{2}(\omega),
\\&\bu_{0,\kappa} \in L^{2}_{\divx}(\Omega_{\eta_{0,\kappa}} )\text{ is such that }\bu_{0,\kappa} \circ \bm{\varphi}_{\eta_{0,\kappa}} =\eta_{\star,\kappa} \bn \text{ on } \omega,
\\&
\rho_{0,\kappa}\in L^{2}(\Omega_{\eta_{0,\kappa}}), \quad
\bT_{0,\kappa}\in L^{2}(\Omega_{\eta_{0,\kappa}}),
\\&
\rho_{0,\kappa}\geq 0,\,\, \bT_{0,\kappa}>0 \quad \text{a.e. in } \Omega_{\eta_0}.
\end{aligned}
\end{equation}
We call 
$(\eta_\kappa, \bu_\kappa, \rho_\kappa,\bT_\kappa)$
a \textit{weak solution} of   \eqref{divfreek}--\eqref{solutek}   if:
\begin{itemize}
\item[(a)] the following properties 
\begin{align*}
&\eta_\kappa\in  W^{1,\infty}\big(I;L^{2}(\omega)  \big)  \cap W^{1,2}\big(I;W^{1,2}(\omega)  \big) \cap L^{\infty}\big(I;W^{2,2}(\omega)  \big),
\\&
 \Vert\eta_\kappa\Vert_{L^\infty(I\times\omega)}<L,
\\
&\bu_\kappa\in
L^{\infty} \big(I; L^{2}(\Omega_{\zeta_\kappa}) \big)\cap L^2\big(I;W^{1,2}_{\divx}(\Omega_{\zeta_\kappa})  \big),
\\
&
\rho_\kappa \in   L^{\infty}\big(I;L^{2}(\Omega_{\zeta_\kappa})  \big)
\cap 
L^2\big(I;W^{1,2}(\Omega_{\zeta_\kappa})  \big),
\\
&
\bT_\kappa \in   L^{\infty}\big(I;L^{2}(\Omega_{\zeta_\kappa})  \big)
\cap 
L^2\big(I;W^{1,2}(\Omega_{\zeta_\kappa})  \big),
\\&
\rho_\kappa\geq 0,\,\, \bT_\kappa>0 \quad \text{a.e. in } I\times \Omega_{\zeta_\kappa}
\end{align*}
holds;
\item[(b)] for all  $  \psi  \in C^\infty (\overline{I}\times \R^2 )$, we have
\begin{equation}
\begin{aligned} 
\label{weakRhoEqReg}
\int_I  \frac{\mathrm{d}}{\dt}
\int_{\Omega_{\zeta_\kappa}} \rho_\kappa\psi \dx \dt 
&=
\int_I
\int_{\Omega_{\zeta_\kappa}}[\rho_\kappa\partial_t\psi + (\rho_\kappa\bv_\kappa\cdot\nabx) \psi] \dx\dt
\\&-
\int_I\int_{\Omega_{\zeta_\kappa}}\nabx \rho_\kappa \cdot\nabx \psi  \dx\dt;
\end{aligned}
\end{equation}
\item[(c)] for all  $  \mathbb{Y}  \in C^\infty (\overline{I}\times \R^2)$, we have
\begin{equation}
\begin{aligned}
\label{weakFokkerPlanckEqKappa} 
\int_I  \frac{\mathrm{d}}{\dt}
\int_{\Omega_{\zeta_\kappa}} \bT_\kappa:\mathbb{Y} \dx \dt 
&=
\int_I
\int_{\Omega_{\zeta_\kappa}}[\bT_\kappa :\partial_t\mathbb{Y} + \bT_\kappa:(\bv_\kappa\cdot\nabx) \mathbb{Y}] \dx\dt
\\&+
\int_I\int_{\Omega_{\zeta_\kappa}}
[(\nabx \bu_\kappa )\bT_\kappa  + \bT_\kappa (\nabx \bu_\kappa )^\top]:\mathbb{Y} \dx\dt
\\&
-2\int_I\int_{\Omega_{\zeta_\kappa}}(\bT_\kappa :\mathbb{Y}  - \rho_\kappa  \mathrm{tr}(\mathbb{Y} ))\dx\dt
\\&-
\int_I\int_{\Omega_{\zeta_\kappa}}\nabx \bT_\kappa ::\nabx \mathbb{Y}  \dx\dt;
\end{aligned}
\end{equation}
\item[(d)] for all $(\bm{\phi},\phi)  \in C^\infty_{\divx} (\overline{I}\times \R^2)\otimes  C^\infty (\overline{I}\times \omega )$ with $\bm{\phi}(T,\cdot)=0$, $\phi(T,\cdot)=0$ and $\bm{\phi}\circ \bm{\varphi}_{\zeta_\kappa}= \phi \bn$, we have
\begin{equation}
\begin{aligned}
\label{weakFluidStrutReg}
\int_I\frac{\dd}{\dt}\bigg(
\int_{\omega}&  \partial_t \eta_\kappa \, \phi\dy
+
\int_{\Omega_{\zeta_\kappa}}
\bu_\kappa  \cdot  \bm{\phi} \dx
\bigg)dt
\\
&=
\int_I
\int_{\omega} \bigg(
\partial_t\eta_\kappa\partial_t\phi
-
\partial_t\naby\eta_\kappa\naby\phi
+
g_\kappa\phi 
-
 \Dely \phi \Dely\eta_\kappa\bigg)\dy\dt
\\&+
\int_I\int_{\omega} \bigg(\frac{1}{2}\bn_{\zeta_\kappa}\circ\bm{\varphi}_{\zeta_\kappa}\cdot \bn^\top \phi \,  \partial_t\zeta_\kappa \,\partial_t\eta_\kappa  \,
 \det(\naby\bm{\varphi}_{\zeta_\kappa})
  \bigg)\dy\dt
\\
&+
\int_I
 \int_{\Omega_{\zeta_\kappa}}\Big(  \bu_\kappa\cdot \partial_t \bm{\phi} 
-
\frac{1}{2}((\bv_\kappa\cdot\nabx)\bu_\kappa)\cdot  \bm{\phi} \Big) \dx\dt
\\
&+
\int_I
 \int_{\Omega_{\zeta_\kappa}}\Big(  
\frac{1}{2}((\bv_\kappa\cdot\nabx) \bm{\phi}) \cdot \bu_\kappa
-\nabx \bu_\kappa:\nabx  \bm{\phi}
+
\bff_\kappa\cdot \bm{\phi}
\Big) \dx\dt
\\
&-
\int_I
 \int_{\Omega_{\zeta_\kappa}} \mathbb{T}_\kappa :\nabx \bm{\phi}  \dx\dt;
\end{aligned}
\end{equation}
\item[(e)] the energy inequality
\begin{equation}
\begin{aligned}
\sup_{t\in I}&
\bigg(
\int_{\Omega_{\zeta_\kappa}}\mathrm{tr}(\bT_\kappa(t))\dx
+
\Vert \bu_\kappa(t)\Vert_{L^2(\Omega_{\zeta_\kappa})}^2 
+
 \Vert\partial_t \eta_\kappa(t)\Vert_{L^2(\omega)}^2 
+ 
\Vert\partial_y^2\eta_\kappa(t)\Vert_{L^2(\omega)}^2
\bigg)
\\&+
\int_I
\int_{\Omega_{\zeta_\kappa}}\mathrm{tr}(\bT_\kappa)\dx\dt
+\int_I\Vert \nabx \bu_\kappa \Vert_{L^2(\Omega_{\zeta_\kappa})}^2\dt
+\int_I\Vert\partial_t\partial_y \eta_\kappa \Vert_{L^2(\omega)}^2\dt
\\&\lesssim
\int_{\Omega_{\eta_{0,\kappa}}}\mathrm{tr}(\bT_{0,\kappa})\dx
+
\Vert \bu_0\Vert_{L^2(\Omega_{\eta_{0,\kappa}})}^2 
+
 \Vert\eta_{\star,\kappa}\Vert_{L^2(\omega)}^2 
+ 
\Vert\partial_y^2\eta_{0,\kappa}\Vert_{L^2(\omega)}^2
\\&+
T
\int_{\Omega_{\eta_{0,\kappa}}}\rho_{0,\kappa} \dx
+
\int_I\Vert \bff_\kappa\Vert_{L^2(\Omega_{\zeta_\kappa})}^2\dt
+
\int_I\Vert g_\kappa\Vert_{L^2(\omega)}^2\dt
\end{aligned}
\end{equation}
holds.
\item[(f)] In addition, we have
\begin{equation}
\begin{aligned}
\label{strgEstFP3Reg}
\sup_{t\in I}&
\big(\Vert\rho_\kappa(t)\Vert_{L^2(\Omega_{\zeta_\kappa})}^2
+
\Vert\bT_\kappa(t)\Vert_{L^2(\Omega_{\zeta_\kappa})}^2
\big)
\\&\qquad\qquad+\int_I\big(\Vert \rho_\kappa \Vert_{W^{1,2}(\Omega_{\zeta_\kappa})}^2
+
\Vert \bT_\kappa \Vert_{W^{1,2}(\Omega_{\zeta_\kappa})}^2\big)\dt
\\&
\lesssim
\Vert\rho_{0,\kappa} \Vert_{L^2(\Omega_{\eta_0})}^2
+ 
\Vert\bT_{0,\kappa}\Vert_{L^2(\Omega_{\eta_0})}^2
\exp\bigg(
c
\int_I
\Vert\nabx \bu_\kappa\Vert_{L^2(\Omega_{\zeta_\kappa})}^2 \dt\bigg)
\\&\qquad\qquad 
+T \Vert\rho_{0,\kappa}  \Vert_{L^2(\Omega_{\eta_{0}})}^2
\exp\bigg(
c
\int_I
\Vert\nabx \bu_\kappa\Vert_{L^2(\Omega_{\zeta_\kappa})}^2 \dt\bigg).
\end{aligned}
\end{equation}
\end{itemize}
\end{definition} 
Our main result in this section is now given as follows.
\begin{theorem}
\label{thm:linReg}
Let $\kappa>0$ be fixed.
For a dataset $(\bff_\kappa, g_\kappa, \rho_{0,\kappa}, \bT_{0,\kappa}, \bu_{0,\kappa}, \eta_{0,\kappa}, \eta_{\star,\kappa})$
that satisfies \eqref{mainDataForAllReg}, there  exists a 
 weak solution $(\eta_\kappa, \bu_\kappa, \rho_\kappa,\bT_\kappa)$ of   \eqref{divfreek}--\eqref{solutek}.
\end{theorem} 
\begin{remark}
We remark here that the subscript $\kappa$ for the unknowns are meant to emphasis that these are the solutions to the regularised system constructed  from a given regularised dataset.
\end{remark}

We now devote the rest of this section to the proof of Theorem \ref{thm:linReg}.
To obtain $(\eta_\kappa, \bu_\kappa, \rho_\kappa,\bT_\kappa)$, we consider the basis $(\overline{\bm{X}}_n)_{n\in \mathbb{N}}$ and $(\overline{Y}_n)_{n\in \mathbb{N}}$ of $W^{1,2}_{0,\divx}(\Omega)$ and $W^{2,2}(\omega)$ respectively. Then by \cite[Theorem A.3]{lengeler2014weak}, there exists solenoidal vector fields $(\overline{\bm{Y}}_n)_{n\in \mathbb{N}}$ that solve a Stokes system on the fixed reference domain with boundary data $(\bn\overline{Y}_n)_{n\in \mathbb{N}}$. Now, for all $t\in \overline{I}$, we obtain from $(\bm{\overline{X}}_n)_{n\in \mathbb{N}}$, the following basis
\begin{align*}
 \bm{X}_n(t,\cdot):=\mathcal{J}_{\zeta_\kappa(t)}\overline{\bm{X}}_n
\end{align*}
for $W^{1,2}_{0,\divx}(\Omega_{\zeta_\kappa(t)})$ where a vector field $\mathbf{v}$, $\mathcal{J}_{\zeta_\kappa}\bv$ defined by
\begin{align}
\label{piolaTransform} 
\mathcal{J}_{\zeta_\kappa}\bv
=
\big(\nabx  \bm{\Psi}_{\zeta_\kappa}(\mathrm{det}\nabx  \bm{\Psi}_{\zeta_\kappa})^{-1}
\bv
\big)\circ \bm{\Psi}_{\zeta_\kappa}^{-1}
\end{align}
is Piola transform of $\mathbf{v}$  with respect to a mapping $\zeta:\omega \rightarrow \mathbb{R}$. The Piola transform is invertible with inverse
\begin{align}
\label{piolaTransformInverse}
\mathcal{J}_{\zeta_\kappa}^{-1}\bv
=
\big((\nabx  \bm{\Psi}_{\zeta_\kappa})^{-1}(\mathrm{det}\nabx  \bm{\Psi}_{\zeta_\kappa}) 
\bv
\big)\circ \bm{\Psi}_{\zeta_\kappa} 
.
\end{align}
In order to ensure that the basis for the solvent system matches with the basis for the structure at the solvent-structure interface, additionally, we consider the Piola transform of the  solenoidal vector fields $(\overline{\bm{Y}}_n)_{n\in \mathbb{N}}$ by setting
\begin{align*}
\bm{Y}_n(t,\cdot):=\mathcal{J}_{\zeta_\kappa(t)}\overline{\bm{Y}}_n.
\end{align*}
Consequently, if we set
\begin{align*}
Y_n(t,\cdot):=(\det(\partial_y\bm{\varphi}_{\zeta_\kappa(t)}))^{-1}\overline{Y}_k
\end{align*}
we obtain the interface condition
\begin{align*}
Y_n(t,\cdot)\bn=\bm{Y}_n(t,\cdot)\circ\bm{\varphi}_{\zeta_\kappa(t)}.
\end{align*}
Our extended basis for the moving domain will now consist of  the pair $(\bm{\psi}_n,\psi_n)_{n\in\mathbb{N}}$ where
\begin{equation}
\label{basis}
\bm{\psi}_n  = \left\{
  \begin{array}{lr}
    \bm{X}_n & : n \text{ even},\\
    \bm{Y}_n & : n \text{ odd}.
  \end{array}
\right.
\qquad\text{and}\qquad
\psi_n\bn:=\bm{\psi}_n\circ \bm{\varphi}_{\zeta_\kappa(t)}.
\end{equation}
With the basis \eqref{basis} in hand, 
we can use Picard--Lindel\"of theorem to find  functions $\alpha^N_n\in C^1(\overline{I})$, $n,N\in \mathbb{N}$ such that $\bu^N_\kappa:=\alpha^N_n\bm{\psi}_n$ and $\eta^N_\kappa=\int_0^t\alpha^N_n\psi_n\ds+\eta^N_{\kappa,0}$ solves\footnote{The dependence of $\bu^N_\kappa$ and $\eta^N_\kappa$ on $\kappa$ follows from the implicit dependence of $\bm{\psi}_n$ and $\psi_n$ on $\kappa$}
\begin{equation}
\begin{aligned}
\label{galerkinweak1}
\frac{\dd}{\dt}\bigg(
&\int_{\omega}  \partial_t \eta^N_\kappa \, \psi_j\dy
+
\int_{\Omega_{\zeta_\kappa}}
\bu^N_\kappa  \cdot  \bm{\psi}_j \dx
\bigg)
\\
&=
\int_{\omega} \bigg(
\partial_t\eta^N_\kappa\partial_t\psi_j
-
\partial_t\naby\eta^N_\kappa\naby\psi_j+g^N_\kappa\psi_j 
-
 \Dely \psi_j \Dely\eta^N_\kappa\bigg)\dy
\\&+
\int_{\omega} \bigg(\frac{1}{2}\bn_{\zeta_\kappa}\circ\bm{\varphi}_{\zeta_\kappa}\cdot \bn^\top \psi_j \,  \partial_t\zeta_\kappa \,\partial_t\eta^N_\kappa  \,
\det(\naby\bm{\varphi}_{\zeta_\kappa})
  \bigg)\dy
\\
&+
 \int_{\Omega_{\zeta_\kappa}}\Big(  \bu^N_\kappa\cdot \partial_t \bm{\psi}_j 
-
\frac{1}{2}((\bv_\kappa\cdot\nabx)\bu^N_\kappa)\cdot  \bm{\psi}_j \Big) \dx
\\
&+
 \int_{\Omega_{\zeta_\kappa}}\Big(  
\frac{1}{2}((\bv_\kappa\cdot\nabx) \bm{\psi}_j) \cdot \bu^N_\kappa
-\nabx \bu^N_\kappa:\nabx  \bm{\psi}_j  
+
\bff^N_\kappa\cdot \bm{\psi}_j
-\mathbb{T}^N_\kappa :\nabx \bm{\psi}_j \Big) \dx
\end{aligned}
\end{equation}
for all $1\leq j\leq N$ with the
pair $(\rho^N_\kappa,\bT^N_\kappa)$ determined by
\begin{align}  
\partial_t \rho^N_\kappa+ (\bv_\kappa\cdot \nabx) \rho^N_\kappa
= 
\Delx \rho^N_\kappa 
\label{rhoAlone}
, 
\\ 
\partial_t \bT^N_\kappa + (\bv_\kappa\cdot \nabx) \bT^N_\kappa
=
(\nabx \bu^N_\kappa)\bT^N_\kappa + \bT^N_\kappa(\nabx  \bu^N_\kappa)^\top - 2(\bT^N_\kappa - \rho^N_\kappa \mathbb{I})
+ 
\Delx \bT^N_\kappa \label{soluteAlone}
\end{align}
on $I\times \Omega_{\zeta_\kappa} \subset \mathbb R^{1+2}$. Here, 
we complement \eqref{rhoAlone}--\eqref{soluteAlone} with the following initial and boundary conditions
\begin{align}
&\rho^N_{0,\kappa}\geq 0,\quad \bT^N_{0,\kappa}>0 \quad   &\text{a.e in } I\times\Omega_{\zeta_\kappa}
\label{positiveInitSoluteAlone}
\\&\rho^N_\kappa(0,\cdot)=\rho^N_{0,\kappa}(\cdot),\quad\bT^N_\kappa(0,\cdot)=\bT^N_{0,\kappa}(\cdot) &\text{in }\Omega_{\zeta_{0,\kappa}},
\label{initialCondSolvAlone} 
\\ 
& 
\bn_{\zeta_\kappa}\cdot\nabx\rho^N_\kappa=0 &\text{on }I\times\partial\Omega_{\zeta_\kappa},
\label{bddCondSolvRhoAlone}
\\&
\bn_{\zeta_\kappa}\cdot\nabx\bT^N_\kappa=0
 &\text{on }I\times\partial\Omega_{\zeta_\kappa}.
\label{bddCondSolvTAlone}
\end{align}
Indeed, for a fixed $N\in \mathbb{N}$, suppose that $(\overline{\eta^N_\kappa},\overline{\bu^N_\kappa})$ satisfying
\begin{align*}
&\overline{\eta^N_\kappa}\in  W^{1,\infty}\big(I;L^{2}(\omega)  \big)  \cap W^{1,2}\big(I;W^{1,2}(\omega)  \big) \cap L^{\infty}\big(I;W^{2,2}(\omega)  \big),
\\&
\overline{\bu^N_\kappa}\in
L^{\infty} \big(I; L^{2}(\Omega_{\zeta_\kappa}) \big)\cap L^2\big(I;W^{1,2}_{\divx}(\Omega_{\zeta_\kappa})  \big),
\\&
\overline{\bu^N_\kappa}\circ\bm{\varphi}_{\zeta_\kappa}=\bn\partial_t \overline{\eta^N_\kappa} \text{ on } I\times\omega
\end{align*} 
is given. Then due to the bilinearity of \eqref{rhoAlone}-\eqref{soluteAlone},  a solution pair  $(\rho^N_\kappa,\bT^N_\kappa)$ of \eqref{rhoAlone}--\eqref{bddCondSolvTAlone} satisfying
\begin{equation}
\begin{aligned} 
\label{weakRhoNKappa}
\int_I  \frac{\mathrm{d}}{\dt}
\int_{\Omega_{\zeta_\kappa} } \rho^N_\kappa\phi_i \dx \dt 
&=
\int_I
\int_{\Omega_{\zeta_\kappa}}[\rho^N_\kappa\partial_t\phi_i + (\rho^N_\kappa\bv_\kappa\cdot\nabx) \phi_i] \dx\dt
\\&-
\int_I\int_{\Omega_{\zeta_\kappa}}\nabx \rho^N_\kappa \cdot\nabx \phi_i  \dx\dt
\end{aligned}
\end{equation}
and   
\begin{equation}
\begin{aligned}
\label{weakFokkerPlanckNKappa}
\int_I  \frac{\mathrm{d}}{\dt}
\int_{\Omega_{\zeta_\kappa} } \bT^N_\kappa:\mathbb{Y}_i \dx \dt 
&=
\int_I
\int_{\Omega_{\zeta_\kappa} }[\bT^N_\kappa :\partial_t\mathbb{Y}_i + \bT^N_\kappa:(\bv_\kappa\cdot\nabx) \mathbb{Y}_i] \dx\dt
\\&+
\int_I\int_{\Omega_{\zeta_\kappa}}
[(\nabx \overline{\bu^N_\kappa} )\bT^N_\kappa  + \bT^N_\kappa (\nabx \overline{\bu^N_\kappa} )^\top]:\mathbb{Y}_i \dx\dt
\\&
-2\int_I\int_{\Omega_{\zeta_\kappa} }(\bT^N_\kappa :\mathbb{Y}_i  - \rho^N_\kappa  \mathrm{tr}(\mathbb{Y}_i ))\dx\dt
\\&-
\int_I\int_{\Omega_{\zeta_\kappa}}\nabx \bT^N_\kappa ::\nabx \mathbb{Y}_i  \dx\dt
\end{aligned}
\end{equation}
for all $1\leq i\leq N$ is obtained as a limit $M\rightarrow\infty$ of, yet again, a Galerkin approximation $(\rho^{N,M}_\kappa,\bT^{N,M}_\kappa)$ for a basis $(\phi_i,\mathbb{Y}_i )$ in $W^{1,2}_0(\Omega_{\zeta_\kappa})\otimes W^{1,2}_0(\Omega_{\zeta_\kappa})$. Furthermore,   if we take the trace in \eqref{soluteAlone}, then similar to  \eqref{freeEnergyEst}, we obtain
\begin{equation}
\begin{aligned}
\label{freeEnergyEstGar}
\int_{\Omega_{\zeta_\kappa}} \mathrm{tr}(\bT^N_\kappa(t))\dx
+&
  2
  \int_0^t
  \int_{\Omega_{\zeta_\kappa}}
  \mathrm{tr}(\bT^N_\kappa)\dx\dt' 
\\\leq&
\int_0^t
 \Vert\bT^N_\kappa\Vert_{L^2(\Omega_{\zeta_\kappa})}^2 \dt'
+
\int_0^t
\Vert\nabx\overline{\bu^N_\kappa}\Vert_{L^2(\Omega_{\zeta_\kappa})}^2 \dt'
\\&
+
6
\int_{\Omega_{\eta_{0,\kappa}}}\rho^N_{\kappa,0} \dx
+ 
\int_{\Omega_{\eta_{0,\kappa}}}\mathrm{tr}(\bT^N_{\kappa,0})\dx
\end{aligned}
\end{equation}
for all $t\in I$. On the other hand, if we take the inner product of   \eqref{soluteAlone} with $\bT^N_\kappa$ and apply Gr\"onwall's lemma,  we obtain
\begin{equation}
\begin{aligned}
\label{strgEstFP1Alone}
\sup_{t\in I}\Vert\bT^N_\kappa(t)&\Vert_{L^2(\Omega_{\zeta_\kappa})}^2
+
\int_I\Vert \bT^N_\kappa \Vert_{W^{1,2}(\Omega_{\zeta_\kappa})}^2\dt 
\\&
\lesssim
\big(
\Vert\bT^N_{0,\kappa}\Vert_{L^2(\Omega_{\eta_{0,\kappa}})}^2
+
T
 \Vert\rho^N_{0,\kappa} \Vert_{L^2(\Omega_{\eta_{0,\kappa}})}^2
 \big)
 \exp\bigg(c
\int_I
\Vert\nabx \overline{\bu^N_\kappa}\Vert_{L^2(\Omega_{\zeta_\kappa})}^2 
\dt\bigg)
\end{aligned}
\end{equation}
with a similar but simpler estimate holding for $\rho^N_\kappa$.
If we now combine  \eqref{freeEnergyEstGar} and \eqref{strgEstFP1Alone}, we obtain
\begin{equation}
\begin{aligned}
\label{eneryEstSoluteAlone}
\sup_{t\in I}
\int_{\Omega_{\zeta_\kappa}}\mathrm{tr}(\bT^N_\kappa(t))\dx
&
+
\sup_{t\in I}
\Vert\bT^N_\kappa(t)\Vert_{L^2(\Omega_{\zeta_\kappa})}^2
+
\int_I
\int_{\Omega_{\zeta_\kappa}}\mathrm{tr}(\bT^N_\kappa)\dx\dt
\\&+
\int_I\Vert  \bT^N_\kappa \Vert_{W^{1,2}(\Omega_{\zeta_\kappa})}^2\dt 
  \lesssim
\mathcal{D}(\overline{\bu^N_\kappa},\rho^N_{0,\kappa},\bT^N_{0,\kappa})
\end{aligned}
\end{equation}
where
\begin{align*}
\mathcal{D}(\overline{\bu^N_\kappa},\rho^N_{0,\kappa},\bT^N_{0,\kappa})&:=
\int_I
\Vert\nabx\overline{\bu^N_\kappa}\Vert_{L^2(\Omega_{\zeta_\kappa})}^2 \dt
+
\int_{\Omega_{\eta_{0,\kappa}}}\rho^N_{0,\kappa} \dx
+ 
\int_{\Omega_{\eta_{0,\kappa}}}\mathrm{tr}(\bT^N_{0,\kappa})\dx
\\&
+
\big(
\Vert\bT^N_{0,\kappa}\Vert_{L^2(\Omega_{\eta_{0,\kappa}})}^2
+
T
 \Vert\rho^N_{0,\kappa} \Vert_{L^2(\Omega_{\eta_{0,\kappa}})}^2
 \big)
 \exp\bigg(c
\int_I
\Vert\nabx \overline{\bu^N_\kappa}\Vert_{L^2(\Omega_{\zeta_\kappa})}^2 
\dt\bigg).
\end{align*} 
The right-hand side of \eqref{eneryEstSoluteAlone} is finite and thus, we have constructed a solution
 $(\rho^N_\kappa,\bT^N_\kappa)$ of \eqref{rhoAlone}--\eqref{bddCondSolvTAlone} satisfying
\begin{align*}
&
\rho^N_\kappa \in   L^{\infty}\big(I;L^{2}(\Omega_{\zeta_\kappa})  \big)
\cap 
L^2\big(I;W^{1,2}(\Omega_{\zeta_\kappa})  \big),
\\
&
\bT^N_\kappa \in   L^{\infty}\big(I;L^{2}(\Omega_{\zeta_\kappa})  \big)
\cap 
L^2\big(I;W^{1,2}(\Omega_{\zeta_\kappa})  \big).
\end{align*}
Now, with the constructed pair  $(\rho^N_\kappa,\bT^N_\kappa)$ in hand, we revisit \eqref{galerkinweak1}. As stated earlier, its solution follows from Picard--Lindel\"of theorem.
Furthermore, by taking the pair $(\eta^N_\kappa, \bu^N_\kappa)$ as test functions, we obtain  the global bound  
\begin{equation}
\begin{aligned} 
\label{eneryEstSolvStruAlone} 
\sup_{t\in I}\Big(\Vert \bu^N_\kappa\Vert_{L^2(\Omega_{\zeta_\kappa})}^2 
&+
\Vert\partial_t \eta^N_\kappa\Vert_{L^2(\omega)}^2 
+  
\Vert\partial_y^2\eta^N_\kappa\Vert_{L^2(\omega)}^2
\Big)
\\&+
 \int_I\Big(\Vert \nabx \bu^N_\kappa\Vert_{L^2(\Omega_{\zeta_\kappa})}^2+  \Vert\partial_t\partial_y \eta^N_\kappa\Vert_{L^2(\omega)}^2\Big)\dt
\\
\lesssim&\int_I\Vert\bT^N_\kappa \Vert_{L^2(\Omega_{\zeta_\kappa})}^2\dt
+
\int_I\Vert\bff^N_\kappa \Vert_{L^2(\Omega_{\zeta_\kappa})}^2\dt
+
\int_I \Vert g^N_\kappa\Vert_{L^2(\Omega_{\zeta_\kappa})}^2 \dt
\\&+
\Vert \bu^N_{0,\kappa}\Vert_{L^2(\Omega_{\eta_{0,\kappa}})}^2 
+
\Vert \eta^N_{\star,\kappa}\Vert_{L^2(\omega)}^2 
+  
\Vert\partial_y^2\eta^N_{0,\kappa}\Vert_{L^2(\omega)}^2
\\
\lesssim& 1
\end{aligned}
\end{equation}
leading to $(\eta^N_\kappa,\bu^N_\kappa) \in X^I$ where
\begin{align*}
X^I
:=&
\big(
W^{1,\infty}(I;L^2(\omega))\cap W^{1,2}(I;W^{1,2}(\omega))\cap L^{\infty}(I;W^{2,2}(\omega))
\big)
\\&\otimes
\big(L^{\infty} \big(I; L^{2}(\Omega_{\zeta_\kappa}) \big)\cap L^2\big(I;W^{1,2}_{\divx}(\Omega_{\zeta_\kappa})  \big)
\big).
\end{align*}
At this point, on the one hand, we have obtained a solution  $(\rho^N_\kappa,\bT^N_\kappa)$ to  the solute system \eqref{rhoAlone}--\eqref{bddCondSolvTAlone}  given a solvent-structure pair $(\overline{\eta^N_\kappa},\overline{\bu^N_\kappa})$. On the other hand, we have also constructed a solvent-structure pair $(\eta^N_\kappa,\bu^N_\kappa)$ given a solute pair  $(\rho^N_\kappa,\bT^N_\kappa)$.
We can now use a fixed point argument to get a local solution $(\eta^N_\kappa, \bu^N_\kappa, \rho^N_\kappa,\bT^N_\kappa)$ for the mutually coupled system.
To do this, for a time $T_*>0$, $I_*:=(0,T_*)$ to be determined soon, we consider the solution map $\mathtt{T}=\mathtt{T}_1\circ \mathtt{T}_2: X^{I_*} \rightarrow  X^{I_*}$ where
\begin{align*}
\mathtt{T}(\overline{\eta^N_\kappa},\overline{\bu^N_\kappa})=(\eta^N_\kappa,\bu^N_\kappa), \quad \mathtt{T}_2(\overline{\eta^N_\kappa},\overline{\bu^N_\kappa})=(\rho^N_\kappa,\bT^N_\kappa), \quad \mathtt{T}_1(\rho^N_\kappa,\bT^N_\kappa)=(\eta^N_\kappa,\bu^N_\kappa)
\end{align*}
and the associated set  
\begin{align*} 
B_R^{I_*}:=\Big\{ (\overline{\eta^N_\kappa},\overline{\bu^N_\kappa}) \in X^{I_*}&\text{ with } \overline{\bu^N_\kappa}:=\overline{\alpha}^N_n\bm{\psi}_n,\, \overline{\eta^N_\kappa}:=\int_0^t\overline{\alpha}^N_n\psi_n\ds+\eta^N_{\kappa,0}
\\&\qquad
\text{ such that } \Vert (\overline{\eta^N_\kappa},\overline{\bu^N_\kappa}) \Vert_{X^{I_*}}^2  \leq R^2\text{ for }t\in \overline{I}_*\Big\}
\end{align*}
for $R>0$ large enough and for fixed $\kappa>0$.
From  \eqref{eneryEstSoluteAlone} and \eqref{eneryEstSolvStruAlone}, it follows that  $\mathtt{T}:B_R^{I_*}     \rightarrow B_R^{I_*}$ for $R>0$ large and $T_*>0$ small, i.e. the solution mapping  $\mathtt{T}$ maps the ball $B_R^{I_*} $ into itself. 
%
To show the contraction property leading to the existence of a unique local solution for the fully coupled system, we consider any two solution pair $(\eta^{N,i}_\kappa,\bu^{N,i}_\kappa)$, $i=1,2$ of the solvent-structure system with datasets 
$(\bff^N_\kappa, g^N_\kappa, \bu^N_{0,\kappa}, \eta^N_{0,\kappa}, \eta^N_{\star,\kappa}, \rho^{N,i}_{\kappa}, \bT^{N,i}_{\kappa})$, $i=1,2$, respectively.
Thus, the difference 
\begin{align*}
(\eta^{N,12}_\kappa\,,\, \bu^{N,12}_\kappa)
:=
(\eta^{N,1}_\kappa-\eta^{N,2}_\kappa\, ,\,\bu^{N,1}_\kappa - \bu^{N,2}_\kappa)
\end{align*}
solves  
\begin{align}
\partial_t^2 \eta^{N,12}_\kappa - \partial_t\partial_y^2 \eta^{N,12}_\kappa + \partial_y^4 \eta^{N,12}_\kappa = - \bn^\top \mathbb{S}^{N,12}_\kappa\circ \bm{\varphi}_{\zeta_\kappa}\bn_{\zeta_\kappa} \det(\partial_y\bm{\varphi}_{\zeta_\kappa})
\end{align}
in $I_* \times \omega$ and 
\begin{equation}
\begin{aligned}
\partial_t \bu^{N,12}_\kappa  + (\bv_\kappa\cdot \nabx)\bu^{N,12}_\kappa
= 
\delx \bu^{N,12}_\kappa -\nabx p^{N,12}_\kappa 
+
\divx   \bT^{N,12}_\kappa
\end{aligned}
\end{equation}
in $I_* \times  \Omega_{\zeta_\kappa}$
where 
\begin{align*}
\mathbb{S}^{N,12}_\kappa=&(\nabx\bu^{N,12}_\kappa + \nabx(\bu^{N,12}_\kappa)^\top)-p^{N,12}_\kappa\mathbb{I}
+\bT^{N,12}_\kappa
\end{align*} 
with $p^{N,12}_\kappa:=p^{N,1}_\kappa-p^{N,2}_\kappa$ and $\bT^{N,12}_\kappa:=\bT^{N,1}_\kappa-\bT^{N,2}_\kappa$.
Similar to \eqref{enerFormalFluiStru}, we obtain
\begin{equation}
\begin{aligned} 
\label{energyEstAxxY} 
\sup_{t\in I_*} &\Big(\Vert \bu^{N,12}_\kappa(t)\Vert_{L^2(\Omega_{\zeta_\kappa})}^2 
+
\Vert\partial_t \eta^{N,12}_\kappa(t)\Vert_{L^2(\omega)}^2 
+ 
\Vert\partial_y^2\eta^{N,12}_\kappa(t)\Vert_{L^2(\omega)}^2
\Big) 
\\&+
 \int_{I_*}\Big(\Vert \nabx \bu^{N,12}_\kappa\Vert_{L^2(\Omega_{\zeta_\kappa})}^2+\Vert\partial_t\partial_y \eta^{N,12}_\kappa\Vert_{L^2(\omega)}^2\Big)\dt
 \\\lesssim&
T_*\sup_{t\in I_*}\Vert\bT^{N,12}_\kappa(t)\Vert_{L^2(\Omega_{\zeta_\kappa})}^2.
\end{aligned}
\end{equation} 
To obtain a contraction estimate for $\bT^{N}_\kappa$,
let us consider any two solution $(\rho^{N,i}_\kappa,\bT^{N,i}_\kappa)$, $i=1,2$ of the \eqref{rhoAlone}-\eqref{soluteAlone} with datasets $(\rho^N_{0,\kappa}, \bT^N_{0,\kappa},\overline{\eta^{N,i}_\kappa},\overline{\bu^{N,i}_\kappa})$, $i=1,2$, respectively, so that the difference $\rho^{N,12}_\kappa:=\rho^{N,1}_\kappa-\rho^{N,2}_\kappa$ and $\bT^{N,12}_\kappa:=\bT^{N,1}_\kappa-\bT^{N,2}_\kappa$ solves
\begin{align}  
\partial_t \rho^{N,12}_\kappa+ (\bv_\kappa\cdot \nabx) \rho^{N,12}_\kappa
= 
\Delx \rho^{N,12}_\kappa 
\label{rhoAloneDiff}
\end{align}
and
\begin{equation}
\begin{aligned}
\partial_t \bT^{N,12}_\kappa + (\bv_\kappa \cdot \nabx) \bT^{N,12}_\kappa
&=
(\nabx \overline{\bu^{N,12}_\kappa})\bT^{N,1}_\kappa + \bT^{N,1}_\kappa(\nabx \overline{\bu^{N,12}_\kappa})^\top 
\\&- 2(\bT^{N,12}_\kappa - \rho^{N,12}_\kappa \mathbb{I})+\Delx \bT^{N,12}_\kappa
\\&
+(\nabx \overline{\bu^{N,2}_\kappa})\bT^{N,12}_\kappa + \bT^{N,12}_\kappa(\nabx \overline{\bu^{N,2}_\kappa})^\top 
 \label{soluteAloneDiff}
\end{aligned} 
\end{equation}
in $I_* \times  \Omega_{\zeta_\kappa}$.  
We now  test \eqref{rhoAloneDiff} with $\rho^{N,12}_\kappa$ and use the boundary condition \eqref{bddCondSolvRhoAlone} which leads to
\begin{equation}
\begin{aligned}  
\label{weakRhoEqxy}
\sup_{t\in I_*} 
\Vert\rho^{N,12}_\kappa(t)\Vert_{L^2(\Omega_{\zeta_\kappa})}^2 
+
\int_{I_*} \Vert \nabx\rho^{N,12}_\kappa \Vert_{L^{2}(\Omega_{\zeta_\kappa})}^2   \dt
=0
\end{aligned}
\end{equation}
Next we test \eqref{soluteAloneDiff} with $\bT^{N,12}_\kappa$. 
We obtain
\begin{equation}  
\begin{aligned}
\label{strgEstFP1x}
\Vert\bT^{N,12}_\kappa(t)&\Vert_{L^2(\Omega_{\zeta_\kappa})}^2
+
\int_0^t \Vert\bT^{N,12}_\kappa\Vert_{W^{1,2}(\Omega_{\zeta_\kappa})}^2\dt'
\\
\lesssim& 
T_* \Vert\rho^{N,12}_\kappa(t) \Vert_{L^2(\Omega_{\zeta_\kappa})}^2
+
\int_0^t \Vert\bT^{N,12}_\kappa\Vert_{L^2(\Omega_{\zeta_\kappa})}^2\dt'
\\&
+\int_0^t\int_{\Omega_{\zeta_\kappa}}\big(\vert \nabx \overline{\bu^{N,12}_\kappa}\vert\,\vert\bT^{N,1}_\kappa \vert
+
\vert \nabx \overline{\bu^{N,2}_\kappa}\vert\,\vert\bT^{N,12}_\kappa \vert
\big) \,\vert\bT^{N,12}_\kappa \vert\dx\dt'
\end{aligned}
\end{equation}
for all $t\in I_*$ whereas by using Ladyzhenskaya inequality, we find that the estimate
\begin{align*}
\int_0^t\int_{\Omega_{\zeta_\kappa}}& \vert \nabx \overline{\bu^{N,12}_\kappa}\vert\,\vert\bT^{N,1}_\kappa \vert
\,\vert\bT^{N,12}_\kappa \vert\dx\dt'
\\&
\lesssim
\int_0^t\Vert\bT^{N,1}_\kappa \Vert_{L^4(\Omega_{\zeta_\kappa})} 
\Vert\bT^{N,12}_\kappa \Vert_{L^4(\Omega_{\zeta_\kappa})}
\Vert \nabx \overline{\bu^{N,12}_\kappa}\Vert_{L^2(\Omega_{\zeta_\kappa})}
\dt'
\\&
\leq \delta 
\int_0^t
\Vert  \bT^{N,12}_\kappa \Vert_{W^{1,2}(\Omega_{\zeta_\kappa})}^2\dt'
+
\delta
\int_0^t 
\Vert\bT^{N,1}_\kappa \Vert_{L^2(\Omega_{\zeta_\kappa})}
\Vert \nabx \overline{\bu^{N,12}_\kappa}\Vert_{L^2(\Omega_{\zeta_\kappa})}^2\dt
\\&
+
c
\int_0^t
\Vert \bT^{N,1}_\kappa \Vert_{W^{1,2}(\Omega_{\zeta_\kappa})}^2
\Vert\bT^{N,12}_\kappa \Vert_{L^2(\Omega_{\zeta_\kappa})}^2
\dt'
\end{align*}
holds for any $\delta>0$ and
\begin{align*}
\int_0^t\int_{\Omega_{\zeta_\kappa}}& \vert \nabx \overline{\bu^{N,2}_\kappa}\vert
\,\vert\bT^{N,12}_\kappa \vert^2 \dx\dt'
\lesssim
\int_0^t \Vert \nabx \overline{\bu^{N,2}_\kappa}\Vert_{L^2(\Omega_{\zeta_\kappa})}
\Vert \bT^{N,12}_\kappa\Vert_{L^4(\Omega_{\zeta_\kappa})}^2\dt'
\\&
\leq \delta
\int_0^t 
\Vert   \bT^{N,12}_\kappa\Vert_{W^{1,2}(\Omega_{\zeta_\kappa})}^2 \dt'
+
c\int_0^t
\Vert \nabx \overline{\bu^{N,2}_\kappa}\Vert_{L^2(\Omega_{\zeta_\kappa})}^2
\Vert \bT^{N,12}_\kappa\Vert_{L^2(\Omega_{\zeta_\kappa})}^2\dt'.
\end{align*} 
If we now use the regularity of $ \bT^{N,1}_\kappa , \overline{\bu^{N,2}_\kappa}$
 as well as \eqref{weakRhoEqxy}, we obtain from Gr\"onwall's lemma that
\begin{equation}
\begin{aligned}
\label{strgEstFP1xy}
\sup_{t\in I_*}\Vert\bT^{N,12}_\kappa(t)\Vert_{L^2(\Omega_{\zeta_\kappa})}^2
& 
+
\int_{I_*} \Vert\bT^{N,12}_\kappa\Vert_{W^{1,2}(\Omega_{\zeta_\kappa})}^2\dt 
\\&
\leq  
\delta
e^{cT_*}
\int_{I_*}
  \Vert\nabx \overline{\bu^{N,12}_\kappa}\Vert_{L^2(\Omega_{\zeta_\kappa})}^2
 \dt.
\end{aligned}
\end{equation}
By combining \eqref{energyEstAxxY} and \eqref{strgEstFP1xy}, we obtain
\begin{align*}
\Vert \mathtt{T}(\overline{\eta^N_\kappa},\overline{\bu^N_\kappa}) \Vert_{X^{I_*}}^2 =
\Vert (\eta^N_\kappa, \bu^N_\kappa) \Vert_{X^{I_*}}^2 
\leq
c\delta T_*
e^{cT_*}\Vert (\overline{\eta^N_\kappa},\overline{\bu^N_\kappa}) \Vert_{X^{I_*}}^2 
\leq
\tfrac{1}{2}\Vert (\overline{\eta^N_\kappa},\overline{\bu^N_\kappa}) \Vert_{X^{I_*}}^2 
\end{align*}
for the choice of $T_*$ such that $2T_*e^{cT_*}\leq(c
\delta)^{-1}$. This completes the proof of the existence of a unique local-in-time weak solution for the fully coupled finite-dimensional system \eqref{galerkinweak1}-\eqref{soluteAlone}. The fact that this solution is global follows from the energy estimate. Similar to  \eqref{apriori1}, we obtain 
\begin{equation}
\begin{aligned} 
\label{apriori1xReg}
\sup_{t\in I}&
\bigg(
\int_{\Omega_{\zeta_\kappa}}\mathrm{tr}(\bT^N_\kappa(t))\dx
+
\Vert \bu^N_\kappa(t)\Vert_{L^2(\Omega_{\zeta_\kappa})}^2 
+
 \Vert\partial_t \eta^N_\kappa(t)\Vert_{L^2(\omega)}^2 
+ 
\Vert\partial_y^2\eta^N_\kappa(t)\Vert_{L^2(\omega)}^2
\bigg)
\\&+
\int_I
\int_{\Omega_{\zeta_\kappa}}\mathrm{tr}(\bT^N_\kappa)\dx\dt
+
\int_I\Vert \nabx \bu^N_\kappa \Vert_{L^2(\Omega_{\zeta_\kappa})}^2\dt
+
\int_I\Vert\partial_t\partial_y \eta^N_\kappa \Vert_{L^2(\omega)}^2\dt
\\&\lesssim
\int_{\Omega_{\zeta_\kappa(0)}}\mathrm{tr}(\bT^N_{0,\kappa})\dx
+
\Vert \bu^N_{0,\kappa} \Vert_{L^2(\Omega_{\zeta_\kappa(0)})}^2 
+
 \Vert\eta^N_{\star,\kappa}\Vert_{L^2(\omega)}^2 
+ 
\Vert\partial_y^2\eta^N_{0,\kappa}\Vert_{L^2(\omega)}^2
\\&+
T
\int_{\Omega_{\zeta_\kappa(0)}}\rho^N_{0,\kappa} \dx
+
\int_I\Vert \bff^N_\kappa\Vert_{L^2(\Omega_{\zeta_\kappa})}^2\dt
+
\int_I\Vert g^N_\kappa\Vert_{L^2(\omega)}^2\dt
\end{aligned}
\end{equation}
and similar to \eqref{strgEstFP3}, we obtain
\begin{equation}
\begin{aligned} 
\label{strgEstFP3xReg}
\sup_{t\in I}&
\big(\Vert\rho^N_\kappa(t)\Vert_{L^2(\Omega_{\zeta_\kappa})}^2
+
\Vert\bT^N_\kappa(t)\Vert_{L^2(\Omega_{\zeta_\kappa})}^2
\big)
\\&+
\int_I\big(\Vert  \rho^N_\kappa \Vert_{W^{1,2}(\Omega_{\zeta_\kappa})}^2
+
\Vert  \bT^N_\kappa \Vert_{W^{1,2}(\Omega_{\zeta_\kappa})}^2\big)\dt 
\\
\lesssim&
\Vert\rho^N_{0,\kappa} \Vert_{L^2(\Omega_{\zeta_\kappa(0)})}^2
+ 
\Vert\bT^N_{0,\kappa}\Vert_{L^2(\Omega_{\zeta_\kappa(0)})}^2\exp\bigg(
c
\int_I
\Vert\nabx \bu^N_\kappa\Vert_{L^2(\Omega_{\zeta_\kappa})}^2 \dt\bigg)
\\&+
T \Vert\rho^N_{0,\kappa}  \Vert_{L^2(\Omega_{\zeta_\kappa(0)})}^2 
\exp\bigg(
c
\int_I
\Vert\nabx \bu^N_\kappa\Vert_{L^2(\Omega_{\zeta_\kappa})}^2 \dt\bigg).
\end{aligned}
\end{equation}
By using the boundedness of the initial conditions, it follows from \eqref{apriori1xReg} and  \eqref{strgEstFP3xReg} that
\begin{align*}
\eta^N_\kappa \rightarrow \eta_\kappa
&\qquad\text{in}\qquad  \big(L^{\infty}(I;W^{2,2}(\omega)), w^*  \big),
\\
\partial_t\eta^N_\kappa \rightarrow \partial_t\eta_\kappa
&\qquad\text{in}\qquad
\big(L^{\infty}(I;L^{2}(\omega)),w^*  \big)  \cap \big(L^{2}(I;W^{1,2}(\omega)),w  \big),
\\
\bu^N_\kappa \rightarrow \bu_\kappa
&\qquad\text{in}\qquad
\big(L^{\infty} \big(I; L^{2}(\Omega_{\zeta_\kappa}) \big),w^*\big)\cap \big(L^2\big(I;W^{1,2}_{\divx}(\Omega_{\zeta_\kappa})  \big), w\big),
\\
\rho^N_\kappa \rightarrow \rho_\kappa
&\qquad\text{in}\qquad
\big(L^{\infty} \big(I; L^{2}(\Omega_{\zeta_\kappa}) \big),w^*\big)\cap \big(L^2\big(I;W^{1,2}(\Omega_{\zeta_\kappa})  \big), w\big),
\\
\bT^N_\kappa \rightarrow \bT_\kappa
&\qquad\text{in}\qquad
\big(L^{\infty} \big(I; L^{2}(\Omega_{\zeta_\kappa}) \big),w^*\big)\cap \big(L^2\big(I;W^{1,2}(\Omega_{\zeta_\kappa})  \big), w\big).
\end{align*}
Furthermore, by using a density argument and lower semi-continuity of norms, the convergences above offers all the ingredients to pass to the limit in \eqref{galerkinweak1}, \eqref{weakRhoNKappa}, \eqref{weakFokkerPlanckNKappa}, \eqref{apriori1xReg} and \eqref{strgEstFP3xReg}  to complete the proof of Theorem \ref{thm:linReg}.

\subsection{The regularised fully-coupled system}
In the previous section, we constructed a weak solution $(\eta_\kappa, \bu_\kappa, \rho_\kappa,\bT_\kappa)$ to the linearised system posed on the given regularised spacetime geometry $I\times \Omega_{\zeta_\kappa}$. In this section, we are going to use a fixed-point argument to show that $(\eta_\kappa, \bu_\kappa, \rho_\kappa,\bT_\kappa)$ actually solves the fully-coupled nonlinear system posed on the unknown regularized geometry, i.e.
\begin{align}
\divx \bu_\kappa=0, \label{divfreekNon}
\\
\partial_t \rho_\kappa+ (\bu_\kappa\cdot \nabx) \rho_\kappa
= 
\Delx \rho_\kappa 
,\label{rhoEqukNon}
\\
\partial_t \bu_\kappa  + (\bu_\kappa\cdot \nabx)\bu_\kappa
= 
\delx \bu_\kappa -\nabx p_\kappa
+\bff_\kappa
+
\divx   \bT_\kappa, \label{momEqukNon}
\\
 \partial_t^2 \eta_\kappa -\partial_t\partial_y^2 \eta_\kappa +  \partial_y^4 \eta_\kappa = g_\kappa - ( \mathbb{S}_\kappa\bn_{\eta_\kappa})\circ \bm{\varphi}_{\eta_\kappa}\cdot \bn \,\det(\partial_y\bm{\varphi}_{\eta_\kappa}) , \label{shellEqukNon}
\\
\partial_t \bT_\kappa + (\mathbf{u}_\kappa\cdot \nabx) \bT_\kappa
=
(\nabx \bu_\kappa)\bT_\kappa + \bT_\kappa(\nabx \bu_\kappa)^\top - 2(\bT_\kappa - \rho_\kappa \mathbb{I})+\Delx \bT_\kappa \label{solutekNon}
\end{align}
on $I\times\Omega_{\eta_\kappa}\subset \mathbb R^{1+2}$ where
\begin{align*}
\mathbb{S}_\kappa=  (\nabx \bu_\kappa +(\nabx \bu_\kappa)^\top) -p_\kappa\mathbb{I}+  \bT_\kappa.
\end{align*}
A weak solution of \eqref{divfreekNon}-\eqref{solutekNon} is defined in analogy to Definition \ref{def:weaksolmartFPReg}. Our main result now reads. 
 \begin{theorem}
\label{thm:linRegNon}
Let $\kappa>0$ be arbitrary.
For a dataset $(\bff_\kappa, g_\kappa, \rho_{0,\kappa}, \bT_{0,\kappa}, \bu_{0,\kappa}, \eta_{0,\kappa}, \eta_{\star,\kappa})$
that satisfies \eqref{mainDataForAllReg}, there  exists a  
weak solution $(\eta_\kappa, \bu_\kappa, \rho_\kappa,\bT_\kappa)$ of   \eqref{divfreekNon}--\eqref{solutekNon}.
\end{theorem} 
\begin{proof}
We note that the only differences between \eqref{divfreek}-\eqref{solutek} and the anticipated system \eqref{divfreekNon}-\eqref{solutekNon} consists of the linearisation
by the given velocity $\bv_\kappa$ (rather than $\bu_\kappa$) in the advection term, the stress tensor term on the right-hand side of \eqref{shellEquk} being transformed by a coordinate transform with respect of $\zeta_\kappa$ (rather than $\eta_\kappa$), and finally the full system posed on $\Omega_{\zeta_\kappa}$ (rather than $\Omega_{\eta_\kappa}$). The proof of Theorem \ref{thm:linRegNon} therefore follows from the construction of a fixed point $\mathtt{T}( \zeta_\epsilon, \bv_\epsilon)=(\eta_\epsilon,\bu_\epsilon)$ of a certain solution map $\mathtt{T}$. Unfortunately, since we are dealing with weak solutions and the anticipated system \eqref{divfreekNon}-\eqref{solutekNon} is nonlinear, we are unable to use a Banach fixed-point-type argument as we did in the previous section. Consequently, we resort to a fixed-point theorem for set-valued mappings which gives the existence (but not uniqueness) of a fixed-point  \cite[Theorem A.4]{lengeler2014weak}. 
For this, with a slight abuse of notation, we (still) consider the interval $I_*:=(0,T_*)$ where $T_*$ is to be chosen later. For a tubular neighbourhood $S_\alpha$ of $\partial\Omega$ with $\alpha\leq L$, we set
\begin{align*}
X:=C(\overline{I}_* \times \omega ) \otimes L^2(I_*;L^2(\Omega\cup S_\alpha))
\end{align*}
and define the ball
\begin{align*}
B_R^X=\big\{ ( {\zeta}_\kappa, { {\bv}}_\kappa)\in X \,:\, {\zeta}_\kappa(0)= {\eta}_{0,\kappa}, \quad   \Vert ( {\zeta}_\kappa, { {\bv}}_\kappa)\Vert_{X}\leq R \big\}
\end{align*}
for $R>0$ large enough and for fixed $\kappa>0$. Now let us consider the solution map $\mathtt{T}:B_R^X \subset X  \rightarrow 2^{B_R^X}$
defined by $\mathtt{T}( {\zeta}_\kappa, { {\bv}}_\kappa)=( {\eta}_\kappa, { {\bu}}_\kappa)$. The critical requirement for a fixed point is to show compactness of the map $\mathtt{T}$. Thus, for a sequence $(\rho^n_\kappa,\bT^n_\kappa)$ satisfying  \eqref{rhoAlone}-\eqref{soluteAlone}, we consider any sequence $( {\zeta}_\kappa^n, { {\bv}}_\kappa^n)_{n\in \mathbb{N}}\in B_R^X$ with $\mathtt{T}( {\zeta}_\kappa^n, { {\bv}}_\kappa^n)=( {\eta}_\kappa^n, { {\bu}}_\kappa^n)$ (where the existence of such a solution map is guaranteed by  \eqref{galerkinweak1} and \eqref{eneryEstSolvStruAlone}). Consequently, we have in particular that 
\begin{equation}
\begin{aligned} 
\label{eneryEstSolvStruAlonesmallN}
\sup_{t\in I}\Big(
\Vert\partial_t \eta^n_\kappa\Vert_{L^2(\omega)}^2 
+  
\Vert\partial_y^2\eta^n_\kappa\Vert_{L^2(\omega)}^2
\Big) 
\lesssim 1
\end{aligned}
\end{equation}
uniformly in $n\in \mathbb{N}$. Given that the embedding $W^{2,2}(\omega)\hookrightarrow C(\omega)$ is compact and the embedding $C(\omega) \hookrightarrow L^2(\omega)$ is continuous, it follows from Aubin-Lions lemma that
\begin{align*}
\eta^n_\kappa \rightarrow \eta_\kappa
\qquad\text{in}\qquad C(\overline{I}_*\times \omega).
\end{align*}
Also, just as in  \cite[Lemma 6.3]{muha2019existence},
we can use a reformulated Aubin-Lions lemma \cite[Theorem 5.1]{muha2019existence} and the existence of a smooth solenoidal extension operator \cite[Corollary 3.4]{muha2019existence} to also obtain
\begin{align}
 \partial_t\eta^n_\kappa \rightarrow  \partial_t\eta_\kappa
&\qquad\text{in}\qquad L^2(I_*;L^2( \omega)),\label{strongCon1}
\\
  \mathbb{I}_{\Omega_{\zeta^n_\kappa}} \bu^n_\kappa \rightarrow 
   \mathbb{I}_{\Omega_{\zeta_\kappa}}  \bu_\kappa
&\qquad\text{in}\qquad L^2(I_*;L^2(\Omega\cup S_\alpha))
\label{strongCon2}
\end{align}
which finishes the proof of compactness of $\mathtt{T}$. Consequently,  the map
 $\mathtt{T}$ posses a fixed point, i.e., there exists $({\eta}_\kappa, { {\bu}}_\kappa) \in B_R^X$ with  $\mathtt{T}( {\eta}_\kappa, { {\bu}}_\kappa)=( {\eta}_\kappa, { {\bu}}_\kappa)$. The fact that the solution is global follows from  \eqref{apriori1}.  
\end{proof}
\begin{remark}
We remark that the now standard method  \cite[Lemma 6.3]{muha2019existence} for obtaining compactness for the velocity sequence has been adapted to various settings including a momentum equation with a forcing of divergence form \cite[Page 31]{breit2021incompressible} (just as we have in our present setting) and to stochastic models \cite[Page 24]{breit2023martingale}.
\end{remark}

\subsection{Limits of the regularised system}
\label{sec:limit}
We are now going to pass to the limit $\kappa\rightarrow\infty$ in the regularisation parameter to complete the proof of Theorem \ref{thm:main}. Due to Theorem \ref{thm:linRegNon} (and  \eqref{apriori1}), it follows that
\begin{align*}
\eta_\kappa \rightarrow \eta
&\qquad\text{in}\qquad  \big(L^{\infty}(I;W^{2,2}(\omega)), w^*  \big),
\\
\partial_t\eta_\kappa \rightarrow \partial_t\eta
&\qquad\text{in}\qquad
\big(L^{\infty}(I;L^{2}(\omega)),w^*  \big)  \cap \big(L^{2}(I;W^{1,2}(\omega)),w  \big),
\\
\mathbb{I}_{\Omega_{\eta_\kappa}} \bu_\kappa \rightarrow \mathbb{I}_{\Omega_{\eta}} \bu
&\qquad\text{in}\qquad
\big(L^{\infty} \big(I; L^{2}(\Omega\cup S_\alpha)\big),w^*\big)\cap \big(L^2\big(I;W^{1,2}_{\divx}(\Omega\cup S_\alpha) \big), w\big),
\\
\mathbb{I}_{\Omega_{\eta_\kappa}} \rho_\kappa \rightarrow \mathbb{I}_{\Omega_{\eta}} \rho
&\qquad\text{in}\qquad
\big(L^{\infty} \big(I; L^{2}(\Omega\cup S_\alpha)\big),w^*\big)\cap \big(L^2\big(I;W^{1,2}(\Omega\cup S_\alpha) \big), w\big),
\\
\mathbb{I}_{\Omega_{\eta_\kappa}} \bT_\kappa \rightarrow \mathbb{I}_{\Omega_{\eta}} \bT
&\qquad\text{in}\qquad
\big(L^{\infty} \big(I; L^{2}(\Omega\cup S_\alpha)\big),w^*\big)\cap \big(L^2\big(I;W^{1,2}(\Omega\cup S_\alpha) \big), w\big).
\end{align*}
Furthermore, just as in \eqref{strongCon1}-\eqref{strongCon2}, we also obtain
\begin{align}
 \partial_t\eta_\kappa \rightarrow  \partial_t\eta
&\qquad\text{in}\qquad L^2(I_*;L^2( \omega)),\label{strongCon1x}
\\
  \mathbb{I}_{\Omega_{\eta_\kappa}} \bu_\kappa \rightarrow 
   \mathbb{I}_{\Omega_{\eta}}  \bu
&\qquad\text{in}\qquad L^2(I_*;L^2(\Omega\cup S_\alpha)).
\label{strongCon2x}
\end{align}
The above convergence results allow us to pass to the limit in the weak formulation and the energy inequality.

\section{Strong solutions}
\label{sec:strongSol}
We now attend to the proof of Theorem \ref{thm:strongSol}.
Our strategy for constructing a solution also involves solving the solvent-structure subproblem and the solute subproblem independently of each other and using a fixed-point argument to get a local solution to the fully coupled system. The extension to a global solution will then follow from the  estimate \eqref{eq:thm:mainFP00}. 
\subsection{The solvent-structure subproblem}
In the following, for a given pair $(\underline{\rho}, \underline{\bT})$, given body forces $\bff$ and $g$, we wish to find a strong solution to the following system of equations
\begin{align}
\divx \bu=0, \label{divfreeAlone} 
\\
\partial_t \bu  + (\mathbf{u}\cdot \nabx)\mathbf{u} 
= 
 \delx \bu -\nabx p
+\bff
+
\divx   \underline{\bT}, 
\\
 \partial_t^2 \eta - \partial_t\partial_y^2 \eta + \partial_y^4 \eta = g - ( \mathbb{S}\bn_\eta )\circ \bm{\varphi}_\eta\cdot\bn \,\det(\partial_y\bm{\varphi}_\eta),\label{shellAlone}
\end{align}
defined on $I\times\Oeta\subset \mathbb R^{1+2}$ where
\begin{align*}
\mathbb{S}= (\nabx \bu +(\nabx \bu)^\top) -p\mathbb{I}+ \underline{\bT}.
\end{align*}
We complement \eqref{divfreeAlone}--\eqref{shellAlone} with the following initial and interface conditions
\begin{align}
&\eta(0,\cdot)=\eta_0(\cdot), \qquad\partial_t\eta(0,\cdot)=\eta_\star(\cdot) & \text{in }\omega,  
\\  
&\bu(0,\cdot)=\bu_0(\cdot) & \text{in }\Omega_{\eta_0}.
\\
\label{interfaceAlone}
&\bu\circ\bm{\varphi}_\eta=(\partial_t\eta)\bn & \text{on }I\times \omega.
\end{align}
The precise definition of a strong solution is given as follows.
\begin{definition}[Strong solution]
\label{def:strongSolutionAlone}
Let $(\bff, g, \eta_0, \eta_\star, \bu_0, \underline{\bT})$ be a dataset such that
\begin{equation}
\begin{aligned}
\label{datasetAlone}
&\bff \in L^2\big(I; L^2_{\mathrm{loc}}(\mathbb{R}^2 )\big),\quad
g \in L^2\big(I; L^{2}(\omega)\big), \quad
\eta_0 \in W^{3,2}(\omega) \text{ with } \Vert \eta_0 \Vert_{L^\infty( \omega)} < L, 
\\
&\eta_\star \in W^{1,2}(\omega), \quad
\underline{\bT}\in L^2(I;W^{1,2}_{\mathrm{loc}}(\R^2)), \\
&
\bu_0\in W^{1,2}_{\mathrm{\divx}}(\Omega_{\eta_0} ) \text{ is such that }\bu_0 \circ \bm{\varphi}_{\eta_0} =\eta_\star \bn \text{ on } \omega.
\end{aligned}
\end{equation} 
We call 
$( \eta, \bu,  p )$
a \textit{strong solution}  of  \eqref{divfreeAlone}--\eqref{interfaceAlone} with  data $(\bff, g, \eta_0, \eta_\star, \bu_0, \underline{\bT})$ provided that the following holds:
\begin{itemize}
\item[(a)]  the structure-function $\eta$ is such that $
\Vert \eta \Vert_{L^\infty(I \times \omega)} <L$ and
\begin{align*}
\eta \in &W^{1,\infty}\big(I;W^{1,2}(\omega)  \big)\cap L^{\infty}\big(I;W^{3,2}(\omega)  \big) \cap  W^{1,2}\big(I; W^{2,2}(\omega)  \big)
\\&\cap  W^{2,2}\big(I;L^{2}(\omega)  \big)
\cap
L^{2} (I;W^{4,2}(\omega )) ;
\end{align*}
\item[(b)] the velocity $\bu$ is such that $\bu  \circ \bm{\varphi}_{\eta} =(\partial_t\eta)\bn$ on $I\times \omega$ and
\begin{align*} 
\bu\in  W^{1,2} \big(I; L^2_{\divx}(\Oeta ) \big)\cap L^2\big(I;W^{2,2}(\Oeta)  \big);
\end{align*}
\item[(c)] the pressure $p$ is such that 
\begin{align*}
p\in L^2\big(I;W^{1,2}(\Oeta)  \big);
\end{align*}
\item[(d)] equations \eqref{divfreeAlone}--\eqref{shellAlone} are satisfied a.e. in spacetime with $\eta(0)=\eta_0$ and $\partial_t\eta=\eta_\star$ a.e. in $\omega$ as well as $\bfu(0)=\bfu_0$ a.e. in $\Omega_{\eta_0}$.
\end{itemize}
\end{definition}
\noindent 
The existence of a unique global-in-time strong solution to \eqref{divfreeAlone}--\eqref{interfaceAlone} in the sense of Definition \ref{def:strongSolutionAlone} is already shown in \cite[Theorem 2.5]{breit2022regularity} so we can proceed to the solute subproblem.   

\subsection{The solute subproblem}
For a known flexible domain $\Omega_\zeta$ and a known solenoidal vector field $\bv$, we aim in this section to construct a strong solution of the following solute subproblem:
\begin{align} 
\partial_t \rho+ (\bv\cdot \nabx) \rho
= 
\Delx \rho 
,\label{rhoEquAlone}
\\
\partial_t \bT + (\bv\cdot \nabx) \bT
= 
(\nabx \bv)\bT + \bT(\nabx \bv)^\top - 2(\bT - \rho \mathbb{I})+\Delx \bT \label{soluteSubPro}
\end{align}
on $I\times\Omega_\zeta\subset \mathbb R^{1+2}$ subject to the following initial and boundary conditions
\begin{align} 
&\rho(0,\cdot)=\rho_0(\cdot),\quad\bT(0,\cdot)=\bT_0(\cdot) &\text{in }\Omega_{\zeta(0)},
\label{initialCondSolvSubPro} 
\\
& 
\bn_{\zeta}\cdot\nabx\rho=0,\qquad
\bn_{\zeta}\cdot\nabx\bT=0 &\text{on }I\times\partial\Omega_{\zeta}.
\label{bddCondSolvAlone}
\end{align}
The two unknowns $\rho$ and $\bT$ for the solute component of the polymer fluid are related via the identities
\begin{align*} 
\bT(t, \bx)= \int_{B} f(t,\bx,\bq)\bq\otimes\bq \dq, 
\qquad
\rho(t, \bx)= \int_{B} f(t,\bx,\bq) \dq
\end{align*}
where $f$ solves \eqref{fokkerPlanck}.

Let us start with a precise definition of what we mean by a strong solution.  
\begin{definition}
\label{def:strsolmartFP}
Assume that  $(\rho_0,\bT_0, \bv, \zeta)$ satisfies
\begin{equation}
\begin{aligned}
\label{fokkerPlanckDataAlone}
&\rho_0,\bT_0\in  W^{1,2}( \Omega_{\zeta(0)}),
\\&
\rho_{0}\geq 0,\,\, \bT_{0}>0 \quad \text{a.e. in } \Omega_{\zeta(0)} ,
\\&
\bv\in W^{1,2} \big(I; L^2_{\divx}(\Omega_\zeta ) \big)\cap L^2\big(I;W^{2,2}(\Omega_\zeta )  \big),
\\
&\zeta\in W^{1,\infty}\big(I;W^{1,2}(\omega)  \big)\cap L^{\infty}\big(I;W^{3,2}(\omega)  \big) \cap  W^{1,2}\big(I; W^{2,2}(\omega)  \big)
\\&\qquad\cap  W^{2,2}\big(I;L^{2}(\omega)  \big)
\cap
L^{2} (I;W^{4,2}(\omega))
,
\\& \bv  \circ \bm{\varphi}_{\zeta} =(\partial_t\zeta)\bn
\quad \text{on }I \times \omega,  \quad\|\zeta\|_{L^\infty(I\times\omega)}<L.
\end{aligned}
\end{equation}
We call
$(\rho,\bT)$
a \textit{strong solution} of   \eqref{rhoEquAlone}-\eqref{bddCondSolvAlone} with dataset $(\rho_0,\bT_0, \bv, \zeta)$ if 
\begin{itemize}
\item[(a)] $(\rho,\bT)$ satisfies
\begin{align*}
\rho,\bT&\in   W^{1,2} \big(I; L^2(\Omega_\zeta ) \big)\cap L^2\big(I;W^{2,2}(\Omega_\zeta )  \big);
\end{align*}
\item[(b)] for all  $  \psi  \in C^\infty (\overline{I}\times \R^2)$, we have
\begin{equation}
\begin{aligned} 
\label{weakRhoEqAlone}
\int_I  \frac{\mathrm{d}}{\dt}
\int_{\Omega_\zeta } \rho\psi \dx \dt 
&=
\int_I
\int_{\Omega_\zeta }[\rho\partial_t\psi + (\rho\bv\cdot\nabx) \psi] \dx\dt
\\&-
\int_I\int_{\Omega_\zeta }\nabx \rho \cdot\nabx \psi  \dx\dt
\end{aligned}
\end{equation}
\item[(c)] for all  $  \mathbb{Y}  \in C^\infty (\overline{I}\times \R^2)$, we have
\begin{equation}
\begin{aligned}
\label{weakFokkerPlanckEqAlone}
\int_I  \frac{\mathrm{d}}{\dt}
\int_{\Omega_\zeta } \bT:\mathbb{Y} \dx \dt 
&=
\int_I
\int_{\Omega_{\zeta }}[\bT :\partial_t\mathbb{Y} + \bT:(\bv\cdot\nabx) \mathbb{Y}] \dx\dt
\\&+
\int_I\int_{\Omega_\zeta }
[(\nabx \bv )\bT  + \bT (\nabx \bv )^\top]:\mathbb{Y} \dx\dt
\\&
-2\int_I\int_{\Omega_\zeta }(\bT :\mathbb{Y}  - \rho  \mathrm{tr}(\mathbb{Y} ))\dx\dt
\\&-
\int_I\int_{\Omega_\zeta }\nabx \bT ::\nabx \mathbb{Y}  \dx\dt.
\end{aligned}
\end{equation}
\end{itemize}
\end{definition}
\noindent We now formulate our result on the existence of a  unique strong solution of \eqref{rhoEquAlone}-\eqref{bddCondSolvAlone}
\begin{theorem}\label{thm:mainFP}
Let  $(\rho_0,\bT_0, \bv, \zeta)$ satisfy  \eqref{fokkerPlanckDataAlone}.
Then there is a unique strong solution $(\rho,\bT)$ of  \eqref{rhoEquAlone}-\eqref{bddCondSolvAlone}, in the sense of Definition \ref{def:strsolmartFP},
 such that
\begin{equation}
\begin{aligned} 
\int_I\big(\Vert \partial_t \rho&\Vert_{L^2(\Ozeta)}^2
+
\Vert \partial_t \bT\Vert_{L^2(\Ozeta)}^2
\big)\dt
+
\sup_{t\in I} \big(\Vert \rho(t)\Vert_{W^{1,2}(\Ozeta)}^2
+
\Vert \bT(t)\Vert_{W^{1,2}(\Ozeta)}^2
\big)
\\&
+\int_I\big(\Vert  \rho\Vert_{W^{2,2}(\Ozeta)}
+
\Vert  \bT\Vert_{W^{2,2}(\Ozeta)}\big)
\\\lesssim
\Vert  \rho_0&\Vert_{W^{1,2}(\Omega_{\zeta(0)})}
+
\Vert  \bT_0\Vert_{W^{1,2}(\Omega_{\zeta(0)})}
+
\int_I\big(\Vert\partial_t\zeta  \Vert_{W^{2,2}(\omega)}^2+\Vert\bv\Vert_{W^{2,2}(\Ozeta)}^2\big)\dt
\\&
+
T
\big(
\Vert\bT_0\Vert_{L^2(\Omega_{\zeta(0)})}^2
+
T \Vert\rho_0  \Vert_{L^2(\Omega_{\zeta(0)})}^2
\big)
\exp\bigg(
c
\int_I
\Vert\nabx \bv\Vert_{L^2(\Ozeta)}^2 \dt\bigg).
\label{eq:thm:mainFP}
\end{aligned}
\end{equation} 
holds 
with a constant depending on the $L^\infty(I;W^{1,2}(\omega))$-norm of $\partial_t\zeta$ and the $L^\infty(I;W^{1,\infty}(\omega))$-norm of $\zeta$  but otherwise being independent of the data.
\end{theorem}
\noindent Since \eqref{rhoEquAlone} and \eqref{soluteSubPro} are dissipative and bilinear,  a strong solution of  \eqref{rhoEquAlone}-\eqref{bddCondSolvAlone} is directly obtained by way of a limit to a Galerkin approximation. 
In particular, the bound \eqref{eq:thm:mainFP} for the finite-dimensional solution is obtained in the same manner as
\eqref{eq:thm:mainFPxxx} after which one can pass to the limit.  

\subsection{The fully coupled system}
\label{subsec:fully}
In this section, we are going to use a fixed point argument to first establish the existence of a unique local strong solution to the fully coupled solute-solvent-structure system. This solution will hold globally in time because of Proposition \ref{prop:strongEst}. The fixed point  argument requires showing the closedness of an anticipated solution in a ball and a contraction argument. These two properties will be shown in the following spaces
\begin{align*}
X_\eta&:=W^{1,2} (I_*;L^{2}(\Oeta)  )
\cap
L^\infty (I_*;W^{1,2}(\Oeta)  )
\cap L^{2} (I_*;W^{2,2}(\Oeta)  ),
\\
Y_\eta&:=  L^{\infty} (I_*;L^2(\Oeta ) )
\cap 
L^2(I_*;W^{1,2}(\Oeta) ),
\end{align*}
equipped with their canonical norms $\Vert \cdot\Vert_{X_\eta}$ and $\Vert \cdot\Vert_{Y_\eta}$, respectively.
Here, $I_*$ with $I_*=(0,T_*)$ is to be determined later.
\\ 
Now, for $(\underline{\rho}, \underline{\bT})\in Y_\eta\otimes Y_\eta$, let $( \eta, \bu,  p )$
be a unique strong solution  of  \eqref{divfreeAlone}--\eqref{interfaceAlone} with data $(\bff, g, \eta_0, \eta_\star, \bu_0, \underline{\bT})$ as shown in  \cite[Theorem 2.5]{breit2022regularity}. On the other hand, for 
\begin{align*}
(\eta,\bu)\in&  
W^{1,\infty}\big(I_*;W^{1,2}(\omega)  \big)\cap L^{\infty}\big(I_*;W^{3,2}(\omega)  \big) \cap  W^{1,2}\big(I_*; W^{2,2}(\omega)  \big)
\\&\qquad\cap  W^{2,2}\big(I_*;L^{2}(\omega)  \big)\cap
L^{2} (I;W^{4,2}(\omega))
\\&\otimes
W^{1,2} \big(I_*; L^2_{\divx}(\Oeta ) \big)\cap L^2\big(I_*;W^{2,2}(\Oeta )  \big), 
\end{align*}
let $(\rho,\bT)$ be the unique strong solution of \eqref{rhoEquAlone}-\eqref{bddCondSolvAlone} with dataset $(\rho_0,\bT_0, \bu, \eta)$ as shown in Theorem \ref{thm:mainFP}. Now define the mapping $\mathtt{T}=\mathtt{T}_1\circ\mathtt{T}_2$ where
\begin{align*}
\mathtt{T}(\underline{\rho}, \underline{\bT})=({\rho}, {\bT}), \qquad \mathtt{T}_2(\underline{\rho}, \underline{\bT})=(\eta,\bu,p), \qquad \mathtt{T}_1(\eta,\bu,p)=( {\rho}, {\bT})
\end{align*}
and let
\begin{align*}
B_R:=\big\{  (\underline{\rho}, \underline{\bT})\in X_\eta \otimes X_\eta \,:\,\Vert (\underline{\rho}, \underline{\bT})\Vert_{X_\eta \otimes X_\eta}^2 \leq R^2  \big\}.
\end{align*}
Let show that $\mathtt{T}:X_\eta\otimes X_\eta\rightarrow X_\eta\otimes X_\eta$ maps $B_R$ into $B_R$, i.e., for any $(\underline{\rho}, \underline{\bT})\in   B_R$, we have that 
\begin{align*}
\Vert (\rho, \bT) \Vert_{X_\eta \otimes X_\eta}^2=\Vert \mathtt{T}(\underline{\rho}, \underline{\bT})\Vert_{X_\eta \otimes X_\eta}^2=\Vert \mathtt{T}_1\circ \mathtt{T}_2(\underline{\rho}, \underline{\bT})\Vert_{X_\eta\otimes X_\eta}^2=\Vert \mathtt{T}_1(\eta,\bu, p) \Vert_{X_\eta\otimes X_\eta}^2 \leq R^2.
\end{align*}
Indeed if we let $(\underline{\rho}, \underline{\bT})\in X_\eta \otimes X_\eta$ then by the a priori estimate  \eqref{eq:thm:mainFP},
\begin{equation}
\begin{aligned} 
\Vert (\rho, \bT) &\Vert_{X_\eta \otimes X_\eta}^2
\\\lesssim&
\Vert  \rho_0\Vert_{W^{1,2}(\Omega_{\eta_0})}
+
\Vert  \bT_0\Vert_{W^{1,2}(\Omega_{\eta_0})}
\\&+
\int_{I_*}\big(\Vert\partial_t\eta  \Vert_{W^{2,2}(\omega)}^2+\Vert\bu\Vert_{W^{2,2}(\Oeta)}^2\big)\dt
\\&
+
cT_*
\big(
\Vert\bT_0\Vert_{L^2(\Omega_{\eta_0})}^2
+
T_* \Vert\rho_0  \Vert_{L^2(\Omega_{\eta_0})}^2
\big)
\exp\bigg(
c
\int_{I_*}
\Vert\nabx \bu\Vert_{L^2(\Oeta)}^2 \dt\bigg).
\label{eq:thm:mainFPx}
\end{aligned}
\end{equation} 
However, by \cite[(4.14) \& Lemma 4.2]{breit2022regularity},  
a unique strong solution $( \eta, \bu,  p )$  of  \eqref{divfreeAlone}--\eqref{interfaceAlone} with data $(\bff, g, \eta_0, \eta_\star, \bu_0, \underline{\bT})$ satisfies\footnote{A careful analysis of \cite[(4.14)]{breit2022regularity} shows that $\Vert g\Vert_{L^{2}(\omega)}^2$ (rather than $\Vert g\Vert_{W^{1,2}(\omega)}^2$) is sufficient on the right of \eqref{eq:thm:mainFPxy}.}
\begin{equation}
\begin{aligned} 
\int_{I_*}\big(\Vert\partial_t\eta  \Vert_{W^{2,2}(\omega)}^2+&\Vert\bu\Vert_{W^{2,2}(\Oeta)}^2\big)\dt
\\\lesssim&
\Vert \bu_0 \Vert_{W^{1,2}(\eta_0)}^2+ \Vert\eta_0\Vert_{W^{3,2}(\omega)}^2+ \Vert \eta_\star\Vert_{W^{1,2}(\omega)}^2
\\& 
+\int_{I_*}\big(\Vert g\Vert_{L^{2}(\omega)}^2
+
\Vert \bff\Vert_{L^{2}(\Oeta)}^2
+
\Vert \underline{\bT}\Vert_{W^{1,2}(\Oeta)}^2\big)\dt.
\label{eq:thm:mainFPxy}
\end{aligned}
\end{equation} 
Given the regularity of  the dataset and the fact that $ \underline{\bT}\in X_\eta$, for a large enough $R>0$ and $T_*>0$ small enough, we obtain
\begin{align*}
\Vert (\rho, \bT) \Vert_{X_\eta \otimes X_\eta}^2\leq R^2
\end{align*}
by substituting \eqref{eq:thm:mainFPxy} into \eqref{eq:thm:mainFPx}. Thus $\mathtt{T}:B_R\rightarrow B_R$.

To show the contraction property, we let $(\rho_i,\bT_i)$, $i=1,2$ be two strong solutions of \eqref{rhoEquAlone}-\eqref{bddCondSolvAlone} with dataset $(\rho_0,\bT_0, \bu_i, \eta_i)$, $i=1,2$, respectively.
Since the fluid domain depends on the deformation of the shell, we have to transform one solution, say $\bT_2$, to the domain of $\bT_1$ in order to get a difference estimate. To get the equation for the transformation of $\bT_2$, we make use of the distributional formulation
\begin{align*}
\int_{I_*}
&
\int_{\Omega_{\eta_2}}[\partial_t\bT_2+(\bu_2\cdot\nabx)\bT_2 ]:\mathbb{Y}_2\dx
=
\int_{I_*}\int_{\Omega_{\eta_2}}
[(\nabx \bu_2)\bT_2 + \bT_2(\nabx \bu_2)^\top]:\mathbb{Y}_2\dx\dt
\\&
-2\int_{I_*}\int_{\Omega_{\eta_2}}(\bT_2:\mathbb{Y}_2 - \rho_2 \mathrm{tr}(\mathbb{Y}_2))\dx\dt
-
\int_{I_*}\int_{\Omega_{\eta_2}}\nabx \bT_2::\nabx \mathbb{Y}_2 \dx\dt
\end{align*}
for all $\mathbb{Y}\in C^\infty(\overline{{I_*}}\times \mathbb{R}^2)$.
Let us set $\overline{\bT}_2=\bT_2\circ \bm{\Psi}_{\eta_2-\eta_1}$, $\overline{\bu}_2=\bu_2\circ \bm{\Psi}_{\eta_2-\eta_1}$, $\overline{\rho}_2=\rho_2\circ \bm{\Psi}_{\eta_2-\eta_1}$ and $\mathbb{Y}_2=\overline{\mathbb{Y}}_2\circ \bm{\Psi}_{\eta_2-\eta_1}^{-1}$, then
\begin{align*}
\int_{I_*}\int_{\Omega_{\eta_2}}&
 \partial_t\bT_2:\mathbb{Y}_2 \dx\dt
\\&=
 \int_{I_*}\int_{\Omega_{\eta_2}}
 (\partial_t\overline{\bT}_2\circ \bm{\Psi}_{\eta_2-\eta_1}^{-1}+\nabx \overline{\bT}_2\circ \bm{\Psi}_{\eta_2-\eta_1}^{-1} \partial_t \bm{\Psi}_{\eta_2-\eta_1}^{-1})
 :
 \overline{\mathbb{Y}}_2\circ \bm{\Psi}_{\eta_2-\eta_1}^{-1} \dx\dt
 \\
 &=
 \int_{I_*}\int_{\Omega_{\eta_1}}
 J_{\eta_2-\eta_1}\,
 (\partial_t \overline{\bT}_2+\nabx  \overline{\bT}_2 \cdot\partial_t \bm{\Psi}_{\eta_2-\eta_1}^{-1}\circ \bm{\Psi}_{\eta_2-\eta_1})
 :
 \overline{\mathbb{Y}}_2  \dx\dt
\end{align*}
where $J_{\eta_2-\eta_1}=\det(\nabx\bm{\Psi}_{\eta_2-\eta_1})$. We also have
\begin{align*}
\int_{I_*}\int_{\Omega_{\eta_2}}&
 \bu_2\cdot\nabx\bT_2: \mathbb{Y}_2 \dx\dt
\\&=
 \int_{I_*}\int_{\Omega_{\eta_2}}
\overline{\bu}_2\circ \bm{\Psi}_{\eta_2-\eta_1}^{-1}\cdot \nabx\overline{\bT}_2\circ \bm{\Psi}_{\eta_2-\eta_1}^{-1} \nabx \bm{\Psi}_{\eta_2-\eta_1}^{-1}: \overline{\mathbb{Y}}_2\circ \bm{\Psi}_{\eta_2-\eta_1}^{-1}\dx\dt
 \\
 &=
 \int_{I_*}\int_{\Omega_{\eta_1}}
 J_{\eta_2-\eta_1}\,\overline{\bu}_2 \cdot\nabx \overline{\bT}_2 \nabx \bm{\Psi}_{\eta_2-\eta_1}^{-1}\circ \bm{\Psi}_{\eta_2-\eta_1}  : \overline{\mathbb{Y}}_2 \dx\dt
\end{align*}
as well as
\begin{align*}
\int_{I_*}\int_{\Omega_{\eta_2}}&
 [(\nabx \bu_2)\bT_2 + \bT_2(\nabx \bu_2)^\top]:\mathbb{Y}_2 \dx\dt
\\&=
 \int_{I_*}\int_{\Omega_{\eta_2}}
[\nabx\overline{\bu}_2\circ \bm{\Psi}_{\eta_2-\eta_1}^{-1}\nabx \bm{\Psi}_{\eta_2-\eta_1}^{-1} \overline{\bT}_2\circ \bm{\Psi}_{\eta_2-\eta_1}^{-1} 
\\&
\qquad\qquad
+
\overline{\bT}_2\circ \bm{\Psi}_{\eta_2-\eta_1}^{-1} 
(\nabx \bm{\Psi}_{\eta_2-\eta_1}^{-1})^\top(
\nabx\overline{\bu}_2\circ \bm{\Psi}_{\eta_2-\eta_1}^{-1})^\top]
:\overline{\mathbb{Y}}_2\circ \bm{\Psi}_{\eta_2-\eta_1}^{-1}\dx\dt
 \\
 &=
 \int_{I_*}\int_{\Omega_{\eta_1}}
 J_{\eta_2-\eta_1}\,\nabx\overline{\bu}_2\nabx \bm{\Psi}_{\eta_2-\eta_1}^{-1}\circ \bm{\Psi}_{\eta_2-\eta_1} \overline{\bT}_2
:\overline{\mathbb{Y}}_2\dx\dt
\\&
\qquad\qquad
=
 \int_{I_*}\int_{\Omega_{\eta_1}}
 J_{\eta_2-\eta_1}\, 
\overline{\bT}_2
(\nabx \bm{\Psi}_{\eta_2-\eta_1}^{-1}\circ \bm{\Psi}_{\eta_2-\eta_1})^\top(
\nabx\overline{\bu}_2 )^\top
:\overline{\mathbb{Y}}_2\dx\dt
\end{align*}
Similarly,
\begin{align*}
\int_{I_*}\int_{\Omega_{\eta_2}}(\bT_2:\mathbb{Y}_2 - \rho_2 \mathrm{tr}(\mathbb{Y}_2))\dx\dt
=
\int_{I_*}\int_{\Omega_{\eta_1}}J_{\eta_2-\eta_1}\,(\overline{\bT}_2:\overline{\mathbb{Y}}_2 - \overline{\rho}_2 \mathrm{tr}(\overline{\mathbb{Y}}_2))\dx\dt
\end{align*}
and
\begin{align*}
\int_{I_*}\int_{\Omega_{\eta_2}}&
 \nabx\bT_2::\nabx\mathbb{Y}_2 \dx\dt
\\&=
 \int_{I_*}\int_{\Omega_{\eta_2}}
\nabx \bm{\Psi}_{\eta_2-\eta_1}^{-1}\nabx \overline{\bT}_2\circ \bm{\Psi}_{\eta_2-\eta_1}^{-1} : \nabx \bm{\Psi}_{\eta_2-\eta_1}^{-1} \nabx\overline{\mathbb{Y}}_2\circ \bm{\Psi}_{\eta_2-\eta_1}^{-1}\dx\dt
 \\
 &=
 \int_{I_*}\int_{\Omega_{\eta_1}}
 J_{\eta_2-\eta_1}\,(\nabx \bm{\Psi}_{\eta_2-\eta_1}^{-1}\circ \bm{\Psi}_{\eta_2-\eta_1})^\top\nabx \bm{\Psi}_{\eta_2-\eta_1}^{-1}\circ \bm{\Psi}_{\eta_2-\eta_1}\nabx \overline{\bT}_2 ::\nabx\overline{\mathbb{Y}}_2 \dx\dt.
\end{align*}
Now let
\begin{align*}
&\mathbb{A}_{\eta_2-\eta_1}:=
(\nabx \bm{\Psi}_{\eta_2-\eta_1}^{-1}\circ \bm{\Psi}_{\eta_2-\eta_1})^\top\mathbb{B}_{\eta_2-\eta_1}
\quad\text{where}\quad
\mathbb{B}_{\eta_2-\eta_1}:=
J_{\eta_2-\eta_1} \nabx \bm{\Psi}_{\eta_2-\eta_1}^{-1}\circ \bm{\Psi}_{\eta_2-\eta_1}
\end{align*}
then we obtain the equation
\begin{align*}
\partial_t\overline{\bT}_2
+
\overline{\bu}_2 \cdot\nabx \overline{\bT}_2 
&=
\nabx\overline{\bu}_2\overline{\bT}_2
+
\overline{\bT}_2 (\nabx\overline{\bu}_2)^\top
-
2
(\overline{\bT}_2-\overline{\rho}_2\mathbb{I})
+
\Delx\overline{\bT}_2
\\&
-
\divx((\mathbb{I}-\mathbb{A}_{\eta_2-\eta_1})\nabx \overline{\bT}_2)
+
\mathbb{H}_{\eta_2-\eta_1}(\overline{\rho}_2,\overline{\bu}_2,\overline{\bT}_2)
\end{align*}
defined on ${I_*} \times \Omega_{\eta_1}$ where
\begin{align*}
\mathbb{H}_{\eta_2-\eta_1}&(\overline{\rho}_2,\overline{\bu}_2,\overline{\bT}_2)
\\&=
(1-J_{\eta_2-\eta_1})\partial_t\overline{\bT}_2
-
J_{\eta_2-\eta_1} \nabx \overline{\bT}_2\cdot\partial_t\bm{\Psi}_{\eta_2-\eta_1}^{-1}\circ \bm{\Psi}_{\eta_2-\eta_1}
+
\nabx\overline{\bu}_2(\mathbb{B}_{\eta_2-\eta_1}-\mathbb{I})\overline{\bT}_2
\\&+
\overline{\bT}_2(\mathbb{B}_{\eta_2-\eta_1}-\mathbb{I})^\top (\nabx\overline{\bu}_2)^\top
+
\overline{\bu}_2 \cdot\nabx \overline{\bT}_2 (\mathbb{I}-\mathbb{B}_{\eta_2-\eta_1})
+
2
(1-J_{\eta_2-\eta_1}) (\overline{\bT}_2-\overline{\rho}_2\mathbb{I}).
\end{align*}
Now take the first solution
\begin{align*}
\partial_t \bT_1 + \mathbf{u}_1\cdot \nabx \bT_1
=
\nabx \bu_1\bT_1 + \bT_1(\nabx \bu_1)^\top - 2(\bT_1 - \rho_1 \mathbb{I})+\Delx \bT_1
\end{align*}
defined on ${I_*} \times \Omega_{\eta_1}$ and set
\begin{align*}
\bT_{12}:=\bT_1-\overline{\bT}_2,
\quad
\bu_{12}=\bu_1-\overline{\bu}_2, 
\quad
\rho_{12}=\rho_1-\overline{\rho}_2,
\quad
\eta_{12}=\eta_1-\eta_2.
\end{align*}
Then $\bT_{12}$ solves
\begin{equation}
\begin{aligned}
\label{diffEquaSigma}
\partial_t \bT_{12}+ \mathbf{u}_1\cdot \nabx \bT_{12}
&=
\nabx \bu_1\bT_{12} + \bT_{12}(\nabx \bu_1)^\top - 2(\bT_{12} - \rho_{12} \mathbb{I})
+\Delx \bT_{12}
\\&- \mathbf{u}_{12}\cdot \nabx \overline{\bT}_2
+
\nabx \bu_{12}\overline{\bT}_2 + \overline{\bT}_2(\nabx \bu_{12})^\top 
\\&+
\divx((\mathbb{I}-\mathbb{A}_{-\eta_{12}})\nabx \overline{\bT}_2)
-
\mathbb{H}_{-\eta_{12}}(\overline{\rho}_2,\overline{\bu}_2,\overline{\bT}_2)
\end{aligned}
\end{equation}
on ${I_*} \times \Omega_{\eta_1}$ with identically zero initial condition.
If we now test \eqref{diffEquaSigma} with $\bT_{12}$, then for $t\in {I_*}$, we obtain
\begin{equation}
\begin{aligned}\label{contracUniq}
\frac{1}{2}
\frac{\dd}{\dt}&\Vert \bT_{12}\Vert_{L^2(\Omega_{\eta_1})}^2 
+ 
\Vert \bT_{12}\Vert_{W^{1,2}(\Omega_{\eta_1})}^2
\\&\leq
\Vert \bT_{12}\Vert_{L^2(\Omega_{\eta_1})}^2 
+
\Vert \rho_{12}\Vert_{L^2(\Omega_{\eta_1})}^2
+
2\int_{\Omega_{\eta_1}}\vert\nabx\bu_1\vert
\vert \bT_{12}\vert^2\dx
\\&
+
\int_{\Omega_{\eta_1}}
\big[
\nabx \bu_{12}\overline{\bT}_2 + \overline{\bT}_2(\nabx \bu_{12})^\top 
-
\mathbf{u}_{12}\cdot \nabx \overline{\bT}_2
\big]:\bT_{12} \dx
\\&
+
\int_{\Omega_{\eta_1}}
\big[
\divx((\mathbb{I}-\mathbb{A}_{-\eta_{12}})\nabx \overline{\bT}_2)
-
\mathbb{H}_{-\eta_{12}}(\overline{\rho}_2,\overline{\bu}_2,\overline{\bT}_2)\big]:\bT_{12} \dx.
\end{aligned}
\end{equation} 
Firstly, we obtain
\begin{equation}
\begin{aligned}
2\int_{\Omega_{\eta_1}}\vert\nabx\bu_1\vert
\vert \bT_{12}\vert^2\dx
&\lesssim
\Vert \nabx\bu_1\Vert_{L^2(\Omega_{\eta_1})}
\Vert \bT_{12}\Vert_{L^4(\Omega_{\eta_1})}^2
\\&
\lesssim
\Vert \nabx\bu_1\Vert_{L^2(\Omega_{\eta_1})}
\Vert \bT_{12}\Vert_{L^2(\Omega_{\eta_1})}
\Vert  \bT_{12}\Vert_{W^{1,2}(\Omega_{\eta_1})} 
\\&
\leq
c
\Vert \nabx\bu_1\Vert_{L^2(\Omega_{\eta_1})}^2
\Vert \bT_{12}\Vert_{L^2(\Omega_{\eta_1})}^2 
+
\delta
\Vert \bT_{12}\Vert_{W^{1,2}(\Omega_{\eta_1})}^2
\end{aligned}
\end{equation} 
for any $\delta>0$.
Next, by using the embedding $W^{2,2}(\Omega_{\eta_1}) \hookrightarrow L^\infty(\Omega_{\eta_1})$ to obtain,
\begin{equation}
\begin{aligned}
\int_{\Omega_{\eta_1}}\big[\nabx \bu_{12}\overline{\bT}_2 &+ \overline{\bT}_2(\nabx \bu_{12})^\top 
\big]:\bT_{12}\dx
\\&\lesssim
\Vert\nabx\bu_{12}\Vert_{L^2(\Omega_{\eta_1})}
\Vert \overline{\bT}_2\Vert_{L^{\infty}(\Omega_{\eta_1})}
\Vert \bT_{12}\Vert_{L^2(\Omega_{\eta_1})}
\\
&\leq
\delta
\Vert \nabx\bu_{12}\Vert_{L^2(\Omega_{\eta_1})}^2
+ 
c
\Vert \overline{\bT}_2\Vert_{W^{2,2}(\Omega_{\eta_1})}^2
\Vert \bT_{12}\Vert_{L^2(\Omega_{\eta_1})}^2
\end{aligned}
\end{equation}
for any $\delta>0$.
Similarly,
\begin{equation}
\begin{aligned}
\int_{\Omega_{\eta_1}}
\mathbf{u}_{12}\cdot \nabx \overline{\bT}_2
 &:\bT_{12}\dx
\lesssim
\Vert  \bu_{12}\Vert_{L^{4}(\Omega_{\eta_1})}
\Vert \nabx\overline{\bT}_2\Vert_{L^{4}(\Omega_{\eta_1})}
\Vert \bT_{12}\Vert_{L^2(\Omega_{\eta_1})} 
\\
&\leq
\delta
\Vert \nabx\bu_{12}\Vert_{L^2(\Omega_{\eta_1})}^2
+ 
c
\Vert \overline{\bT}_2\Vert_{W^{2,2}(\Omega_{\eta_1})}^2
\Vert \bT_{12}\Vert_{L^2(\Omega_{\eta_1})}^2.
\end{aligned}
\end{equation} 
Next, we write
\begin{equation}
\begin{aligned}
\int_{\Omega_{\eta_1}}
\divx((\mathbb{I}-\mathbb{A}_{-\eta_{12}})\nabx \overline{\bT}_2)
:\bT_{12} \dx
&=
\int_{\partial\Omega_{\eta_1}}
\bn_{\eta_1}\cdot(\mathbb{I}-\mathbb{A}_{-\eta_{12}})\nabx \overline{\bT}_2
:\bT_{12} \dx
\\&-
\int_{\Omega_{\eta_1}}
(\mathbb{I}-\mathbb{A}_{-\eta_{12}}) \nabx \overline{\bT}_2:
:\nabx\bT_{12} \dx
\end{aligned}
\end{equation}
where, by trace theorem and the fact that
$\mathbb{I}-\mathbb{A}_{-\eta_{12}}\sim -\naby\eta_{12}$ holds in norm
\begin{equation}
\begin{aligned}
\int_{\partial\Omega_{\eta_1}}
\bn_{\eta_1}&\cdot(\mathbb{I}-\mathbb{A}_{-\eta_{12}})\nabx \overline{\bT}_2
:\bT_{12} \dx
\\&\lesssim
\Vert
\nabx \overline{\bT}_2
\Vert_{L^4(\partial\Omega_{\eta_1})}
\Vert
\mathbb{I}-\mathbb{A}_{-\eta_{12}}
\Vert_{L^4(\partial\Omega_{\eta_1})}
\Vert
\bT_{12}
\Vert_{L^2(\partial\Omega_{\eta_1})}
\\&
\lesssim
\Vert
\nabx \overline{\bT}_2
\Vert_{W^{3/4,2}( \Omega_{\eta_1})}
\Vert
\eta_{12}
\Vert_{W^{1,4}(\omega)}
\Vert 
\bT_{12}
\Vert_{W^{3/4,2}(\Omega_{\eta_1})} 
\\&
\lesssim
\Vert
  \overline{\bT}_2
\Vert_{W^{2,2}( \Omega_{\eta_1})}
\Vert
\eta_{12}
\Vert_{W^{2,2}(\omega)}
\Vert 
\bT_{12}
\Vert_{L^2(\Omega_{\eta_1})}^{1/4}
\Vert 
\bT_{12}
\Vert_{W^{1,2}(\Omega_{\eta_1})}^{3/4}
\\&\leq
\delta
\Vert  \bT_{12}\Vert_{W^{1,2}(\Omega_{\eta_1})}^2
+
\delta
\Vert 
\bT_{12}
\Vert_{L^2(\Omega_{\eta_1})}^2
+
c
\Vert  \overline{\bT}_2
\Vert_{W^{2,2}(\Omega_{\eta_1})}^2
\Vert  \eta_{12} \Vert_{W^{2,2}(\omega )}^2.
\end{aligned}
\end{equation}
We also obtain
\begin{equation}
\begin{aligned}
\int_{\Omega_{\eta_1}}
 (\mathbb{I}-&\mathbb{A}_{-\eta_{12}})\nabx \overline{\bT}_2:
:\nabx\bT_{12} \dx
\\&
\leq
\delta
\Vert \nabx\bT_{12}\Vert_{L^2(\Omega_{\eta_1})}^2
+
c
\Vert  \overline{\bT}_2
\Vert_{W^{2,2}(\Omega_{\eta_1})}^2
\Vert  \eta_{12} \Vert_{W^{2,2}(\omega )}^2.
\end{aligned}
\end{equation}
Next, we rewrite
\begin{equation}
\begin{aligned}
\int_{\Omega_{\eta_1}}
\mathbb{H}_{-\eta_{12}}(\overline{\rho}_2,\overline{\bu}_2,\overline{\bT}_2) :\bT_{12} \dx
&=
\int_{\Omega_{\eta_1}}
(1-J_{-\eta_{12}})\partial_t\overline{\bT}_2
 :\bT_{12} \dx
\\&-
\int_{\Omega_{\eta_1}}
J_{-\eta_{12}} \nabx \overline{\bT}_2\cdot\partial_t\bm{\Psi}_{-\eta_{12}}^{-1}\circ \bm{\Psi}_{-\eta_{12}}
 :\bT_{12} \dx
\\&+
\int_{\Omega_{\eta_1}}
\nabx\overline{\bu}_2(\mathbb{B}_{-\eta_{12}}-\mathbb{I})\overline{\bT}_2
 :\bT_{12} \dx
 \\&+
\int_{\Omega_{\eta_1}}
\overline{\bT}_2(\mathbb{B}_{-\eta_{12}}-\mathbb{I})^\top (\nabx\overline{\bu}_2)^\top
:\bT_{12} \dx
 \\&+
\int_{\Omega_{\eta_1}}
\overline{\bu}_2 \cdot\nabx \overline{\bT}_2 (\mathbb{I}-\mathbb{B}_{-\eta_{12}})
 :\bT_{12} \dx
  \\&+2
\int_{\Omega_{\eta_1}}
(1-J_{-\eta_{12}}) (\overline{\bT}_2-\overline{\rho}_2\mathbb{I}) :\bT_{12} \dx
\\&=:I_1+\ldots+I_6.
\end{aligned}
\end{equation} 
Then we have, by interpolation 
\begin{align*}
I_1&\lesssim
\Vert  \eta_{12}\Vert_{W^{1,4}(\omega )}
\Vert \partial_t\overline{\bT}_2\Vert_{L^{2}(\Omega_{\eta_1})}
\Vert  \bT_{12}\Vert_{L^{4}(\Omega_{\eta_1})}
\\
&\lesssim
\Vert  \eta_{12}\Vert_{W^{2,2}(\omega )}
\Vert \partial_t\overline{\bT}_2\Vert_{L^{2}(\Omega_{\eta_1})}
\Vert  \bT_{12}\Vert_{L^{2}(\Omega_{\eta_1})}^{1/2}
\Vert   \bT_{12}\Vert_{W^{1,2}(\Omega_{\eta_1})}^{1/2}
\\
&\leq
c
\Vert  \eta_{12}\Vert_{W^{2,2}(\omega )}^2
\Vert \partial_t\overline{\bT}_2\Vert_{L^{2}(\Omega_{\eta_1})}^2
+2\delta
\Vert  \bT_{12}\Vert_{L^{2}(\Omega_{\eta_1})}
\Vert  \bT_{12}\Vert_{W^{1,2}(\Omega_{\eta_1})}
\\
&\leq
c
\Vert  \eta_{12}\Vert_{W^{2,2}(\omega )}^2
\Vert \partial_t\overline{\bT}_2\Vert_{L^{2}(\Omega_{\eta_1})}^2
+\delta
\Vert  \bT_{12}\Vert_{L^{2}(\Omega_{\eta_1})}^2
+\delta
\Vert  \bT_{12}\Vert_{W^{1,2}(\Omega_{\eta_1})}^2
\end{align*}
For $I_2$, we use the equivalence $W^{3/2,2}(\Omega_{\eta_1}) \equiv W^{1,4}(\Omega_{\eta_1})$ in 2-d and  interpolation to obtain
\begin{align*}
I_2 
&\lesssim
\Vert  \partial_t\eta_{12}\Vert_{L^{2}(\omega )}
\Vert \overline{\bT}_2\Vert_{W^{3/2,2}(\Omega_{\eta_1})}
\Vert \bT_{12}\Vert_{L^4(\Omega_{\eta_1})}
\\
&\lesssim
\Vert  \partial_t\eta_{12}\Vert_{L^{2}(\omega )}
\Vert \overline{\bT}_2\Vert_{W^{1,2}(\Omega_{\eta_1})}^{1/2}
\Vert \overline{\bT}_2\Vert_{W^{2,2}(\Omega_{\eta_1})}^{1/2}
\Vert \bT_{12}\Vert_{L^2(\Omega_{\eta_1})}^{1/2}
\Vert \bT_{12}\Vert_{W^{1,2}(\Omega_{\eta_1})}^{1/2}
\\
&\leq
\delta
\Vert \bT_{12}\Vert_{W^{1,2}(\Omega_{\eta_1})}^2
+
c
\Vert  \partial_t\eta_{12}\Vert_{L^{2}(\omega )}^{4/3}
\Vert \overline{\bT}_2\Vert_{W^{1,2}(\Omega_{\eta_1})}^{2/3}
\Vert \overline{\bT}_2\Vert_{W^{2,2}(\Omega_{\eta_1})}^{2/3}
\Vert \bT_{12}\Vert_{L^2(\Omega_{\eta_1})}^{2/3}
\\
&\leq
\delta
\Vert \bT_{12}\Vert_{W^{1,2}(\Omega_{\eta_1})}^2
+
c
\Vert  \partial_t\eta_{12}\Vert_{L^{2}(\omega )}^2
\Vert \overline{\bT}_2\Vert_{W^{1,2}(\Omega_{\eta_1})}
+
c\Vert \overline{\bT}_2\Vert_{W^{2,2}(\Omega_{\eta_1})}^2
\Vert \bT_{12}\Vert_{L^2(\Omega_{\eta_1})}^2.
\end{align*}
Finally, we also obtain
\begin{align*}
&I_3 +I_4+I_5
\lesssim
\Vert \eta_{12}\Vert_{W^{2,2}(\omega )}^2
\Vert \overline{\bT}_2\Vert_{W^{2,2}(\Omega_{\eta_1})}^2
+
\Vert \overline{\bu}_2\Vert_{W^{2,2}(\Omega_{\eta_1})}^2\Vert \bT_{12}\Vert_{L^2(\Omega_{\eta_1})}^2,
\\&I_6
\leq
c
\Vert  \eta_{12} \Vert_{W^{2,2}(\omega )}^2
\big(\Vert \overline{\bT}_2\Vert_{W^{1,2}(\Omega_{\eta_1})}^2
+
\Vert \overline{\rho}_2\Vert_{W^{1,2}(\Omega_{\eta_1})}^2\big)
+
\delta
\Vert \bT_{12}\Vert_{L^2(\Omega_{\eta_1})}^2.
\end{align*}
We have shown that
\begin{equation}
\begin{aligned}
\sup_{t\in {I_*}}&
\Vert \bT_{12}(t)\Vert_{L^2(\Omega_{\eta_1})}^2 
+
\int_{I_*}
\Vert \bT_{12}\Vert_{W^{1,2}(\Omega_{\eta_1})}^2\dt
\\\lesssim&
\exp\bigg(
\int_{I_*}\big(1+
\Vert \overline{\bu}_2\Vert_{W^{2,2}(\Omega_{\eta_1})}^2
+
\Vert \nabx\bu_1\Vert_{L^2(\Omega_{\eta_1})}^2
+
\Vert \overline{\bT}_2\Vert_{W^{2,2}(\Omega_{\eta_1})}^2\big)
\dt\bigg)
\\& 
\times\bigg[
\int_{I_*}
\Vert \rho_{12}\Vert_{L^2(\Omega_{\eta_1})}^2\dt
+
\int_{I_*}
\big(1+
\Vert \overline{\bT}_2\Vert_{W^{1,2}(\Omega_{\eta_1})}^2\big) 
\Vert  \bu_{12}\Vert_{W^{1,2}(\Omega_{\eta_1})}^2 
\dt
\\&
\qquad+
\int_{I_*}
\big(1+
\Vert \overline{\bT}_2\Vert_{W^{1,2}(\Omega_{\eta_1})}^2\big) 
\Vert  \partial_t\eta_{12}\Vert_{L^{2}(\omega )}^2 
\dt
\\&
\qquad+
\int_{I_*}
\big(\Vert \partial_t\overline{\bT}_2\Vert_{L^{2}(\Omega_{\eta_1})}^2
+
\Vert  \overline{\bT}_2
\Vert_{W^{2,2}(\Omega_{\eta_1})}^2
+
\Vert \overline{\rho}_2\Vert_{W^{1,2}(\Omega_{\eta_1})}^2
\big)
\Vert  \eta_{12} \Vert_{W^{2,2}(\omega )}^2
\dt 
\bigg]
\end{aligned}
\end{equation} 
which can be simplified to
\begin{equation}
\begin{aligned}
\sup_{t\in {I_*}}&
\Vert \bT_{12}(t)\Vert_{L^2(\Omega_{\eta_1})}^2
+
\int_{I_*}
\Vert  \bT_{12}\Vert_{W^{1,2}(\Omega_{\eta_1})}^2
\dt 
\\&\lesssim
\int_{I_*}\Vert \rho_{12}\Vert_{L^2(\Omega_{\eta_1})}^2\dt
+
\int_{I_*}
\big(
\Vert  \bu_{12}\Vert_{W^{1,2}(\Omega_{\eta_1})}^2+
\Vert  \partial_t\eta_{12}\Vert_{L^{2}(\omega )}^2\big)
\dt
\\&
+
\sup_{t\in {I_*}}
\Vert  \eta_{12} \Vert_{W^{2,2}(\omega )}^2 
\end{aligned}
\end{equation} 
by using the regularity of the individual terms.
Similarly, for the difference of two strong solutions of \eqref{rhoEquAlone}, we also obtain
\begin{equation}
\begin{aligned}\label{contracUniq1}
\sup_{t\in {I_*}}&
\Vert \rho_{12}(t)\Vert_{L^2(\Omega_{\eta_1})}^2
+
\int_{I_*}
\Vert  \rho_{12}\Vert_{W^{1,2}(\Omega_{\eta_1})}^2
\dt 
\\&\lesssim
\int_{I_*}
\big(
\Vert  \bu_{12}\Vert_{W^{1,2}(\Omega_{\eta_1})}^2+
\Vert  \partial_t\eta_{12}\Vert_{L^{2}(\omega )}^2\big)
\dt 
+
\sup_{t\in {I_*}}
\Vert  \eta_{12} \Vert_{W^{2,2}(\omega )}^2 .
\end{aligned}
\end{equation} 
Combining the two estimates above therefore yields
\begin{equation}
\begin{aligned}
\label{contrEst0}
\sup_{t\in {I_*}}&
\big(\Vert \rho_{12}(t)\Vert_{L^2(\Omega_{\eta_1})}^2
+
\Vert \bT_{12}(t)\Vert_{L^2(\Omega_{\eta_1})}^2
\big)
\\+&
\int_{I_*}
\big(\Vert \rho_{12}\Vert_{W^{1,2}(\Omega_{\eta_1})}^2
+
\Vert \bT_{12}\Vert_{W^{1,2}(\Omega_{\eta_1})}^2
\big)
\dt 
\\&\lesssim
(1+T_*)\bigg[
\int_{I_*}
\big(
\Vert  \bu_{12}\Vert_{W^{1,2}(\Omega_{\eta_1})}^2+
\Vert  \partial_t\eta_{12}\Vert_{L^{2}(\omega )}^2\big)
\dt 
+
\sup_{t\in {I_*}}
\Vert  \eta_{12} \Vert_{W^{2,2}(\omega )}^2 
\bigg]
\end{aligned}
\end{equation} 
Now,  let consider two strong solutions $( \eta_i, \bu_i,  p_i )$, $i=1,2$  of  \eqref{divfreeAlone}--\eqref{interfaceAlone} with data $(\bff, g, \eta_0, \eta_\star, \bu_0, \underline{\bT}_i)$, respectively. The existence of these solutions is shown in  \cite[Theorem 2.5]{breit2022regularity}. For 
\begin{align*}
\underline{\bT}_{12}:=\underline{\bT}_1-\underline{\overline{\bT}}_2,
\quad
\bu_{12}=\bu_1-\overline{\bu}_2,  
\quad
\eta_{12}=\eta_1-\eta_2,
\end{align*}
where $\underline{\overline{\bT}}_2:=\underline{\bT}_2\circ\bm{\Psi}_{\eta_2-\eta_1}$, it follows from \cite[Remark 5.2]{BMSS} that
\begin{align*}
\int_{I_*}
\big(
\Vert  \bu_{12}\Vert_{W^{1,2}(\Omega_{\eta_1})}^2+
\Vert  \partial_t\eta_{12}\Vert_{L^{2}(\omega )}^2\big)
\dt
+
\sup_{t\in {I_*}}
\Vert  \eta_{12} \Vert_{W^{2,2}(\omega )}^2
&\lesssim
\int_{I_*}
\Vert  \underline{\bT}_{12}\Vert_{L^{2}(\Omega_{\eta_1})}^2
\dt
\\&\lesssim
T_*
\Vert ( \underline{\rho}_{12},  \underline{\bT}_{12})\Vert_{Y_{\eta_1}\otimes Y_{\eta_1}}^2.
\end{align*}
Inserting into \eqref{contrEst0} then yields
\begin{equation}
\begin{aligned}
\label{contrEst1}
\Vert (\rho_{12}, \bT_{12})\Vert_{Y_{\eta_1}\otimes Y_{\eta_1}}^2
&\lesssim
(1+T_*)\bigg[
T_*
\Vert ( \underline{\rho}_{12},  \underline{\bT}_{12})\Vert_{Y_{\eta_1}\otimes Y_{\eta_1}}^2
\bigg].
\end{aligned}
\end{equation} 
By choosing $T_*>0$ small enough, we obtain
\begin{equation}
\begin{aligned}
\label{contrEst2}
\Vert (\rho_{12}, \bT_{12})\Vert_{Y_{\eta_1}\otimes Y_{\eta_1}}^2
&\leq
\tfrac{1}{2}
\Vert ( \underline{\rho}_{12},  \underline{\bT}_{12})\Vert_{Y_{\eta_1}\otimes Y_{\eta_1}}^2.
\end{aligned}
\end{equation} 
The existence of the desired fixed point now follows. 

The strong solution being global in time follows from Proposition \ref{prop:strongEst}. 
 
\section{Weak-Strong uniqueness}
\label{sec:weak-strong}
Due to the fixed-point argument shown above, we obtain a unique strong solution in the fixed ball. Nevertheless, one can extend uniqueness to the whole space using a continuity argument. In the following, however, we show a stronger weak-strong uniqueness result where one solution is allowed to be a generalized weak solution. Unlike the $3$-D/$2$-D framework in \cite{BMSS} where a weak-strong uniqueness result is shown for the interaction of a $3$-D fluid with a $2$-D shell under the condition that
\begin{align}
\label{ladyz1}
&\eta\in L^\infty(I; C^1(\omega)), 
\\&\label{ladyz2}
\bu\in L^r(I;L^s(\Oeta)), \quad \tfrac{2}{r}+\tfrac{3}{s}\leq 1, 
\end{align}
we can take advantage of our $2$-D/$1$-D setup to obtain an \textbf{unconditional} weak-strong uniqueness result for our fully coupled solute-solvent-structure system. 

To simplify our work, we first explain why \eqref{ladyz1}-\eqref{ladyz2} was needed in  \cite{BMSS} and then justify why they are not needed in our setting.

Firstly, the requirement \eqref{ladyz1} where needed in \cite{BMSS} for three reasons:

\begin{itemize}
\item The regularity for strong solutions consists of proving a so-called \textit{acceleration estimate} where a key tool is to estimate $\nabx^2\bu$ and  $\nabx p$ by means of $\partial_t\bu +\bu\cdot\nabx\bu$. This can be done by means of a steady Stokes theory for irregular domains \cite{breit2022regularity} that  strongly requires a boundary with a small local Lipschitz constant and thus,  \eqref{ladyz1} is essentially needed. Note that since the Lipschitz constant needs to be small, it is not enough to have spatial regularity $C^{0,1}(\omega)$.
\item To obtain weak-strong uniqueness, the authors in  \cite{BMSS} required the introduction of the so-called \textit{Universal Bogovskij operator} (see Section \ref{sec:prelim} and  \cite[Theorem 2.1]{BMSS}) which further require that the structure or shell is Lipschitz continuous in space. For weak solutions, the highest spatial regularity for $\eta$ is $W^{2,2}(\omega)$ and since the embedding $W^{2,2}(\omega)\hookrightarrow C^{0,1}(\omega)$ fails in $2$-D, one needs \eqref{ladyz1} to obtain weak-strong uniqueness. The aforementioned embedding is however true for our $1$-D shell so \eqref{ladyz1} is not needed. 
\item The final reason for needing \eqref{ladyz1} has to do with Sobolev multipliers where one need Lipschitz continuous shell displacements to get estimates of the form \cite[(2.5)]{BMSS}.
\end{itemize}
 None of these issues arise for our $1$-D shell due to the validity of $W^{2,2}(\omega)\hookrightarrow C^{0,1}(\omega)$  in $1$-D with small Lipschitz constant.

The second requirement \eqref{ladyz2}, however, is crucial in the estimate of the convective term $(\bu\cdot\nabx) \bu$ when one wants to improve weak solutions to strong ones in $3$-D. See the estimate for Roman number I \cite[(4.8)]{BMSS} in the acceleration estimate. The corresponding $2$-D acceleration estimate, as shown in \cite{breit2022regularity}, does not require any additional condition such as \eqref{ladyz2}. Also, even though \eqref{ladyz2} is used in constructing strong solution in \cite{BMSS}, it plays no explicit role in the actual weak-strong uniqueness result. Consequently, we also do not need  \eqref{ladyz2} in our current setting.

From the explanation above, it is clear that any potential conditional weak-strong uniqueness result can only depend on $\rho$ or $\mathbb{T}$. A quick inspection of the contraction argument shown in the preceding section, however, shows that we don't need any extra assumption on $\rho$ nor $\mathbb{T}$ to obtain our desired result. Indeed, if we let $(\eta_w, \bu_w, \rho_w,\bT_w)$   be a  weak solution of   \eqref{divfree}--\eqref{interface} with dataset $(\bff_w, g_w, \rho_{0,w}, \bT_{0,w}, \bu_{0,w}, \eta_{0,w}, \eta_{\star,w})$ and $(\eta_s, \bu_s, \rho_s,\bT_s)$   be a  strong solution of   \eqref{divfree}--\eqref{interface} with dataset $(\bff_s, g_s, \rho_{0,s}, \bT_{0,s}, \bu_{0,s}, \eta_{0,s}, \eta_{\star,s})$ , then in particular
\begin{align*}
&
\rho_w \in   L^{\infty}\big(I;L^{2}(\Oeta)  \big)
\cap 
L^2\big(I;W^{1,2}(\Oeta)  \big),
\\
&
\bT_w \in   L^{\infty}\big(I;L^{2}(\Oeta)  \big)
\cap 
L^2\big(I;W^{1,2}(\Oeta)  \big).
\end{align*}
This regularity (which should be compared with $\rho_1$ and $\mathbb{T}_1$ in the contraction argument earlier) is enough to make sense of all the terms in \eqref{contracUniq} and \eqref{contracUniq1} without change. More importantly, the  right-hand side of \eqref{contracUniq} (and of \eqref{contracUniq1})  does not contain the individual weak solution terms $\rho_w$ and $\mathbb{T}_w$.

With the explanation above in hand, we can now proceed to prove Theorem \ref{thm:weakstrong}. For this, we set
\begin{align*}
&\rho_{ws}:=\rho_w-\overline{\rho}_s,
\quad
\bT_{ws}:=\bT_w-\overline{\bT}_s,
\quad
\bu_{ws}=\bu_w-\overline{\bu}_s, 
\\&
\eta_{ws}=\eta_w-\eta_s,
\quad
\bff_{ws}:=\bff_w-\overline{\bff}_s,
\quad
g_{ws}=g_w-g_s, 
\end{align*} 
and define the following
\begin{align*}
\mathbb{B}_{-\eta_{ws}}&:=
J_{-\eta_{ws}} \nabx \bm{\Psi}_{-\eta_{ws}}^{-1}\circ \bm{\Psi}_{-\eta_{ws}},
\\
\mathbb{A}_{-\eta_{ws}}&:=
(\nabx \bm{\Psi}_{-\eta_{ws}}^{-1}\circ \bm{\Psi}_{-\eta_{ws}})^\top\mathbb{B}_{-\eta_{ws}},
\\
\mathbf{h}_{-\eta_{ws}}(\overline{\bu}_s)
&:=
(1-J_{-\eta_{ws}})\partial_t\overline{\bu}_s
-
J_{-\eta_{ws}} \nabx \overline{\bu}_s\cdot\partial_t\bm{\Psi}_{-\eta_{ws}}^{-1}\circ \bm{\Psi}_{-\eta_{ws}}
\\&\qquad+
\overline{\bu}_s \cdot\nabx \overline{\bu}_s (\mathbb{I}-\mathbb{B}_{-\eta_{ws}}), 
\\
h_{-\eta_{ws}}(\overline{\rho}_s,\overline{\bu}_s)
&:=
(1-J_{-\eta_{ws}})\partial_t\overline{\rho}_s
-
J_{-\eta_{ws}} \nabx \overline{\rho}_s\cdot\partial_t\bm{\Psi}_{-\eta_{ws}}^{-1}\circ \bm{\Psi}_{-\eta_{ws}}
\\&\qquad+
\overline{\bu}_s \cdot\nabx \overline{\rho}_s (\mathbb{I}-\mathbb{B}_{-\eta_{ws}}),
\\
\mathbb{H}_{-\eta_{ws}}(\overline{\rho}_s,\overline{\bu}_s,\overline{\bT}_s)
&:=
(1-J_{-\eta_{ws}})\partial_t\overline{\bT}_s
-
J_{-\eta_{ws}} \nabx \overline{\bT}_s\cdot\partial_t\bm{\Psi}_{-\eta_{ws}}^{-1}\circ \bm{\Psi}_{-\eta_{ws}}
\\&\qquad+
\nabx\overline{\bu}_s(\mathbb{B}_{-\eta_{ws}}-\mathbb{I})\overline{\bT}_s
+
\overline{\bT}_s(\mathbb{B}_{-\eta_{ws}}-\mathbb{I})^\top (\nabx\overline{\bu}_s)^\top
\\&\qquad+
\overline{\bu}_s \cdot\nabx \overline{\bT}_s (\mathbb{I}-\mathbb{B}_{-\eta_{ws}})
+
2
(1-J_{-\eta_{ws}}) (\overline{\bT}_s-\overline{\rho}_s\mathbb{I}).
\end{align*}
As in Section \ref{subsec:fully} (compare with \cite[Section 4]{mensah2024vanishing}), the difference equation  satisfies
\begin{equation}
\begin{aligned}\label{rsss}
\frac{1}{2}&\Big(
\Vert\rho_{ws}(t)\Vert_{L^2(\Omega_{\eta_w(t)})}^2
+
\Vert\bT_{ws}(t)\Vert_{L^2(\Omega_{\eta_w(t)})}^2
+
\Vert \bu_{ws}(t)\Vert_{L^2(\Omega_{\eta_w(t)})}^2 
+
 \Vert\partial_t \eta_{ws}(t)\Vert_{L^2(\omega)}^2 
 \Big)
\\&\qquad+\frac{1}{2} 
\Vert\partial_y^2\eta_{ws}(t)\Vert_{L^2(\omega)}^2
+\int_0^t\Big(
\Vert \nabx \rho_{ws} \Vert_{L^2(\Omega_{\eta_w})}^2
+
\Vert  \nabx\bT_{ws} \Vert_{L^2(\Omega_{\eta_w})}^2\Big)
\dt'
\\
& \qquad+\int_0^t\Big(
\Vert \nabx \bu_{ws}\Vert_{L^2(\Omega_{\eta_w})}^2 
+
\Vert \partial_t\naby  \eta_{ws} \Vert_{L^2(\omega)}^2
\Big)
\dt'
\\
&\leq
\frac{1}{2}
\Big(\Vert\rho_{ws}(0) \Vert_{L^2(\Omega_{\eta_w(0)})}^2
+ 
\Vert\bT_{ws}(0)\Vert_{L^2(\Omega_{\eta_w(0)})}^2
+
\Vert \bu_{ws}(0) \Vert_{L^2(\Omega_{\eta_w(0)})}^2
\Big)
\\
&\qquad +
\frac{1}{2} \Big(\Vert\partial_t \eta_{ws}(0)\Vert_{L^2(\omega)}^2
+
\Vert \naby^2 \eta_{ws}(0) \Vert_{L^2(\omega)}^2\Big)
+
\mathfrak{R}_1
+
\mathfrak{R}_2
+
\mathfrak{R}_3
+
\mathfrak{R}_4
\end{aligned}
\end{equation}
for any $t\in I$
where
\begin{align*}
\mathfrak{R}_1
&:=  
\int_0^t \int_{\Omega_{\eta_w}} (1-J_{ -\eta_{ws}})\divx\overline{\mathbb T}_s\cdot ( \bu_{ws}+\mathrm{Bog}_{\eta_w}(\Div \overline \bu_s))\dx\dt'
\\
 &\quad+\int_0^t\int_{\Omega_{\eta_w}}\divx\mathbb{T}_{ws}\cdot( \bu_{ws}+\mathrm{Bog}_{\eta_w}(\Div \overline \bu_s))\dx\dt',
\\
\mathfrak{R}_2&
:= 
\int_0^t
\int_{\partial\Omega_{\eta_w}}
(\bn_{\eta_w}\cdot\nabx)\rho_{ws} \rho_{ws}
\dd\mathcal{H}^1 \dt'
-
\int_0^t
\int_{\Omega_{\eta_w}} 
(\mathbf{u}_{ws}\cdot \nabx) \overline{\rho}_s\,
\rho_{ws} \dx\dt'
\\&
\quad+\int_0^t 
\int_{\Omega_{\eta_w}}
\big[
\divx((\mathbb{I}-\mathbb{A}_{-\eta_{ws}})\nabx \overline{\rho}_s)
-
h_{-\eta_{ws}}(\overline{\rho}_s,\overline{\bu}_s)\big]\rho_{ws} \dx\dt', 
\\
\mathfrak{R}_3&
:=
\int_0^t\big(\Vert \bT_{ws}\Vert_{L^2(\Omega_{\eta_1})}^2
+
\Vert \rho_{ws}\Vert_{L^2(\Omega_{\eta_1})}^2\big)\dt' 
\\&\quad+
\int_0^t\int_{\partial\Omega_{\eta_w}}
(\bn_{\eta_w}\cdot\nabx)\bT_{ws}:\bT_{ws}
\dd\mathcal{H}^1\dt'
+
2\int_0^t\int_{\Omega_{\eta_w}}\vert\nabx\bu_w\vert
\vert \bT_{ws}\vert^2\dx\dt'
\\&
\quad+
\int_0^t
\int_{\Omega_{\eta_w}}
\big[
\nabx \bu_{ws}\overline{\bT}_s + \overline{\bT}_s(\nabx \bu_{ws})^\top 
-
\mathbf{u}_{ws}\cdot \nabx \overline{\bT}_s
\big]:\bT_{ws} \dx\dt'
\\&
\quad+
\int_0^t
\int_{\Omega_{\eta_w}}
\big[
\divx((\mathbb{I}-\mathbb{A}_{-\eta_{ws}})\nabx \overline{\bT}_s)
-
\mathbb{H}_{-\eta_{ws}}(\overline{\rho}_s,\overline{\bu}_s,\overline{\bT}_s)\big]:\bT_{ws} \dx\dt'
\end{align*} 
and
\begin{align*}
\mathfrak{R}_4
&=   \int_0^t\int_{\Omega_{\eta_w}}
\big(\overline{\bu}_s\cdot\nabx \overline{\bu}_s  -\mathbf{h}_{-\eta_{ws}}(\overline{\bu}_s)\big)
\cdot \big(\bu_{ws} + \mathrm{Bog}_{\eta_w}(\Div\overline\bu_s)\big)\dx\dt' 
\\
&\quad+\frac{1}{2} \int_0^t\int_{\partial\Omega_{\eta_w}}\bn\circ\bm{\varphi}^{-1}_ {\eta_w}\cdot(\partial_t\eta_w \bn_{\eta_w })\circ\bm{\varphi}^{-1}_ {\eta_w}|\overline\bu_s |^2 \dd \mathcal H^1\dt'
\\
&\quad- \int_0^t\int_{\partial\Omega_{\eta_w}}\bn\circ\bm{\varphi}^{-1}_ {\eta_w}\cdot(\partial_t\eta_s \bn_{\eta_w })\circ\bm{\varphi}^{-1}_ {\eta_w}| \bu_w |^2 \dd \mathcal H^1
\dt'
\\ 
&\quad+
\int_0^t\int_{\Omega_{\eta_w}} \bu_{ws}  \cdot \partial_t\mathrm{Bog}_{\eta_w}(\Div\overline\bu_s)\dx\dt'
-
\int_{\Omega_{\eta_w}}\bu_{ws} \cdot \mathrm{Bog}_{\eta_w}(\Div\overline\bu_s)\dx\nonumber
\\
&\quad+
\int_{\Omega_{\eta_w(0)}}\bu_{ws}(0) \cdot \mathrm{Bog}_{\eta_w(0)}(\Div\overline\bu_s(0))\dx
\\
&\quad- \int_0^t\int_{\Omega_{\eta_w}}\nabx \bu_{ws}:\nabx\mathrm{Bog}_{\eta_w}(\Div\overline\bu_s)\dx\dt'
\\
&\quad+\int_0^t\int_{\Omega_{\eta_w}}\big(\mathbf A_{-\eta_{ws}}-\mathbb I \big)\nabx\overline\bu_s:\nabx(\bu_{ws}+\mathrm{Bog}_{\eta_w}(\Div\overline\bu_s))\dx\dt'
\\
&\quad+\int_0^t\int_{\Omega_{\eta_w}}\big(\mathbb I -\mathbf B_{-\eta_{ws}}\big)\overline p_s:\nabx(\bu_{ws}+\mathrm{Bog}_{\eta_w}(\Div\overline\bu_s))\dx\dt'
\\
&\quad+\int_0^t \int_\omega  g_{ws}\,\partial_t \eta_{ws}\dy\dt'
 +
 \int_0^t \int_{\Omega_{\eta_w}}   \mathbf{f}_{ws}\cdot( \bu_{ws}+\mathrm{Bog}_{\eta_w}(\Div \overline \bu_s))\dx\dt'
 \\
 &\quad+\int_0^t \int_{\Omega_{\eta_w}} (1-J_{ -\eta_{ws}})\overline{\mathbf f}_s\cdot (\bu_{ws}+\mathrm{Bog}_{\eta_w}(\Div \overline \bu_s))\dx\dt'
 \\
 &\quad+\int_0^t\int_{\Omega_{\eta_w}}\bu_w\otimes \bu_w:\nabx\mathrm{Bog}_{\eta_w}(\Div\overline \bu_s)\dx\dt'
\\
&\quad+
 \int_0^t\int_{\Omega_{\eta_w}}\bu_w\cdot\nabx\bu_w\cdot\overline \bu_s\dx\dt'
 \\
 &\quad+\int_0^t\int_{\Omega_{\eta_w}} \mathbf f_{ws}\cdot\mathrm{Bog}_{\eta_w}(\Div\overline \bu_s)\dx\dt'. 
\end{align*}
To estimate $\mathfrak{R}_1$, we use the estimate
\begin{align*}
\Vert \mathrm{Bog}_{\eta_w}(\Div \overline \bu_s)
\Vert_{W^{k,2}(\Omega_{\eta_w})}^2
=
\Vert \mathrm{Bog}_{\eta_w}(\Div \overline \bu_{ws})
\Vert_{W^{k,2}(\Omega_{\eta_w})}^2
\lesssim
\Vert \overline \bu_{ws}
\Vert_{W^{k,2}(\Omega_{\eta_w})}^2,  
\end{align*}
for $k\geq0$
to obtain
\begin{align*}
\mathfrak{R}_1
\leq&
\delta 
\int_0^t 
\Vert\nabx \bu_{ws}\Vert_{L^2(\Omega_{\eta_w})}^2
\dt'
+
c(\delta)
\int_0^t \Vert\nabx \overline{\mathbb T}_s\Vert_{L^2(\Omega_{\eta_w})}^2 
\Vert\naby^2\eta_{ws}\Vert_{L^2(\omega)}^2
\dt'
\\&
+
\delta 
\int_0^t 
\Vert\nabx \mathbb T_{ws}\Vert_{L^2(\Omega_{\eta_w})}^2
\dt'
+
c(\delta)  
\int_0^t 
\Vert  \bu_{ws}\Vert_{L^2(\Omega_{\eta_w})}^2
\dt' 
\end{align*}
for any $\delta>0$. Also, just as in \eqref{contracUniq} and   \eqref{contracUniq1},
\begin{align*}
\mathfrak{R}_2+\mathfrak{R}_3
\leq&
\delta 
\int_0^t \big(
\Vert\nabx \bu_{ws}\Vert_{L^2(\Omega_{\eta_w})}^2
+
\Vert\nabx \rho_{ws}\Vert_{L^2(\Omega_{\eta_w})}^2
+
\Vert\nabx \mathbb T_{ws}\Vert_{L^2(\Omega_{\eta_w})}^2
\big)
\dt' 
\\&
+ 
c(\delta)  
\int_0^t 
\big(1+\Vert  \nabx\bu_w\Vert_{L^2(\Omega_{\eta_w})}^2
+
\Vert  \overline{\bu}_s\Vert_{W^{2,2}(\Omega_{\eta_w})}^2
\\&\qquad\qquad+
\Vert  \overline{\rho}_s\Vert_{W^{2,2}(\Omega_{\eta_w})}^2
+
\Vert  \overline{\bT}_s\Vert_{W^{2,2}(\Omega_{\eta_w})}^2
\big)
\\&\qquad\qquad\qquad
\times
\big(
\Vert  \rho_{ws}\Vert_{L^2(\Omega_{\eta_w})}^2
+
 \Vert  \mathbb T_{ws}\Vert_{L^2(\Omega_{\eta_w})}^2
 \big)
\dt'
\\& 
+ 
c(\delta)  
\int_0^t 
\big( 
\Vert  \overline{\rho}_s\Vert_{W^{2,2}(\Omega_{\eta_w})}^2
+
\Vert  \overline{\bT}_s\Vert_{W^{2,2}(\Omega_{\eta_w})}^2
\\&\qquad\qquad+
\Vert  \partial_t\overline{\rho}_s\Vert_{L^{2}(\Omega_{\eta_w})}^2
+
\Vert  \partial_t\overline{\bT}_s\Vert_{L^{2}(\Omega_{\eta_w})}^2
\big) 
\\&\qquad\qquad\qquad
\times  
\big(
\Vert
\naby^2 \eta_{ws}\Vert_{L^2(\omega)}^2 
+
\Vert
\partial_t  \eta_{ws}\Vert_{L^2(\omega)}^2 
\big) 
\dt'
\end{align*} 
Finally, as shown in \cite[(5.6)]{BMSS} (see also \cite[Remark 5.2.]{BMSS}),
\begin{align*}
\mathfrak{R}_4
\leq&
\delta
\big(\Vert  \bu_{ws}\Vert_{L^2(\Omega_{\eta_w})}^2
+
\Vert \naby^2\eta_{ws}\Vert_{L^2(\omega)}^2
\big)
 \\&+
\delta
\int_0^t 
\big( 
\Vert\nabx \bu_{ws}\Vert_{L^2(\Omega_{\eta_w})}^2
+
\Vert\partial_t\naby\eta_{ws}\Vert_{L^2(\omega)}^2
\big)
\dt'
\\&+
c(\delta)
\int_0^t \big(1+\Vert  \overline{\bu}_s\Vert_{W^{2,2}(\Omega_{\eta_w})}^2
+\Vert  \overline{p}_s\Vert_{W^{1,2}(\Omega_{\eta_w})}^2 \big) 
\Vert\naby^2\eta_{ws}\Vert_{L^2(\omega)}^2
\dt'
\\&+
c(\delta)
\int_0^t \big(
\Vert  \partial_t\overline{\bu}_s\Vert_{L^2(\Omega_{\eta_w})}^2
+\Vert  \overline{\bff}_s\Vert_{L^2(\Omega_{\eta_w})}^2
\big) 
\Vert\naby^2\eta_{ws}\Vert_{L^2(\omega)}^2
\dt'
\\&+
c(\delta)
\int_0^t  \big(1+\Vert  \overline{\bu}_s\Vert_{W^{2,2}(\Omega_{\eta_w})}^2
+
\Vert\partial_t \eta_s\Vert_{W^{2,2}(\omega)}^2
\big) 
\Vert   \bu_{ws}\Vert_{L^2(\Omega_{\eta_w})}^2 
\dt'
\\&+
c(\delta)
\int_0^t  \big(1+\Vert  \overline{\bu}_s\Vert_{W^{2,2}(\Omega_{\eta_w})}^2
+
\Vert\partial_t \eta_s\Vert_{W^{2,2}(\omega)}^2
\big) 
\Vert\partial_t \eta_{ws}\Vert_{L^2(\omega)}^2 
\dt'
\\&
+c(\delta)
\Vert  \bu_{ws}(0)\Vert_{L^2(\Omega_{\eta_w(0)})}^2
+
c(\delta)
\int_0^t\Big(\Vert \mathbf f_{ws}\Vert_{L^2(\Omega_{\eta_w})}^2+\Vert g_{ws}\Vert_{L^2(\omega)}^2 \Big)\dt'.
\end{align*} 
Substituting the estimates for the $\mathfrak{R}_i$s back into \eqref{rsss}  and taking the supremum with respect to time, we obtain Theorem \ref{thm:weakstrong} by applying Gr\"onwall's lemma. The proof is done!

\section{Conclusion}
In conclusion, we have presented in this work, the Oldroyd-B dumbbell model describing the evolution of a two-dimensional dilute polymer fluid interacting with a one-dimensional viscoelastic shell. We have shown that if no degeneracies occur while the structure deforms, a weak and strong solution exist and the strong solution is unique. This result now opens the door to study further properties for this system including the qualitative and quantitative properties of their solutions.


%


\end{document}